\documentclass[a4paper,10pt]{amsart}
\usepackage{fullpage}

\usepackage{listings}
\usepackage[]{mdframed}
\usepackage{xfrac}
\usepackage{thmtools, thm-restate}
\usepackage{tikz}
\usepackage[utf8]{inputenc}
\usepackage[greek,english]{babel}
\usepackage{fancyhdr}
\usepackage{alphabeta}
\usepackage{amsthm}
\usepackage{amsmath}
\usepackage{amssymb}
\usepackage{faktor}
\usepackage{kbordermatrix}
\usepackage{multirow}
\renewcommand{\kbldelim}{(}
\renewcommand{\kbrdelim}{)}
\usepackage{mathtools}
\usepackage{amsfonts}
\usepackage{enumitem}
\usepackage{graphicx}
\usepackage{float}
\graphicspath{ {./images/} }

\newtheorem{theorem}{Theorem}[section]
\newtheorem{theoremanddefinition}[theorem]{Theorem - Definition}
\newtheorem{proposition}[theorem]{Proposition}
\newtheorem{lemma}[theorem]{Lemma}
\newtheorem{corollary}[theorem]{Corollary}
\theoremstyle{definition}
\newtheorem{definition}[theorem]{Definition}
\newtheorem{example}[theorem]{Example}

\newtheorem{remark}[theorem]{Remark}

\DeclareMathOperator {\Irr}{Irr}
\DeclareMathOperator {\Ann}{Ann}
\DeclareMathOperator {\Stab}{Stab}
\DeclareMathOperator {\Br}{Br}

\DeclareMathOperator {\Ind}{Ind}
\DeclareMathOperator {\Res}{Res}
\DeclareMathOperator {\Rel}{Rel}
\DeclareMathOperator {\GL}{GL}
\DeclareMathOperator {\diag}{diag}

\DeclareMathOperator {\Aug}{Aug}
\DeclareMathOperator {\Frac}{Frac}

\makeatletter
\addto\captionsenglish{}

\begin{document}

\title{The representations of the Brauer-Chen algebra}
\thanks{Research supported by the Hellenic Foundation for Research and Innovation (H.F.R.I.) under the Basic Research Financing (Horizontal support for all Sciences), National Recovery and Resilience Plan (Greece 2.0), Project Number: 15659, Project Acronym: SYMATRAL}
\author{Ilias Andreou}
\date{}
\maketitle

\begin{abstract}
In this paper, we determine the structure and representation theory of the Brauer algebra associated to a complex reflection group (here called the Brauer-Chen algebra), defined by Chen in 2011. We prove that it is semisimple and provide a construction for its simple modules for generic values of the parameters, in a uniform way for all complex reflection groups. We then apply these results to the cases of all irreducible complex reflection groups: for the groups in the infinite series, we obtain a numerical formula for the dimension of the corresponding Brauer algebra, and for all exceptional complex reflection groups, we compute the dimension of the corresponding Brauer algebra explicitly, using computational methods. We also obtain a uniformly defined basis for the Brauer algebra of any complex reflection group, defined over a field. Finally, we determine for which complex reflection groups the corresponding Brauer algebra is a free module over its ring of definition. 
\end{abstract}

\section{Introduction}

\subsection{Some history and motivation}
In 1937, Richard Brauer defined a family of algebras, now named after him, which described the centralizer of the action of the orthogonal group on a tensor power of its natural representation \cite{Br}. For the corresponding action of the general linear group, the centralizer had earlier been identified as the group algebra of the symmetric group on the number of factors of the tensor representation, acting by permuting these factors. Such centralizer relations are often referred to as Schur-Weyl dualities.

The representation theory of Brauer algebras was determined in 1988 by Wenzl \cite{W}. The techniques employed for this work were inspired by the famous earlier work of Jones connecting certain trace functions on Von Neumann algebras to invariants of links \cite{Jo, Jo2}. At the same time, Birman and Wenzl \cite{BMW} and, independently, Murakami \cite{Mur}, also inspired by Jones's work, introduced the - now also famous - BMW algebra. This is also an algebra which is closely related to link invariants, by supporting, notably, a trace that lifts to the Kauffman polynomial. The BMW algebra is a deformation of the Brauer algebra, the latter being isomorphic to special cases of the former, i.e., for certain specializations of the parameters.
In fact, it is a deformation of the Brauer algebra in much the same way as the Hecke algebra is a deformation of the group algebra of the symmetric group.  
 
In a spirit similar to how Hecke algebras have been defined and extensively studied for all complex reflection groups (yielding numerous connections to topics in representation theory, algebraic geometry, low-dimensional topology and even theoretical physics), there have been many efforts to extend the definition of Brauer and BMW algebras to complex reflection groups other than the symmetric group (which corresponds to the original case). We mention below some notable examples of these efforts.

In 2001, Häring-Oldenburg \cite{HO} introduced the cyclotomic BMW and Brauer algebras, i.e., BMW and Brauer algebras associated to complex reflection groups of type $G(m,1,n)$, providing connections of the former to link invariants in the solid torus. In 2014, Bowman and Cox \cite{BoC} associated a certain subalgebra of the cyclotomic Brauer algebra with the subgroups $G(m,p,n)$ of $G(m,1,n)$, which comprise almost all irreducible complex reflection groups.

In 2003, Cohen, Gijsbers and Wales \cite{CGW} extended BMW algebras to finite Coxeter groups of simply laced type (that is, types A, D and E, with A being the original case of the symmetric group), which led, in a new way, to earlier discovered faithful representations of the Artin group of the corresponding Coxeter types. Motivated by this, in 2007, Cohen, Frenk and Wales \cite{CFW} introduced Brauer algebras of type ADE, and completeley determined their representations for generic values of their parameters. Finally, extending this work, Cohen, Liu and Yu \cite{CLY}, and Cohen-Liu \cite{CL} defined Brauer algebras of Coxeter types B and C, respectively.

The most uniform definition of a Brauer algebra for every complex reflection group was given by Z.Chen \cite{Ch}, in 2011; we call this the Brauer-Chen algebra, as in \cite{Ma}. It was at least in part intended to serve as a stepping stone for the definition of a general BMW algebra for all complex reflection groups. In particular, it is defined in a way as to support certain formal flat connections which would dictate its deformation to a general BMW algebra. Such a relation exists between complex reflection groups and the corresponding Hecke algebras \cite{BMR}, and was shown by Marin \cite{MaQ} and, later, Chen \cite{Ch}, to exist between the Brauer and BMW algebras of types A and ADE, respectively. As shown originally by Chen, the Brauer-Chen algebra associated to Coxeter groups of type ADE is isomorphic to the corresponding Brauer algebra of Cohen, Frenk, Wales, and, in type $G(m,1,n)$, it contains the cyclotomic Brauer algebra of Häring-Oldenburg as a direct component. As stated by Cohen and Liu in the introduction of \cite{CL}, the Brauer algebras of type B and C are not isomorphic to the Brauer-Chen algebra of these types, but are connected in a way necessitating further research. 

Finally, we mention two more recent generalizations of BMW and Brauer algebras. In 2017, Chen \cite{ChBMW} defined BMW algebras associated to every Coxeter system that give rise to certain irreducible representations of the corresponding Artin group. He conjectured that one can obtain these BMW algebras as deformations of the corresponding Brauer-Chen algebras of the same type, via a flat connection supported by the latter, as the ones mentioned above. To our knowledge this question is still unanswered.

In 2018, in his PhD thesis \cite{Ne}, G. Neaime defined BMW and Brauer algebras of type $G(m,m,n)$, which are the non-real equivalent of simply laced type, and obtained irreducible representations of the corresponding braid groups.  Again, apart from the case $n=3$ and $m$ odd, for which he showed that his Brauer algebra is isomorphic to the Brauer-Chen algebra, the connections between the two algebras are unexplored.

\subsection{Overview of our results}

Motivated by the uniform way of its definition, its relation to other defined Brauer algebras and its possible applications to the generalization of BMW algebras this article is devoted to the study of the structure of the generic Brauer-Chen algebra for any complex reflection group. 

In particular, we show that, for generic values of the parameters and suitable fields of definition (e.g. algebraically closed fields of characteristic zero, see Definition \ref{def:propering}), the Brauer-Chen algebra is semisimple and we provide a complete determination of its simple modules, as well as a formula for its dimension (Theorems \ref{th:simplemodules} $\&$ \ref{th:semisimplicity}, respectively). Such results were obtained for the Brauer algebra of simply laced type (ADE) by Cohen, Frenk and Wales in their original paper. They constructed the simple modules of the corresponding Brauer algebra from pairs consisting of a suitable collection of pairwise orthogonal roots (called \textit{admissible}) and a simple module of a certain subgroup of the corresponding Coxeter group, depending on the chosen collection of orthogonal roots. Generalizing this construction was not straightforward since many of its elements (e.g. the admissibility condition) depended on the language of roots, which does not behave, for general complex reflection groups, as well as for Coxeter groups.
On the other hand, in \cite{Ma} Marin constructed a certain class of simple modules for the Brauer-Chen algebra of every complex reflection group. In the cases of Coxeter groups of simply laced type, these modules correspond to the modules constructed by Cohen, Frenk and Wales arising from collections consisting of a single root (these are always admissible).

The key element that enabled the generalization of the above results is the observation that a simple module of the Brauer-Chen algebra is determined by a certain module of the stabilizer of a certain collection (many equivalent ones, in fact) of reflecting hyperplanes. In particular, as a module of the corresponding complex reflection group, it is induced by this module of the stabilizer. Given such a structure, one can reverse the procedure and study the conditions under which a collection of reflecting hyperplanes and a module of its stabilizer induce indeed a simple module of the Brauer-Chen algebra. This generalizes the notion of admissibility of Cohen, Frenk and Wales. After constructing all simple modules, it is by computing their dimensions and finding a suitably small spanning set that we obtain semisimplicity and the dimension of the algebra. What is noteworthy is that up to this point the treatment is uniform for all complex reflection groups, i.e. without having to resort to a case-by-case analysis, a standard technique when dealing with complex reflection groups.

Applying these results to the case of irreducible complex reflection groups of the infinite series $G(m,p,n)$ suggests a more explicit structure of simple modules, closer to the constructions of Cohen, Frenk, Wales and Marin. In particular, for admissible collections of reflecting hyperplanes, the suitable modules that induce simple modules of the Brauer-Chen algebra are those that induce the trivial representation on a certain subgroup of the stabilizer of the given collection. This condition is then verified to hold for all exceptional complex reflection groups, using programming in GAP4 \cite{GAP4}, after a small modification (namely, a certain non-trivial representation replaces the trivial one on a certain subgroup of the stabilizer). This result provides a much more efficient formula for computing the dimension of the Brauer-Chen algebra, as well as a nice basis over the fields in question (Theorem \ref{th:basisoverK}).

The obtained basis specializes to the original basis of the Brauer algebra, as well as to the basis obtained by Cohen, Frenk and Wales for the Brauer algebra of simply laced types. In these situations it is, in fact, a basis over the ring of definition of the algebra. Generalizing this, we show that our obtained basis over a field is a basis over the ring of definition of the Brauer-Chen algebra for all irreducible complex reflection groups other than the exceptional groups $G_{25},G_{32}$ (Propositions \ref{prop:freenessGmpn}, \ref{cor:freenessmostexceptionals}, \ref{prop:freeness26}). For these last two groups, the Brauer-Chen algebra is not a free module over its ring of definition, which we show by noticing that for certain values of its parameters, the dimension of the algebra over the corresponding field changes (Proposition \ref{prop:freeness2532}). Again, for the study of freeness of the Brauer-Chen algebra in the exceptional groups, we used programming in GAP4 to obtain the corresponding results.

\section{Preliminaries}\label{ch:prelimcrg}

\subsection{Basics of Complex Reflection Groups}
\label{sec:basicscrg}

Although our definitions here may be slightly modified, we refer the reader to the book \cite{URG} of Lehrer and Taylor for more detailed information on reflections and complex reflection groups. 

\begin{definition}[Reflection]\label{def:ref} Let $V$ be a finite dimensional complex vector space. An element $r\in \GL(V)$ is called a \textit{pseudo-reflection} if the subspace $\ker(r-1)$ of fixed vectors of $r$ is of codimension $1$ in $V$. The subspace $\ker(r-1)$ is called the \textit{reflecting hyperplane}  of $r$ and we denote it by $H_r$.
\end{definition}

Here, we will just use the term reflection for any pseudo-reflection; over $\mathbb{R}$ the notions coincide.  If $r$ is a reflection of $V$, and $l$ is an invertible linear transformation of $V$, then $lrl^{-1}$ is again a reflection, and $H_{lrl^{-1}}=l(H_r)$.

\begin{definition}[Complex reflection group]\label{def:crg} A \textit{complex reflection group} $W$ is a finite subgroup of $\GL (V_W)$ for some complex vector space $V_W$ which is generated by reflections. The representation $V_W$ is called the \textit{natural representation} of $W$. If it is irreducible, then $W$ is said to be \textit{irreducible}. Finally, the \textit{rank} of $W$ is the dimension of $V_W$ minus the dimension of the subspace of fixed vectors of $W$. 
\end{definition}

For a complex reflection group $W\subseteq \GL(V_W)$, one can always assume that there is a positive definite hermitian form on $V_W$ which is stable under $W$ (take any positive definite hermitian form on $V_W$ and normalize it under the finite group $W$). A reflection $r$ of $V_W$ respecting that form is a \textit{unitary} or \textit{orthogonal} reflection, and it acts on the orthogonal subspace of its reflecting hyperplane $H_r^\perp$ as some root of unity; we will denote by $\mathbb{U}_m$ the group of $m$-th roots of unity. As a consequence, the reflections of $W$ sharing the same reflecting hyperplane $H$ (together with the identity) form a cyclic group isomorphic to some $\mathbb{U}_m$ according to their action on $H^\perp$. One can choose a distinguished generator of this cyclic group, which is the reflection that acts on $H^\perp$ as multiplication by the root of unity $e^{\frac{2πi}{m}}$, and is hence completely determined by $H$ and its order.  This is called the \textit{distinguished reflection} of $H$.

From here on, $W\subseteq \GL(V_W)$ will be a complex reflection group, and $V_W$ will be assumed to be equiped with a $W$-invariant positive definite hermitian form. We will denote by $\mathcal{R}$ the set of reflections of $W$ and by $\mathcal{H}$ the set of reflecting hyperplanes of $W$. To every reflection in $\mathcal{R}$ corresponds a hyperplane in $\mathcal{H}$. This correspondence is of course surjective but not injective in general, since all powers of a reflection share the same reflecting hyperplane. 

The group $W$ acts naturally on $\mathcal{H}$, mapping each hyperplane $H\subseteq V_W$ to $w(H)$, which will be written as $wH$. This action corresponds to the action of $W$ on $\mathcal{R}$ given by conjugation. Whenever the action on a set $X$ is clear we will denote conjugacy of $a,b \in X$  by $a\sim b$. Also, $\mathcal{R}_{a\to b}$ will denote the set of reflections mapping $a$ to $b$.

\subsection{Transverse hyperplanes}

A key notion for the definition of the Brauer-Chen algebra is the notion of transverse hyperplanes. We define them here and gather certain lemmas we will later be using about them.

\begin{definition}[Transverse hyperplanes]
\label{def:transverse}
Let $H_1,H_2$ be distinct reflecting hyperplanes of $W$. Then $H_1,H_2$ are called \textit{transverse} if they are the only reflecting hyperplanes of $W$ containing their intersection. If this is the case, we write $H_1\pitchfork H_2$.  A collection of pairwise transverse hyperplanes of $W$ will be called a \textit{transverse collection}.
\end{definition}

When we say that two reflecting hyperplanes are \textit{non-transverse} (as opposed to \textit{not} transverse) it will imply that they are distinct and not transverse; we will denote non-transversality of $H_1,H_2$ by $H_1 \not \pitchfork H_2$. If $B$ is a collection of reflecting hyperplanes of $W$, we say that $H$ is transverse with $B$ if it is transverse with every element of $B$, and we write $H\pitchfork B$. Correspondingly, we write $H\not \pitchfork B$ if $H$ is non-transverse with some element of $B$. Finally, if $B'$ is another collection of reflecting hyperplanes, we say that $B$ is transverse with $B'$, and write $B\pitchfork B'$, if every hyperplane of $B$ is transverse with every hyperplane of $B'$.

Throughout this work, an important role is played by transverse collections. The action of $W$ on $\mathcal{H}$ induces an action on the set of all such collections which will be denoted $\mathcal{C}$.  When we talk about the orbit or a conjugate of such a collection it will be with respect to this action, and if $B$ is a collection of transverse hyperplanes, then $\Stab(B)$ will denote the stabilizer of $B$ under this action. For reasons of uniformity of exposition, we will suppose that $\mathcal{C}$ contains the empty collection, as well. Furthermore, the empty collection will be assumed to be transverse with every $H\in \mathcal{H}$, and its stabilizer will be $W$. We denote the set of non-empty transverse collections by $\mathcal{C}^*$.

If $r$ is a unitary reflection of $V_W$, then a \textit{root} of $r$ (or its reflecting hyperplane) is any non-zero vector of the orthogonal complement of its reflecting hyperplane. The following lemma provides a convenient way to check transversality of two hyperplanes using roots. 

\begin{lemma}\label{lem:transroots}  Let $H_1,H_2$ be reflecting hyperplanes of $W$ with corresponding roots $a_1,a_2\in V_W$. Then, $H_1,H_2$ are transverse if and only if no reflecting hyperplane of $W$ other than $H_1,H_2$ has a root in the linear span of $a_1,a_2$.
\end{lemma}
\begin{proof} If $H$ is a reflecting hyperplane of $W$ with root $a$ then $H_1\cap H_2 \subseteq H$ is equivalent to $H^{\perp}\subseteq H_1^{\perp}\oplus H_2^{\perp}$ which, is in turn, equivalent to $a\in \langle a_1,a_2 \rangle $. The result follows.
\end{proof}

\begin{lemma}\label{lem:commrefls}
Let $r_1,r_2$ be unitary reflections of $V_W$ with respective reflecting hyperplanes $H_1, H_2$. Then the following are equivalent:
\begin{enumerate}
\item $r_1,r_2$ commute,
\item $r_1H_2=H_2$,
\item $H_1=H_2$ or $H_1^{\perp}\perp H_2^{\perp}$.
\end{enumerate}
\end{lemma}
\begin{proof}
For the implication (1) $\Rightarrow$ (2), if $r_1,r_2$ commute, then $r_1r_2r_1^{-1}=r_2$. Since, $r_1r_2r_1^{-1}$ is a reflection with reflecting hyperplane $r_1H_2$, this means that $r_1H_2=H_2$.

For (2) $\Rightarrow$ (3), suppose that $r_1H_2=H_2$; then $r_1H_2^\perp=H_2^\perp$. Using the decomposition $H_1\oplus H_1^{\perp}$ of $V_W$ and the action of $r_1$ on $H_1, H_1^\perp$, one can see that a line of $V_W$ is stable under $r_1$ if and only if it is contained in $H_1$ or $H_1^\perp$. The latter implies that $H_2^\perp\subseteq H_1$, which is equivalent to $H_1^\perp \perp H_2^\perp$, and the former that $H_1^\perp=H_2^\perp$, which is equivalent to $H_1=H_2$. 

Finally, for (3) $\Rightarrow$ (1), both of the conditions $H_1=H_2$ or $H_1^\perp \perp H_2^\perp$ imply that $r_1,r_2$ commute. To see that for the latter, in particular, one may use the decomposition $H_1\cap H_2 \oplus H_1 ^\perp \oplus H_2^\perp$ of $V_W$ and the actions of $r_1,r_2$ on $H_1\cap H_2, H_1^\perp, H_2^\perp$.
\end{proof}

\begin{lemma}\label{lem:transverseorthogonal} Let $H_1,H_2$ be transverse reflecting hyperplanes of $W$ with respective reflections $r_1,r_2$. Then $r_1,r_2$ commute.
\end{lemma}
\begin{proof} If $r_1,r_2$ do not commute, then, by the previous lemma, $r_1H_2$ is different from $H_2$, as well as $H_1$, since $r_1H_2=H_1$ would imply that $H_2=r_1^{-1}H_1=H_1$ which contradicts transversality of $H_1,H_2$. But, since $r_1$ is the identity on $H_1\cap H_2$ which is contained in $H_2$, then $r_1H_2$ contains $H_1 \cap H_2$ as well, which is impossible for $H_1,H_2$ transverse.  
\end{proof}

The next two lemmas will be used repeatedly. The first one is  \cite[Lemma 5.4]{Ch} and \cite[Lemma 3.1]{MaKram} and is stated without proof.

\begin{lemma}
\label{lem:transverse1}
Let $H_1,H_2$ be two distinct hyperplanes in some complex vector space, and $s$ a reflection with reflecting hyperplane $H$. If $sH_1=H_2$, then $H_1\cap H_2 \subseteq H$. Consequently, if $H_1,H_2$ are transverse reflecting hyperplanes of a complex reflection group $W$, then there are no reflections in $W$ mapping $H_1$ to $H_2$. 
\end{lemma}

\begin{lemma}
\label{lem:transverse2}
 Let $V$ be a complex vector space, on which we fix a positive definite hermitian form. If $H_1^{\perp}, H_2^{\perp}\perp H^{\perp}$, then a unitary reflection $s$ that maps $H_1$ to $H_2$ leaves $H$ invariant. This is especially the case when $H_1,H_2, H$ are reflecting hyperplanes of $W$ and $H_1,H_2$ are transverse with $H$.
\end{lemma}
\begin{proof} Let $H'$ be the reflecting hyperplane of $s$. Since $sH_1=H_2$, then, by Lemma \ref{lem:transverse1} above, we have that $H_1\cap H_2 \subseteq H'$ which is equivalent to $(H')^{\perp} \subseteq H_1^{\perp}+H_2^{\perp}.$ But, since $H_1^{\perp},H_2^{\perp}$ are perpendicular to $H^{\perp}$, then so is $(H')^{\perp}$. By Lemma \ref{lem:commrefls}, we have that $sH=H$.
\end{proof}

The following lemma will be used only in the very end of this work but we include it here as an application of the two lemmas above.

\begin{lemma}\label{lem:transversepairs} Let $B=\{H_1,H_2\}$ and $B'=\{H_1',H_2'\}$ be two disjoint collections of transverse hyperplanes of $W$. Then, there is at most one reflection mapping $H_1$ to $H_1'$ and $H_2$ to $H_2'$.
\end{lemma}
\begin{proof} Let $s$ be a reflection mapping $B$ to $B'$ and suppose that $sH_1=H_1', sH_2=H_2'$; we will show that $H_s$ is determined by $H_1,H_2,H_1',H_2'$. Hence, if $s'H_1=sH_1$ and $s'H_2=sH_2$ for some reflection $s'$ then $H_{s'}=H_s$. This implies that both $s$ and $s'$ are powers of some reflection with hyperplane $H_s$ and hence $s^{-1}s'$ is either such a reflection as well or the identity. Since $s'H_2=sH_2$, then $s^{-1}s'H_2=H_2$ and so, if $s^{-1}s'\neq 1$, then we have $H_s^\perp \perp H_2^\perp$ or $H_s=H_2$ (Lemma \ref{lem:commrefls}) implying, in any case, that $sH_2=H_2$ which is a contradiction; so $s'=s$.

To show now that $H_s$ is uniquely determined, since $sH_1=H_1' \neq H_1$, then, by Lemma \ref{lem:transverse1}, we have $H_1\cap H_1' \subseteq H_s$ and, similarly, $H_2\cap H_2' \subseteq H_s$. We will show that $H_1\cap H_1' \neq H_2\cap H_2' $ and hence, since these are subspaces of codimension $1$ in $H_s$, we get that $H_s$ is generated by them, which yields the result. So suppose that $H_1\cap H_1' = H_2\cap H_2' =A$. Then, $A$ is contained in $H_1 \cap H_2$ and it is of codimension $2$ in  $V$, so in fact $A=H_1\cap H_2$. This yields a contradiction since $A$ is also contained, for example, in $H_1' \neq H_1, H_2$, and $H_1,H_2$ are transverse.
\end{proof}

\subsection{Classification of complex reflection groups}
\label{sec:crgclass}
The following classical result is Theorem 1.27 of the book \cite{URG} of Lehrer and Taylor.

\begin{theorem} Every complex reflection group is the product of irreducible complex reflection groups.
\end{theorem}

This result reduces the study of complex reflection groups to the study of irreducible ones.
The latter were classified in 1954 by Shephard and Todd \cite{Clas}; they fall into two categories. We give some very introductory information, which is necessary for our analysis later. As usual, the reader is also referred to \cite{URG}.

\subsubsection{The infinite series}\label{subsec:infiniteseries}

The infinite series is a $3$-parameter family of irreducible complex reflection groups denoted by $G(m,p,n)$, where  $m,p,n$ are positive integers, and $p|m$. The group $G(m,p,n)$ is the subgroup of $\GL(\mathbb{C}^n)$ consisting of monomial matrices, i.e. matrices with a unique non-zero entry in every row and every column, whose non-zero entries lie in $\mathbb{U}_m$, and whose product of entries
lies in $\mathbb{U}_{\frac{m}{p}}$. The case where $m=p=1$ corresponds to the symmetric group $S_n$ acting on $\mathbb{C}^n$ by permuting the standard basis. It is the only case when its rank is not $n$ but $n-1$ (the sum of the vectors of the standard basis is fixed under this action).
\vspace{10pt}

\paragraph{\bf Reflections and hyperplanes of $G(m,p,n)$.}

Let $z_1,\dots z_n$ denote the standard coordinates of $\mathbb{C}^n$ and $ζ=e^{\frac{2πi}{m}}$.
The group $G(m,p,n)$ contains at most two classes of distinguished reflections:

\begin{enumerate}
\item for $1\leq i\neq j\leq n$, and $κ\in \mathbb{Z}$,  the reflections $(ij)_κ$ with 
$$(ij)_κ (z_1,\dots, z_i, \dots, z_j, \dots ,z_n)= (z_1,\dots,ζ^κz_j ,\dots ,ζ^{-κ}z_i,\dots, z_n)$$
The reflection $(ij)_κ$ has order $2$ and its reflecting hyperplane is  $$H^κ_{ij}=\{(z_1,\dots ,z_n)\in \mathbb{C}^n\big| z_i=ζ^κz_j\}.$$ When there is no ambiguity as to the system of coordinates, we will say that $H^κ_{ij}$ has equation $z_i=ζ^κz_j$.
Note that, $(ij)_κ=(ji)_{-κ}$, and, if $κ\equiv κ'  \mod m$, then $(ij)_κ=(ij)_{κ'}$. If $κ=0$ we will just omit it, and write $H_{ij}$ and $(ij)$ for the corresponding hyperplane and reflection, respectively.

\item 
if $p \neq m$, then for $1\leq i \leq n$, the reflections $t_i$ with 
$$t_i(z_1,\dots ,z_n)=(z_1,\dots , ζ^pz_i, \dots ,z_n).$$
Each $t_i$ has order $m/p$, and its corresponding reflecting hyperplane, which we denote by $H_i$, has equation $z_i=0$.

\end{enumerate}

\paragraph{\it Transversality relations for $G(m,p,n)$}

The following lemma, the verification of which is left to the reader, sums up all the necessary relations of transversality for the groups of the infinite series.

\begin{lemma}
\label{lem:Gmpn0}
In $G(m,p,n)$,
\begin{enumerate}
\item if $i\neq j$, then $H_i$ and $H_j$ are non-transverse and the reflections that map one to the other (in any order) are the reflections $(ij)_κ, κ\in \mathbb{Z}$;

\item $H_{i'}$ and $H_{ij}^κ$ are transverse if and only if $i,j  \neq i'$; even in the case where $i'\in \{i,j\}$, there exist no reflections mapping one to the other;

\item if $j_1\neq j_2$, then $H_{ij_1}^{κ_1}$ and $H_{ij_2}^{κ_2}$ are non-transverse and the only reflection mapping one to the other (in any order) is $(j_1j_2)_{κ_2-κ_1}$;

\item if $\{i_1,j_1\}\cap \{i_2,j_2\}=\emptyset$, then $H_{i_1j_1}^{κ_1}$ and $H_{i_2j_2}^{κ_2}$ are transverse;

\item for $(m,p)\neq (2,2)$, if  $i\neq j$, then $H_{ij}^{κ_1}$ and $H_{ij}^{κ_2}$ are not transverse and the reflections mapping the former to the latter are  $(ij)_κ$, for $2κ\equiv κ_1+κ_2 \mod m$, and $t_i^{κ_1-κ_2}, t_j^{κ_2-κ_1}$, provided that $κ_2-κ_1 \equiv 0 \mod p$. 

\item if $(m,p)=(2,2)$, then $H_{ij}$ and $H_{ij}^{1}$ are transverse.  

\end{enumerate}
\end{lemma}

\subsubsection{The exceptional groups} 

The exceptional complex reflection groups form a family of 34 exceptional groups denoted $G_4, \dots ,G_{37}$ of increasing ranks $2$ to $8$. They contain the Coxeter groups $H_3\cong G_{23}, F_4 \cong G_{28}, H_4 \cong G_{30}, E_6 \cong G_{35}, E_7 \cong G_{36}, E_8 \cong G_{37}$. Here, we present two examples of non-real exceptional complex reflection groups, to which we will return at the the very end of this article. Again, for a comprehensive exposition of all exceptional groups, the reader is referred to \cite{URG}, where these groups correspond to the \textit{primitive} unitary reflection groups.

\begin{example}[Groups $G_{26},G_{25}$]
\label{ex:g26}
The following description of $G_{26}, G_{25}$ can be found, in the form of line systems, in \cite[p.149]{URG}, where they correspond to the groups generated by the line systems $\mathcal{M}_3$ and $\mathcal{L}_3$, respectively (see also the table of page 275 of the same book).

Let $z_1,z_2,z_3$ denote the standard coordinates of $\mathbb{C}^3$ and $ζ=e^{\frac{2πi}{3}}$. Assume that $\mathbb{C}^3$ is equiped with the standard inner product. The group $G_{26}$ is the subgroup of $\GL(\mathbb{C}^3)$, generated by the following $3$ types of distinguished unitary reflections:

\begin{enumerate}
\item $t_i, i=1,2,3$, with reflecting hyperplanes $H_i$ with equation $z_i=0$, and order $3$,
\item $t_{κ,λ }$, $κ,λ = 0,1,2$,  with reflecting hyperplanes $T_{κ,λ}$ with equation $z_1 + ζ^κ z_2 + {ζ}^λ z_3=0$, and order $3$, and
\item $(ij)_{κ}$, for $1\leq i\neq j \leq 3$ and $κ=0,1,2$, with reflecting hyperplanes $H_{i,j}^{κ}: z_i=ζ^κz_j$, and order $2$. If $κ=0$, we may omit it from the notation.

\end{enumerate}

The group $G_{25}$ is the subgroup of $G_{26}$ generated by the reflections of the first two types. One may notice that reflections of type (1) and (3) generate the group $G(3,1,3)$ of the infinite series.

\end{example}

\section{The Brauer-Chen algebra}\label{ch:defalgebra}

\subsection{The definition of the Brauer-Chen algebra}

Recall that $W$ is a complex reflection group and $\mathcal{R}$, $\mathcal{H}$ denote its sets of reflections and reflecting hyperplanes, respectively.

\begin{definition}[Ring of definition]\label{def:ringofdef} Let $δ$ and $μ_s, s\in \mathcal{R}$, be indeterminates, and set $\underline{μ}:=\{μ_s\big| s\in \mathcal{R}\}$. The ring of definition of the Brauer-Chen algebra of $W$ is the ring $\mathcal{A}=\mathbb{Z}[\underline{μ}^{\pm 1},δ^{\pm 1}]\big/(μ_s-μ_{s'}\big|  s\sim s')$, where $\sim$ denotes conjugation with respect to $W$.
\end{definition}

\begin{definition}[Definition of the Brauer-Chen algebra]
\label{def:algebra}
The Brauer-Chen algebra associated to $W$ is the $\mathcal{A}$-algebra $\Br(W)$
generated by the elements of $W$ subject to the group relations, together with elements $e_H, H\in \mathcal{H}$, subject to the following relations, where $H,H_1,H_2\in \mathcal{H}$ and $w\in W$:

\begin{itemize}
\item[(B1)] $e_H^2=δe_H$, 
\item[(B2)] $we_Hw^{-1}=e_{wH}$, 
\item[(B3)] $we_H=e_H$, if $H\subseteq \ker(w-1)$,
\item[(B4)] $e_{H_1}e_{H_2}=e_{H_2}e_{H_1}$, if $H_1\pitchfork H_2$
\item[(B5)] $e_{H_1}e_{H_2}=\sum_{s\in \mathcal{R}_{H_2\to H_1}} μ_sse_{H_2}$, if $H_1\not\pitchfork H_2$
\end{itemize}
If $R$ is an $\mathcal{A}$-algebra, we denote the $R$-algebra $R\otimes_{\mathcal{A}} \Br(W)$ by $\Br^R(W)$. When there is no ambiguity as to $R$, we will denote the images of $δ$ and $μ_s,s\in \mathcal{R}$ in $R$ with the same letters.
\end{definition}

\begin{remark}\label{rem:tensoringrelations}
With the above convention in mind, $\Br^R(W)$ is isomorphic to the $R$-algebra defined by the same generators and relations as in the definition.
\end{remark}

\begin{remark} Notice that the condition $H\subseteq \ker(w-1)$ of relation (B3) is satisfied either when $w=1$ or when $w$ is a reflection with reflecting hyperplane $H$. 
\end{remark}

\begin{remark}\label{rem:symrel} Applying (B2) to the right hand side of relation (B5) of the above definition we get (B5'): $e_{H_1}e_{H_2}=e_{H_1}\sum_{s\in \mathcal{R}_{H_2\to H_1}} μ_ss$, if $H_1\not \pitchfork H_2$.
\end{remark}

\begin{remark} It is clear from the definition, that $\Br(W)$ is a quotient of the $\mathcal{A}$-algebra $\mathcal{A}\langle \mathcal{H} \rangle \ltimes W$, where $\mathcal{A}\langle \mathcal{H} \rangle$ denotes the free associative $\mathcal{A}$-algebra on $\mathcal{H}$.
\end{remark}

\begin{definition}[Proper rings of definition] \label{def:propering} Let $R$ be an $\mathcal{A}$-algebra. We say that
$R$ is \textit{proper} (for $W$) if it has all of the following properties.
\begin{itemize}
\item[(P1)] $R$ is a domain of characteristic zero,
\item[(P2)] $\Frac(R)$ has a subfield $k_R$ containing (the image in $R$ of) $\underline{μ}$, over which every subgroup of $W$ splits,
\item[(P3)] $δ$ is transcendental over $k_R$, and $\Frac(R)$ is a finite extension of $k_R(δ)$. 
\end{itemize}
\end{definition}

\begin{example}\label{ex:propering} For any field $L$ of characteristic zero over which every subgroup of $W$ splits (e.g., $\overline{\mathbb{Q}}$ or $\mathbb{C}$), if $R_0$ is a quotient of $L[\underline{μ}^{\pm 1}]\big/(μ_s-μ_{s'}\big|  s\sim s')$ which is a domain, then the ring $R=R_0[δ^{\pm 1}]$ is a proper ring of definition, where, for example, one can choose $k_R$ to be the subfield $\Frac(R_0)$ of $\Frac(R)$. Two notable choices for $R_0$ are $L[\underline{μ}^{\pm 1}]\big/(μ_s-μ_{s'}\big|  s\sim s')$ and its quotient $L[\underline{μ}^{\pm 1}]\big/(μ_s-μ_{s'}\big|  s, s' \in \mathcal{R})$.
\end{example}

\begin{remark}\label{rem:properfractionfield} Notice that if $R$ is a proper ring of definition, then so is $\Frac(R)$. Also, notice that any finite extension of a proper field is also a proper field.
\end{remark}

\section{The simple modules of $\Br^{K}(W)$}\label{sec:structure}

\subsection{Structure of simple modules}
\label{sec:structuregeneral}

Let $W$ be a complex reflection group and $K$ a proper field for $W$.
In this subsection we obtain some structural properties that simple modules of the Brauer-Chen algebra of $W$ satisfy. We show that every simple module of $\Br^K(W)$ restricted to $KW$ is isomorphic to a module induced by a simple module of the stabilizer in $W$ of a transverse collection. This spells out the construction of simple modules in the next subsection.

All mentioned modules are finite dimensional. If it is not clear why a constructed module is finite dimensional, we will give an explicit argument.

For a collection of transverse hyperplanes $B$ we will denote by $e_B$ the element $\prod_{H\in B} e_H$ of $\Br^K(W)$; this product is well defined since $e_{H},e_{H'}$ commute for transverse $H, H'$. If $B$ is empty, then we assume that $e_B=1$.

\begin{remark}\label{remark:WactiononB}
By relation (B2) of $\Br^K(W)$, it is clear that for $w\in W$ we have $we_Bw^{-1}=e_{wB}$.
\end{remark}
\begin{remark}\label{remark:ebmult}
Since we can rearrange the factors in $e_B$, we have the following multiplication property (see also Definition \ref{def:algebra}). If $H\in B$, then $e_He_B=δe_B$ (rearrange the factors of $e_B$ so that $e_H$ is first). If $H\pitchfork B$, then $B\cup \{H\}$ is a collection of transverse hyperplanes and $e_He_B=e_{B\cup\{H\}}$. Finally, if $H\not \pitchfork H'$ for some $H'\in B$, then $e_He_B=\sum_{s\in \mathcal{R}_{H'\to H}} μ_sse_B$ (rearrange the factors of $e_B$ so that $e_{H'}$ is first). With a symmetric argument we can obtain a similar result for the product $e_Be_H$, with the only difference arising in the last case, where we get $e_Be_H=e_B\sum_{s\in\mathcal{R}_{H\to H'}}μ_ss$ (see Remark \ref{rem:symrel}). 
\end{remark}

\begin{definition}
\label{def:maximal}
 Let $V$ be a $\Br^K(W)$-module and $B$ a transverse collection. The collection $B$ will be called $V$-maximal if it is maximal with the property that $e_BV\neq 0$. The subspace $e_BV$ will be denoted by $V_B$.
\end{definition} 

\begin{remark}
\label{remark:conjugatemaximals} By Remark \ref{remark:WactiononB}, one can see that for $w\in W$ we have $V_{wB}=wV_B$ and since the elements $w$ are automorphisms of $V$, any conjugate of a $V$-maximal collection is again $V$-maximal.
\end{remark} 

\begin{remark} Notice that if the empty collection is $V$-maximal, this means that $e_HV=0$ for all $H\in \mathcal{H}$.
\end{remark}

\begin{lemma}\label{equation1} Let $V$ be a $\Br^K(W)$-module and $B$ a $V$-maximal collection. Then, for all $v\in V_B$ and $H\in \mathcal{H}$ we have:

\begin{equation} 
\label{eq1}
e_Hv=\begin{cases}
δv , & \textnormal{ if }H\in B,\\
0, & \textnormal{ if }H \pitchfork B, \\
\sum_{s\in \mathcal{R}_{H'\to H}} μ_ssv, & \textnormal{ if }H \not\pitchfork H', H'\in B 
\end{cases} 
\end{equation}

\end{lemma}
\begin{proof} Since $V_B=e_BV$, the verification of the formula is immediate from Remark \ref{remark:ebmult}.
\end{proof}

We will show that if $V$ is simple, then it is the direct sum of the subspaces $V_B$ for $B$ in some orbit of $V$-maximal collections. Before that, we establish existence of a degree on a $KW$-module $V$, which will be very useful for the rest of this subsection. Recall from Definition \ref{def:propering} that $K$ contains a field $k_K$ containing the parameters $\underline{μ}$, and over which all subgroups of $W$ split. Also, $K$ is a finite extension of $k_K(δ)$.

\begin{lemma}
\label{lem:degree}
Let $V$ be a $k_K(δ)W$-module. Then, there is a degree map $\deg: V\to \mathbb{Z}\cup \{-\infty\}$ such that 
$\deg δ^{l}v = \deg v + l $ and $\deg xv \leq \deg v$ for all $v\in V$ and $x\in k_KW$. Also, $\deg v =-\infty $ if and only if $v=0$.
\end{lemma}
\begin{proof} We define the degree of a rational function in $δ$ to be the degree of its numerator minus the degree of its denominator; this does not change if we choose different expressions of the same rational function. It is also clear that this degree satisfies the properties described in the statement.

 Let $V$ be a $k_K(δ)W$-module. Since $W$ splits over $k_K$, there exists a $k_KW$-form of $V$, namely a $k_KW$-submodule $V'$ of $V$ such that the natural homomorphism
$k_K(δ)\otimes_{k_K}V'\to V$ is an isomorphism (in fact, this is true for any field $k_K$, for a proof see \cite[Lemma 5.1]{Ma}). Let $v_1,\dots ,v_n$ be a basis of $V'$. Then, the aforementioned degree map on $k_K(δ)$ can be extended to $k_K(δ) \otimes_{k_K} V'$ setting $\deg (\sum_i^n λ_i \otimes v_i)=\max\{\deg λ_i, i=1,\dots ,n \}$ for all $λ_1,\dots ,λ_n \in k_K(δ)$. This degree satisfies the properties of the statement. 

To verify  that $\deg xv \leq v$ for all $v\in V, x\in k_KW$, notice that multiplication on $V$ by elements of $k_K$ does not raise the degree and that the action of a $w\in W$ on $V$ with respect to the basis $1\otimes v_1, \dots , 1\otimes v_n$ is given by a matrix with coefficients in $k_K$ since $V'$ is a $k_KW$-module.
\end{proof}

\begin{definition}\label{def:orbvb} Let $V$ be a $\Br^K(W)$-module and $\mathcal{B}$ an orbit under the action of $W$ on $V$-maximal collections. The sum of the subspaces $V_B, B\in \mathcal{B}$ will be denoted $V_{\mathcal{B}}$.
\end{definition}

\begin{proposition}
\label{prop:directsum}
Let $V$ be a $\Br^K(W)$-module and $\mathcal{B}$ an orbit of $V$-maximal collections. Then,
\begin{enumerate}
\item $V_{\mathcal{B}}=\oplus_{B\in \mathcal{B}}V_B$,

\item for any $B\in \mathcal{B}$, the set $\Stab(V_B):=\{w\in W\big|  wV_B=V_B\}$ coincides with $\Stab(B)$ (recall that $\Stab(B)$ denotes the stabilizer of $B$ in $W$);

\item $V_{\mathcal{B}}$ is a $\Br^K(W)$-submodule of $V$, and $\Res_{KW}(V_\mathcal{B})\cong  \Ind_{K\Stab(B)}^{KW}(V_B)$.
\end{enumerate}
\end{proposition}

\begin{proof}

On $V$, which is naturally a finite dimensional $k_K(δ)W$-module (recall that $K$ is a finite extension of $k_K(δ)$), we fix a degree as in Lemma \ref{lem:degree}. Let $v=\sum_{B\in \mathcal{B}} v_B$, with $v_B\in V_{B}$ and suppose that $v=0$. Fix a hyperplane $H$. The element $e_H$ acts as multiplication by $δ$ on $v_B$ if $H\in B$, and by some element of $k_KW$ otherwise, according to Lemma \ref{equation1}.

Since $v=0$, we have that $\sum_{B\in \mathcal{B}, H\in B}v_B=-\sum_{B\in \mathcal{B},H\not\in B} v_B$. In particular, the degrees of the two sides are equal. Suppose that both sides are non zero. On the left hand side, $e_H$ acts by $δ$ and hence raises the degree by $1$, while on the right hand side, acting by some element of $k_KW$ on every $v_B$ with $H\not\in B$, it does not raise its degree which leads to a contradiction. So the two sides should be equal to $0$ meaning that $v=\sum_{B\in \mathcal{B},H\in B}v_B=0$. Doing the same, now for all $H$ in some fixed collection $B\in \mathcal{B}$, we get that $v_B$ is zero. So, $v_B=0$ for all $B\in \mathcal{B}$.

Property (1) implies especially that the subspaces $V_B$ are distinct, which implies (2).

From (1) and (2), it is clear that the $KW$-module $V_{B\in \mathcal{B}}=\oplus_{\mathcal{B}} V_B$ (Definition \ref{def:orbvb}) is isomorphic to the induced $KW$-module $\Ind_{K\Stab(B)}^{KW}(V_B)$, for any $B\in \mathcal{B}$. Also, by Lemma \ref{equation1}, every element of $\Br^K(W)$ acts on each $V_B$ as some element of $KW$, which implies that $V_{\mathcal{B}}$ is, in fact, a $\Br^K(W)$-submodule of $V$.
\end{proof}

\begin{remark}\label{rem:vbdet} Note that the $\Br^K(W)$-submodule $V_{\mathcal{B}}$ is completely determined by the $K\Stab(B)$-module $V_B$: as a $KW$-module it is induced by $V_B$, and Lemma \ref{equation1} gives the action of the elements $e_H,H\in \mathcal{H}$. We construct all such modules in the next subsection. 
\end{remark}

%\begin{corollary} If $V$ is a simple $\Br^K(W)$-module, then, $V=V_\mathcal{B}$. In %particular, as a $KW$-module it is induced by $V_B$, for any $B\in \mathcal{B}$.
%\end{corollary}

\begin{lemma}\label{lem:imageofe} Let $V$ be a $\Br^K(W)$-module and $\mathcal{B}$ an orbit of $V$-maximal collections, such that $V=V_\mathcal{B}$. Then, for every transverse collection $C$, we have $e_CV=\oplus_{B\in\mathcal{B}, C\subseteq B} V_B$. In particular, $e_CV\neq 0$ if and only if $C\subseteq B$ for some $B\in \mathcal{B}$.
\end{lemma}
\begin{proof}
From Equation \ref{eq1} of Lemma \ref{equation1} one can see that for any $H$ and $B'\in \mathcal{B}$ we have $e_HV_{B'}\subseteq \oplus_{B\in \mathcal{B}, H\in B} V_B$. We also see that $e_H$ acts as multiplication by $δ$ on $\oplus_{B\in \mathcal{B},H\in B} V_B$. So, since $V=\oplus_{B'\in \mathcal{B}} V_{B'}$, then $e_H V=\oplus_{B\in \mathcal{B},H\in B} V_B$. Now, if $H_1, H_2$ are transverse, then $e_{H_1}e_{H_2}V=e_{H_1}V\cap e_{H_2}V$. To see why this is, recall that $e_{H_1}, e_{H_2}$ commute since $H_1,H_2$ are transverse (Definition \ref{def:algebra}). This implies that $e_{H_1}e_{H_2}V\subseteq e_{H_1}V\cap e_{H_2}V$. On the other hand, since both $e_{H_1}, e_{H_2}$ act on $e_{H_1}V\cap e_{H_2}V$ as scalar multiplcation by $δ$, we obtain that $e_{H_1}V\cap e_{H_2}V\subseteq e_{H_1}e_{H_2}V$, which yields equality of the two subspaces. Hence, 

\begin{equation}
\label{eq2} e_CV=\cap_{H\in C} (\oplus_{B\in \mathcal{B},H\in B}V_B)=\oplus_{B\in \mathcal{B},C\subseteq B} V_B
\end{equation}
\end{proof}

\begin{corollary}\label{cor:maxcols} Let $V$ be a $\Br^K(W)$-module and $\mathcal{B}$ an orbit of $V$-maximal collections such that $V=V_{\mathcal{B}}$. Then, $\mathcal{B}$ is the only orbit of $V$-maximal collections. 
\end{corollary}
\begin{proof}
By Lemma \ref{lem:imageofe}, $e_CV\neq 0$ if and only if $C\subseteq B$ for some $B\in \mathcal{B}$. Hence, if $C$ is $V$-maximal, then $C\in \mathcal{B}$.
\end{proof}

\begin{proposition}
\label{prop:irreducibility}
Let $V$ be a $\Br^K(W)$-module and $\mathcal{B}$ an orbit of $V$-maximal collections.
Then $V$ is simple if and only if $V=V_{\mathcal{B}}$ and $V_B$ is a simple $K\Stab(B)$-module for all $B\in \mathcal{B}$. In this case, it is absolutely irreducible.
\end{proposition}

\begin{proof} The proof follows the lines of \cite[Proposition 5.4]{Ma}. 

By Definition \ref{def:maximal} the subspace $V_B$ and, hence, the $\Br^K(W)$-submodule $V_\mathcal{B}$ of $V$ are non-zero. So, if $V$ is simple, then $V=V_{\mathcal{B}}$. Furthermore, if $V_B'$ is a proper and non-zero $K\Stab(B)$-submodule of $V_B$ for some $B\in \mathcal{B}$, then, since every $e_H$ acts on every $V_B$ as an element of $KW$ (Lemma \ref{equation1}), the subspace $\oplus_{w\in W/\Stab(B)} wV_B'$ is stable under $e_H$ and is, hence, a proper and non-zero $\Br^K(W)$-submodule of $V$. So, if $V$ is a simple $\Br^K(W)$-module, then $V_B$ is a simple $K\Stab(B)$-module for all $V$-maximal $B$.

For the converse, let $V$ be a $\Br^K(W)$-module such that $V=V_{\mathcal{B}}$. Again, $V$ is naturally a finite dimensional $k_K(δ)W$-module; fix, hence, a degree on $V$ as in Lemma \ref{lem:degree}.

Suppose that $V_B$ is a simple $K\Stab(B)$-module for some $B$ and let $v\in V_B\backslash \{0\}$. Then $K\Stab(B)v=V_B$, and so $\Br^{K}(W)v$ contains $\Br^{K}(W)V^K_B=V$. We will show that for every non zero $v\in V$, there exists $B\in \mathcal{B}$ such that $e_Bv\neq 0$. As a consequence, $\Br^K(W)v$ contains an element of some $V_B$, in which case, as we just showed, $\Br^K(W)v=V$. 

Let $v\in V\backslash \{0\}$ and write  $v=\sum_{B\in \mathcal{B}} v_B$ where $v_B\in V_B$ (recall that $V=V_{\mathcal{B}}$). Let $v_{B_0}$ be a term of highest degree. If $e_{B_0}v=0$, then:

$$e_{B_0}v_{B_0}=-\sum_{B\in \mathcal{B}\backslash\{B_0\}} e_{B_0} v_B.$$

Now, since $v_{B_0}$ is nonzero, the left hand side is equal to $δ^{\lvert B_0 \rvert}v_{B_0}$ and has a degree of $\lvert B_0 \rvert + \deg v_{B_0}\neq -\infty $. Since, for $B\neq B_0$, $e_{B_0}v_B$  has a degree of at most $\deg v_B + \lvert B_0 \rvert -1$, we see that the degree of the right hand side is strictly smaller than the one on the left, leading to a contradiction. So, $e_{B_0}v\neq 0$ and, hence, $\Br(W)v=V$.

For the absolute irreducibility, suppose that $V$ is simple; hence, $V=V_\mathcal{B}$ and $V_B$ is a simple $K\Stab(B)$-module for all $B\in \mathcal{B}$.
We need to show that, for any algebraic extension $K'$ of $K$, the $\Br^{K'}(W)$-module $V':=K' \otimes_K V$ is simple; it suffices to consider finite extensions. Recall, first of all, that every finite extension of a proper field (Definition \ref{def:propering}) is again a proper field, so the part of the statement we just proved holds for $K'$ as well, i.e. if $\mathcal{B}$ is an orbit of $V'$-maximal collections such that $V'=V'_{\mathcal{B}}$ and $V'_B$ is a simple $K'\Stab(B)$-module for some $B\in \mathcal{B}$, then $V'$ is simple.

First of all, for all $e_H$ we have $e_HV'=e_H(K'\otimes_KV)=K'\otimes_Ke_HV$ and, hence, $e_BV'=K'\otimes_K e_BV$. From this it is clear that $\mathcal{B}$ is an orbit of $V'$-maximal collections as well, and that $V'_B=K' \otimes_K V_B$ for all $B\in \mathcal{B}$.
Since $V=V_\mathcal{B}=\oplus_{B\in \mathcal{B}}V_B$, then $V'$ decomposes as $V'=\oplus_{B\in \mathcal{B}} (K'\otimes_K V_B)$, and the latter is equal to $V'_\mathcal{B}$ by the previous comment. Furthermore, since $V_B$ is a simple $K\Stab(B)$-module and $K$ already contains $k_K$, over which $\Stab(B)$ splits, then $V'_B=K'\otimes_K V'_B$ is a simple $K'\Stab(B)$-module. Hence, $V'$ is a simple $\Br^{K'}(W)$-module.
\end{proof}

%\begin{corollary}
%\label{cor:absirreducibility}
%If $V$ is a simple $\Br^K(W)$-module, then it is absolutely irreducible.
%With the notation of Proposition \ref{prop:irreducibility}, if $V=V_\mathcal{B}$ and $V_B$ is a simple $K\Stab(B)$-module for some $B$, then $V$ is absolutely irreducible. In this case, it is absolutely irreducible.
%\end{corollary}

%\begin{proof} 
%By Proposition \ref{prop:irreducibility}
%We need to show that, for any algebraic extension $K'$ of $K$, the $\Br^{K'}(W)$-module $K' \otimes_K V$ is simple; it suffices to consider finite extensions. Recall, first of all, that every finite extension of a proper field (Definition \ref{def:propering}) is a proper ring, i.e., Proposition \ref{prop:irreducibility} above holds for $K'$, as well.

%So let $K'$ be a finite extension of $K$ and consider the module $V' = K'\otimes_K V$ which, since $V=V_\mathcal{B}=\oplus_{B\in \mathcal{B}}V_B$, decomposes as $\oplus_{B\in \mathcal{B}} K'\otimes_K V_B$. Now, $K' \otimes_K V_B=K' \otimes_K e_BV= e_B (K' \otimes_K V) =V'_B$, so we have that $V'=\oplus_{B\in \mathcal{B}} V'_B=V'_{\mathcal{B}}$. Also, since $V_B$ is a simple $K\Stab(B)$-module and $K$ already contains $k_W$, over which $\Stab(B)$ splits, then $V'_B=K'\otimes_K V'_B$ is a simple $K'\Stab(B)$-module. Hence. $V'$ satisfies the conditions of Proposition \ref{prop:irreducibility} and is, therefore, a simple $\Br^{K'}(W)$-module.
%\end{proof}

\begin{proposition} 
\label{prop:nonisomorphic}
Let $V,V'$ be two $\Br^K(W)$-modules with respective orbits of maximal collections $\mathcal{B},\mathcal{B}'$, such that $V=V_\mathcal{B}$ and $V'=V'_{\mathcal{B}'}$. Then $V$ and $V'$ are isomorphic if and only if $\mathcal{B}=\mathcal{B}'$ and $V_B,V_B'$ are isomorphic $K\Stab(B)$-modules for some (hence, for all) $B\in \mathcal{B}$.
\end{proposition}
\begin{proof} Let $B$ be a collection of transverse hyperplanes. If the $\Br^K(W)$-modules $V$ and $V'$ are isomorphic, then $e_BV\neq 0$ if and only if  $e_BV' \neq 0$, which, by Lemma \ref{lem:imageofe}, implies that $\mathcal{B}=\mathcal{B}'$.
Now if $B\in \mathcal{B}$, it is clear that a $\Br^K(W)$-isomorphism $V\to V'$ induces a $K\Stab(B)$-isomorphism $V_B=e_BV \to e_BV'=V'_B$. 

Conversely, as mentioned in Remark \ref{rem:vbdet} the $K\Stab(B)$-module $V_B$ completetely determines the $\Br^K(W)$-module $V_{\mathcal{B}}$. So, if $V_B, V'_B$ are isomorphic $K\Stab(B)$-modules, for some $B$, then $V,V'$ are isomorphic $\Br^K(W)$-modules. 
\end{proof}

The above result implies that simple $\Br^K(W)$-modules can be parametrized up to conjugacy by certain pairs $(B,V)$ where $B$ is a collection of transverse hyperplanes and $V$ is a simple $K\Stab(B)$-module. In the next subsection, we describe the pairs that give rise to simple $\Br^K(W)$-modules.

\subsection{Construction of simple modules of $\Br^K(W)$}
\label{sec:construction}

This subsection contains the construction of simple modules, as dictated by the results of the previous subsection, starting with a collection $B$ of transverse hyperplanes and a module $V$ of its stabilizer. The modules $V_{\mathcal{B}}$ are a prototype for this construction. The pairs $(B,V)$ that give rise to such modules are called ``admissible''. We also obtain a characterization of these pairs which will be useful for direct calculation.

\subsubsection{Admissibility}

For an element $y$ of some $KW$-module, let $\Ann_{KW}(y):=\{x\in KW\big| xy=0\}$ be the annihilator of $y$ inside $KW$. Recall that $\mathcal{C}$ denotes the set of collections of transverse hyperplanes.

\begin{definition}[Admissibility]
\label{def:admissible} Let $(B,V)$ be a pair consisting of a collection $B$ of transverse hyperplanes and a $K\Stab(B)$-module $V$ and let $\hat{V}$ denote the image of $V$ inside the induced module $\Ind_{K\Stab(B)}^{KW}V$. The pair $(B,V)$ will be called $K$-\textit{admissible}, or just \textit{admissible} when there is no ambiguity as to the field $K$, if $\Ann_{KW}(e_B)\cdot \hat{V}=0$ (here $e_B$ is considered as an element of the $KW$-module $\Br^K(W)$). If there is at least one admissible pair $(B,V)$  with $V\neq 0$, then the collection $B$ will be called $K$-\textit{admissible}, or just \textit{admissible}. The set of $K$-admissible collections will be denoted by $\mathcal{C}_{adm}^K$.
\end{definition}

\begin{remark}\label{rem:emptyadmissibility} One can see that if $B$ is the empty collection, then all pairs $(B,V)$ where $V$ is a $KW$-module are $K$-admissible.
\end{remark}

\begin{definition}\label{def:rel} For $H_1,H_2,H\in \mathcal{H}$, we define the following element of $\mathbb{Z}[\underline{μ}]W$:

$$σ^H_{H_1,H_2}:=\sum_{s\in \mathcal{R}_{H_1\to H}} μ_ss-\sum_{s\in \mathcal{R}_{H_2\to H}} μ_ss.$$
For $B\in \mathcal{C}$ and $H\in \mathcal{H}$, we define the set
$$Σ_B^H:=\{σ^H_{H_1,H_2}\big| H_1,H_2 \in B, H_1,H_2\not\pitchfork H\}.$$
Finally, we define the sets 
$$Σ_B:=\cup_{H\in \mathcal{H}} Σ_B^H \textnormal{  and }\Rel(B):=(\mathcal{R}_B-1)\cup Σ_B,$$
where $\mathcal{R}_B:=\{r\in \mathcal{R}\big| H_r\in B\}$ and $\mathcal{R}_B-1=\{r-1\big| r\in \mathcal{R}_B\}$.
\end{definition}

\begin{remark}
One can see that $wσ^H_{H_1,H_2}w^{-1}=σ^{wH}_{wH_1,wH_2}$ for all
$w\in W$ and $H_1,H_2,H\in\mathcal{H}$. Furthermore, $H_1,H_2 \not \pitchfork H$ if and only if $wH_1,wH_2 \not \pitchfork wH$. This implies that $wΣ_Bw^{-1}=Σ_{wB}$. It is also  straightforward to verify the same property for the set $\mathcal{R}_B-1$, which yields  $w\Rel(B)w^{-1}=\Rel(wB)$.
\end{remark}

\begin{lemma} Let $B\in \mathcal{C}$. Then, $\Rel(B)\subseteq \Ann_{KW}(e_B)$
\end{lemma}
\begin{proof} By defining relation (B3) of $\Br^K(W)$, we have that $re_{H_r}=e_{H_r}$, for all $r\in \mathcal{R}$. So, if $r\in \mathcal{R}_B$, then rearranging the factors of $e_B$ so that $e_{H_r}$ comes first, we get that $re_B=e_B$. This gives the inclusion $\mathcal{R}_B-1 \subseteq \Ann_{KW}(e_B)$.
Now, recall that if $H'\in B$ is non-transverse with $H$, then $e_He_B=\sum_{s\in\mathcal{R}_{H'\to H}} μ_sse_B$ (Remark \ref{remark:ebmult}). Hence, for $H_1,H_2\in B$ non-transverse with $H$, we have $σ^H_{H_1,H_2}e_B=e_He_B-e_He_B=0$, i.e. $σ^H_{H_1,H_2}\in \Ann_{KW}(e_B)$. This gives the inclusion $Σ_B\subseteq \Ann_{KW}(e_B)$, which concludes the proof.
\end{proof}

\subsubsection{The construction of the $\Br^K(W)$-module  $(B\big| V)$}

For the rest of this subsection, we fix $B_0\in \mathcal{C}$ and a $K\Stab(B_0)$-module $V_0$ with the only assumption that $\Rel(B_0)\hat{V_0}=0$.

Let $\mathcal{B}$ be the orbit of $B_0$, and set $V:=\Ind_{K\Stab(B_0)}^{KW} V_0$. Then, $V$ is the direct sum of the subspaces $w\hat{V_0}, w\in W$. Since $w\hat{V_0}$ depends only on the left coset $w\Stab(B_0)$, or, equivalently, the collection $wB_0$, we denote it by $V^B$, where $B=wB_0$. In particular, $\hat{V_0}=V^{B_0}$. So, we have $V=\oplus_{B\in \mathcal{B}}V^B$.

Since  $\Rel(B_0)\hat{V_0}=0$, by conjugation, we have $\Rel(B)V^{B}=0$ for all $B\in \mathcal{B}$. Now notice that for every $H\in \mathcal{H}$, the relations $Σ_B^H\cdot V^{B}=0$ imply that we can define an operator $ε_{H,B}$ on $V^B$ by:

\begin{equation}
\label{eq:epsilondef}
ε_{H,B}v:=\begin{cases}
δv, & \textnormal{ if } H\in B,\\
0, &\textnormal{ if } H \pitchfork B, \\
\sum_{s\in \mathcal{R}_{H'\to H}} μ_ss v, &\textnormal{ if } H \not\pitchfork H', H'\in B.
\end{cases}
\end{equation}
Since $V=\oplus_{B\in \mathcal{B}} V^B$, this extends to an operator $ε_H$ on $V$, whose restriction on each $V^B$ equals $ε_{H,B}$, for every $H\in \mathcal{H}$. We will show that this makes $V$ a $\Br^K(W)$-module. Before that, we have two rather technical but very useful lemmas. For $B'\in \mathcal{B}$ let $ε_H^{(B')}$ be the projection of $ε_H$ onto $V^{B'}$, i.e. for $v\in V^B$,

\begin{equation}
\label{eq:epsilonprojdef}
ε_H^{(B')}v:=\begin{cases}
δ_{BB'}\cdot δv, &\textnormal{ if } H\in B, \\
0, &\textnormal{ if } H \pitchfork B, \\
\sum_{s\in \mathcal{R}_{H'\to H, B\to B'}} μ_ss v, &\textnormal{ if } H \not\pitchfork H', H'\in B,
\end{cases}
\end{equation}
where $δ_{BB'}$ is the Kronecker delta for $B,B'$, i.e. $0$ or $1$ if $B\neq B'$ and $B=B'$, respectively.

\begin{lemma}\label{lem:precom} Let $B,B'\in \mathcal{B}$ be such that $\mathcal{R}_{B\to B'}\neq \emptyset$, and let $H,H'$ be distinct hyperplanes belonging to $B,B'$ respectively. Then, $H\pitchfork H'$ if and only if $H$ or $H'$ belongs to $B\cap B'$. 
\end{lemma}

\begin{proof} Let $s\in \mathcal{R}_{B\to B'}$, and suppose that $H'\not \in B$.
If $H \pitchfork H'$, then by Lemma \ref{lem:transverse1}, $\mathcal{R}_{H\to H'}=\emptyset $, and hence, $sH \neq H'$. So, $H'=sH''$ for some $H'' \in B$ other than $H$. Now, since $H'', H' \pitchfork H$ and $sH''=H'$, then, by Lemma \ref{lem:transverse2}, we have $H=sH \in B' $. 

Conversely, if $H\in B'$, and $H\neq H'$, then, of course, $H\pitchfork H'$. 
\end{proof}

\begin{lemma}\label{lem:commutativity} Let $B,B'\in \mathcal{B}$, and $H'\in \mathcal{H}$, such that $H'\not \pitchfork B$. Then, for all $v\in V^{B}$, we have 
$$ε_{H'}v=\sum_{s\in \mathcal{R},H'\in sB} \frac{1}{\lvert B \backslash sB \rvert} μ_ssv,$$
and 
$$ε_{Η' }^ {(Β')}v=δ_{H'\in B'}\cdot \sum_{s\in \mathcal{R}_{B\to B'}} \frac{1}{\lvert B \backslash B' \rvert}μ_ssv,$$
where $δ_{H'\in B'}$ is $1$ if $H'\in B'$ and $0$ otherwise.
Finally,  if $H_1,H_2\in B'$ are non-transverse with $B$, then
$ε_{H_1}^{(B')}v=ε_{H_2}^{(B')}v$.

\end{lemma}

\begin{proof} 

For every $H\in B$ such that $H\not \pitchfork H'$, of which, by assumption, there is at least $1$, we have $ε_{H'}v=\sum_{s\in \mathcal{R}_{H\to H'}} μ_ss v$. So, if we denote by $N(H',B)$ the number of $H\in B, H\not \pitchfork H'$, we get 

$$ε_{H'}v=\frac{1}{N(H',B)}\cdot \sum_{s\in \mathcal{R}, H'\in sB} μ_ssv,$$
 and Lemma \ref{lem:precom} implies that $N(H',B)=\lvert B\backslash sB\rvert $, for any $s\in \mathcal{R}$ such that $sB$ contains $H'$.

So, $$ε_{H'}v=\sum_{s\in \mathcal{R},H'\in sB} \frac{1}{\lvert B\backslash sB \rvert }μ_ssv $$
Taking the projection of the two sides onto the subspace $V^{B'}$ gives:

$$ε_{H'}^{(B')}v=\sum_{s\in \mathcal{R}_{B\to B' },H'\in sB} \frac{1}{\lvert B\backslash sB \rvert }μ_ssv=δ_{H'\in B'}\cdot \sum_{s\in \mathcal{R}_{B\to B'}} \frac{1}{\lvert B\backslash B' \rvert }μ_ssv,$$
with the convention that if $B=B'$ (in which case $1/\lvert B\backslash B' \rvert$ is not defined), since $Η'\not \in B$, i.e. $δ_{H'\in B'}=0$, the last expression equals $0$.

Finally, if $H_1,H_2 \in B'$ are not transverse with $B$, then, applying the last equality for $H_1,H_2$ gives:

$$ε_{H_1}^ {(B')}v=\sum_{s\in \mathcal{R}_{B\to B'}} \frac{1}{\lvert B\backslash B' \rvert }μ_ssv=ε_{H_2}^ {(B')}v.$$
\end{proof}

\begin{theoremanddefinition}\label{thdef:construction} Mapping $e_H\mapsto ε_H, H\in \mathcal{H}$ makes $V$ a $\Br^K(W)$-module, i.e. the operators $ε_H$ satisfy the defining relations of $\Br^K(W)$:
\begin{itemize}
\item[(B1)] $ε_H^2=δε_H$,
\item[(B2)] $wε_Hw^{-1}=ε_{wH}$,
\item[(B3)] $wε_H=ε_H$, if $H \subseteq \ker(w-1)$,
\item[(B4)] $ε_{H_1}ε_{H_2}=ε_{H_2}ε_{Η_1}$, if $H_1\pitchfork H_2$,
\item[(B5)] $ε_{H_1}ε_{H_2}=\sum_{s\in \mathcal{R}_{H_2\to H_1}} μ_ss ε_{H_2}$, if $H_1\not \pitchfork H_2$,
\end{itemize}
for all $w\in W$ and $H_1,H_2,H \in \mathcal{H}$. The $\Br^K(W)$-module $V$ is denoted $(B_0\big| V_0)$ (we remind that the pair $(B_0,V_0)$ is such that $\Rel(B_0)\hat{V}_0=0$).

\end{theoremanddefinition}

\begin{proof}

Notice, first of all, that the definition of $ε_H$ implies that 
$ε_HV \subseteq \oplus_{B\in \mathcal{B}, H \in B} V^B$. In fact, since 
$ε_H$ acts as multiplication by $δ$ on $V^B$ if $H\in B$, we have $ε_HV = \oplus_{B\in \mathcal{B}, H \in B} V^B$, and we also get the first relation. 

For (B2), we verify that it holds on each $V^B$; this suffices, since $V=\oplus_{B\in \mathcal{B}}V^B$. Let $B\in \mathcal{B},v\in V^B$ and $w\in W$. There are three cases to consider. For the first, let $wH\in B$, which implies that $ε_{wH}v=δv$. Then $H\in w^{-1}B$, and since $w^{-1}v\in V^{w^{-1}B}$, we have $wε_{H}w^{-1}v=wε_{H}(w^{-1}v)=wδw^{-1}v=δv=ε_{wH}v$. For the second, let $wH\pitchfork B$ implying that $ε_{wH}v=0$. Then, $H\pitchfork w^{-1}B$, and, since $w^{-1}v\in V^{w^{-1}B}$, we have $wε_{H}w^{-1}v=wε_{H}(w^{-1}v)=0=ε_{wH}v$. Finally, for the third case, let $wH$ be non-transverse with some $H' \in B$. Hence, 
$ε_{wH}v=\sum_{s\in \mathcal{R}_{H'\to wH}}μ_ssv$. Now, $H$ is non-transverse with $w^{-1}H'\in w^{-1}B$ and, computing as before, we have $wε_Hw^{-1}v=\sum_{s\in \mathcal{R}_{w^{-1}H'\to H}}μ_sw^{-1}swv$. As $s$ runs through $\mathcal{R}_{w^{-1}H\to H}$ the element $wsw^{-1}$ runs through $\mathcal{R}_{H'\to wH}$; also, by their definition $μ_s=μ_{wsw^{-1}}$. So, $wε_Hw^{-1}v=\sum_{s\in \mathcal{R}_{H'\to wH}}μ_ssv=ε_{wH}v$.

For (B3), if $H\subseteq \ker(w-1)$, and $w\neq 1$, then $w$ is a reflection with hyperplane $H$. Now $w-1 \in \mathcal{R}_{B}-1$ for every $B\in \mathcal{B}$ containing $H_w=H$, and, since $\mathcal{R}_{B}-1$ is contained in $\Rel(B)$, which, by assumption, annihilates $V^{B}$, then $(w-1)V^{B}=0$. Since 
$ε_HV=\oplus_{B\in \mathcal{B}, H\in B} V^B$, this gives $(w-1)ε_H=0$, or $wε_H=ε_H$.

For (B5), if $H_1\not \pitchfork H_2$, then, for every $B\in \mathcal{B}$ containing $H_2$ and $v\in \hat{V}^{B}$, we have $ε_{H_1}v=\sum_{s\in \mathcal{R}_{H_1\to H_2}} μ_ssv$. Since $ε_{H_2}V=\oplus_{B\in \mathcal{B}, H_2\in B} V^B$, the property follows.

We show now (B4), i.e. $ε_{H_1}ε_{H_2}=ε_{H_2}ε_{H_1}$ if $H_1\pitchfork H_2$, proving the relation on $V^B$, for every $B\in \mathcal{B}$. 
For the rest of the proof, fix a $B\in \mathcal{B}$ and a $v\in V^B$.
We distinguish four cases. The proof of the first three is relatively short.

\begin{enumerate}
\item $H_1\in B$, which implies that $ε_{H_1}v=δv$. If $H_2\pitchfork B$ or $H_2\in B$, then $ε_{H_2}v=0$ or $δv$, respectively; in any case, it commutes with $ε_{H_1}$. So suppose that $H_2$ is non-transverse with some $H\in B$, which implies that
$ε_{H_2}v=\sum_{s\in \mathcal{R}_{H\to H_2}}μ_ssv$.  Now,  since $H, H_2 \pitchfork H_1 $, then, by Lemma \ref{lem:transverse2}, for all $s\in \mathcal{R}_{H\to H_2}$, we have $H_1=sH_1\in sB$. Since $sv\in V^{sB}$, this implies that $ε_{H_1}(sv)=δsv$ for all $s\in \mathcal{R}_{H\to H_2}$, and, hence, $ε_{H_1}ε_{H_2}v=δ\sum_{s\in \mathcal{R}_{H\to H_2}}μ_ssv=ε_{H_2}δv=ε_{H_2}ε_{H_1}v$. 

\item $H_1\pitchfork B$, which implies that $ε_{H_1}v=0$. As before, if $H_2\pitchfork B$ or $H_2\in B$, then $ε_{Η_2}v=0$ or $δv$, respectively, and $ε_{H_1},ε_{H_2}$ commute.

So, suppose $H_2\not\pitchfork H\in B$, and, hence, $ε_{H_2}v=\sum_{s\in \mathcal{R}_{H\to H_2}}μ_ssv$. As before, by Lemma \ref{lem:transverse2}, since $H,H_2\pitchfork H_1$, if $s\in \mathcal{R}_{H\to H_2}$, then $sH_1=H_1$, which implies that $H_1=sH_1 \pitchfork sB$. Hence, $ε_{H_1}V^{sB}=0$, for all $s\in \mathcal{R}_{H\to H_2}$, and, so, $ε_{H_1}ε_{H_2}v=0=ε_{H_2}ε_{H_1}v$.

\item There are $H_1',H_2'\in B$, such that $H_1'$ is non-transverse with $H_1$ but transverse with $H_2$, and $H_2'$ is non-transverse with $H_2$ but transverse with $H_1$. Then, $ε_{H_1}v=\sum_{s\in \mathcal{R}_{H_1' \to H_1}}μ_ssv$ and $ε_{H_2}v=\sum_{s\in \mathcal{R}_{H_2'\to H_2}}μ_ssv$.  Again, by Lemma \ref{lem:transverse2}, since $H_1',H_2 \pitchfork H_2'$, if  $s\in \mathcal{R}_{H_1' \to H_1}$, then $Η_2'=sH_2' \in sB$, and, so, $ε_{H_2}sv=\sum_{s'\in \mathcal{R}_{H_2'\to H_2}}μ_{s'}s'sv$. Hence,
 
 $$ε_{H_2}ε_{H_1}v=\sum_{\substack{s\in \mathcal{R}_{H_1'\to H_1}\\s'\in \mathcal{R}_{H_2'\to H_2}}} μ_sμ_{s'}ss' v.$$ With the same reasoning, we obtain the same result for 
 $ε_{H_1}ε_{H_2}v$.

\end{enumerate}

So now, we can assume that there is $H\in B$ non-transverse with both $H_1$ and $H_2$. To lighten the notation, for the rest of the proof, $s$ and $s'$ will always denote reflections, and we will omit mentioning $s,s'\in \mathcal{R}$, which will be assumed for every sum. Also, for $H'\in \mathcal{H}$ let $\mathcal{B}_{H'}$ be the set of elements of $\mathcal{B}$ that contain $H'$.
We will be using the results of Lemma \ref{lem:commutativity} repeatedly. The rest of the proof is rather long but most of it is a detailed  regrouping of certain sums.
We have:

\begin{equation*}
\begin{split}
ε_{H_2}ε_{H_1}v  &  =ε_{H_2}\sum_{sB\in \mathcal{B}_{H_1}} \frac{1}{\lvert B\backslash sB \rvert} μ_ssv= \sum_{sB\in \mathcal{B}_{H_1}\cap \mathcal{B}_{H_2}}\frac{δ}{\lvert B\backslash sB \rvert} μ_ssv 
+\sum_{\substack{sB\in \mathcal{B}_{H_1}\backslash \mathcal{B}_{H_2} \\ s'sB\in \mathcal{B}_{H_1}\cap \mathcal{B}_{H_2}}} \frac{μ_{s'}μ_s}{\lvert B\backslash sB \rvert \lvert sB\backslash s'sB \rvert}s'sv.\\
& =\sum_{sB \in \mathcal{B}_{H_1}\cap \mathcal{B}_{H_2}} \frac{δ}{\lvert B\backslash sB \rvert} μ_ssv + Σ_1 v,
\end{split}
\end{equation*}
where $Σ_1:=\sum_{\substack{sB\in \mathcal{B}_{H_1}\backslash \mathcal{B}_{H_2} \\ s'sB\in \mathcal{B}_{H_1}\cap \mathcal{B}_{H_2}}} \frac{μ_{s'}μ_s}{\lvert B\backslash sB \rvert \lvert sB\backslash s'sB \rvert}s's$. Similarly, we compute,

\begin{equation*}
\begin{split}
ε_{H_1}ε_{H_2}v &=\sum_{sB \in \mathcal{B}_{H_1}\cap \mathcal{B}_{H_2}}\frac{δ}{\lvert B\backslash sB \rvert} μ_ssv + Σ_2v,
\end{split}
\end{equation*}
where $Σ_2:=\sum_{\substack{sB\in \mathcal{B}_{H_2}\backslash \mathcal{B}_{H_1} \\ s'sB\in \mathcal{B}_{H_1}\cap \mathcal{B}_{H_2}}} \frac{μ_{s'}μ_s}{\lvert B\backslash sB \rvert \lvert sB\backslash s'sB \rvert}s's.$
\vspace{20pt}
The first term in both sums being the same, we need to show that $Σ_1v=Σ_2v$. For convenience, let 
$f(s',s):=\frac{μ_{s'}μ_s}{\lvert B\backslash sB \rvert \lvert sB\backslash s'sB \rvert}s's,$
which makes sense as long as $B\neq sB$ and $sB\neq s'sB$.

The following rearrangement of $Σ_1$ just amounts to writing $\mathcal{B}_{H_1}\backslash \mathcal{B}_{H_2} =\mathcal{B}_{H_1} \backslash (\mathcal{B}_{H_1}\cap \mathcal{B}_{H_1})$ while adding the condition $s'sB\neq sB$, which is true for the set of summation of $Σ_1$, but not for the following sums. We have:

\begin{equation*}
\begin{split}
Σ_1=\sum_{\substack{sB\in \mathcal{B}_{H_1}\backslash \mathcal{B}_{H_2} \\ s'sB\in \mathcal{B}_{H_1}\cap \mathcal{B}_{H_2}\\}} f(s',s)&= \sum_{\substack{sB\in \mathcal{B}_{H_1} \\  s'sB \in \mathcal{B}_{H_1}\cap \mathcal{B}_{H_2} \\s'sB\neq sB}} f(s',s) - \sum_{\substack{sB \in \mathcal{B}_{H_1}\cap \mathcal{B}_{H_2} \\s'sB \in \mathcal{B}_{H_1}\cap \mathcal{B}_{H_2}\\s'sB\neq sB}} f(s',s) =S_1-\sum_{\substack{sB\in \mathcal{B}_{H_1}\cap \mathcal{B}_{H_2} \\ s'sB\in \mathcal{B}_{H_1}\cap \mathcal{B}_{H_2}\\s'sB\neq sB}} f(s',s),\\
\end{split}
\end{equation*}
where $S_1:=\sum_{\substack{sB\in \mathcal{B}_{H_1} \\  s'sB \in \mathcal{B}_{H_1}\cap \mathcal{B}_{H_2} \\s'sB\neq sB}} f(s',s)$,
and, in the same way $Σ_2 = S_2 - \sum_{\substack{sB \in \mathcal{B}_{H_1}\cap \mathcal{B}_{H_2}
  \\ s'sB \in \mathcal{B}_{H_1}\cap \mathcal{B}_{H_2}\\s'sB\neq sB}} f(s',s),$
where $S_2:=\sum_{\substack{sB\in \mathcal{B}_{H_2} \\ s'sB\in \mathcal{B}_{H_1}\cap \mathcal{B}_{H_2}\\s'sB\neq sB}}  f(s',s).$
\vspace{20pt}

Again, the second terms of the final expressions for $Σ_1,Σ_2$ are the same, and, so, it suffices to show that $S_1v=S_2v$. We have:

$$S_1v=\sum_{\substack{sB\in \mathcal{B}_{H_1} \\ s'sB \in \mathcal{B}_{H_1}\cap \mathcal{B}_{H_2} \\s'sB\neq sB}} \frac{μ_{s'}μ_s}{\lvert B\backslash sB \rvert \lvert sB\backslash s'sB \rvert}s's v =\sum_{\substack{sB\in \mathcal{B}_{H_1} \\ s'sB \in \mathcal{B}_{H_1}\cap \mathcal{B}_{H_2} \\s'sB\neq sB}} \frac{μ_{s'}s'}{\lvert sB\backslash s'sB\rvert}\cdot \frac{μ_ss}{\lvert B\backslash sB \rvert} v$$

Now, we regroup the reflections $s$ in the last sum with respect to the set $sB$, introducing the variable $b=sB$. We have:

\begin{equation*}
\begin{split}
S_1v&=\sum_{b\in \mathcal{B}_{H_1}}\left( \sum_{\substack{s'b \in \mathcal{B}_{H_1}\cap \mathcal{B}_{H_2}\\ s'b\neq b}} \frac{μ_{s'}s'}{\lvert b\backslash s'b\rvert} \cdot \sum_{sB=b}\frac{μ_ss}{\lvert B\backslash b \rvert}v \right)=\sum_{b\in \mathcal{B}_{H_1}} \left( \sum_{\substack{s'b \in \mathcal{B}_{H_1}\cap \mathcal{B}_{H_2}\\s'b\neq b}} \frac{μ_{s'}s'}{\lvert b \backslash s'b \rvert}\cdot \sum_{sH=H_1, sB=b} μ_ssv\right)  \\
& =\sum_{b\in \mathcal{B}_{H_1}} \sum_{\substack{sH=H_1 \\ sB=b \\ s'sB\in \mathcal{B}_{H_1}\cap \mathcal{B}_{H_2}\\s'b\neq b}}\frac{μ_{s'}s'}{\lvert b \backslash s'b \rvert}μ_ssv=\sum_{\substack{sH=H_1\\ s'sB\in \mathcal{B}_{H_1}\cap \mathcal{B}_{H_2}\\s'sB\neq sB}} \frac{μ_{s'}s'}{\lvert sB \backslash s'sB \rvert} μ_ss v \\
\end{split}
\end{equation*}

Again, we just perform two regroupings of the terms in the last sum with respect to the sets $s'sB=b'$ and $s^{-1}b'=b''$:

\begin{equation*}\label{eq:suma1}
\begin{split}
S_1v& =\sum_{b'\in \mathcal{B}_{H_1}\cap \mathcal{B}_{H_2}}\sum_{\substack{sH=H_1\\ sΒ \neq b'  \\ s'sB=b'}} \frac{μ_{s'}s'}{\lvert sB\backslash b' \rvert} μ_ss v=\sum_{b'\in \mathcal{B}_{H_1}\cap \mathcal{B}_{H_2}} \sum_{b''\in \mathcal{B}}\sum_{\substack{sH=H_1 \\sB \neq b' \\ sb''=b'\\ s'sB=b'}}\frac{μ_{s'}s'}{\lvert sB\backslash b' \rvert} μ_ssv\\
&=\sum_{b'\in \mathcal{B}_{H_1}\cap \mathcal{B}_{H_2}} \sum_{\substack{b''\in \mathcal{B}\\b''\neq B}}\sum_{\substack{sH=H_1 \\ sb''=b'\\ s'sB=b'}}\frac{μ_{s'}s'}{\lvert sB\backslash b' \rvert} μ_ssv.
\end{split}
\end{equation*}

For the passage to the last sum, since in the innermost summation, the third condition is $sb''=b'$, then the second condition, i.e. $sB\neq b'$ becomes $b'\neq B$, which we can take out to the second sum as we did in the last expression. In the same way for $S_2$, we get:

\begin{equation}\label{eq:sumb}
S_2v=\sum_{b'\in \mathcal{B}_{H_1}\cap \mathcal{B}_{H_2}} \sum_{\substack{b''\in \mathcal{B}\\ b''\neq B}}\sum_{\substack{sH=H_2 \\ sb''=b'\\ s'sB=b'}}\frac{μ_{s'}s'}{\lvert sB\backslash b' \rvert} μ_ssv. 
\end{equation}

Now, since in both expressions the first and second summations are the same, it is enough to show that the innermost sums are equal; i.e. that for all $b'\in \mathcal{B}_{H_1}\cap \mathcal{B}_{H_2}$, and $b''\in \mathcal{B}$ with $b'' \neq B$, we have

\begin{equation}\label{eq:lastsumeq}
\begin{split}
\sum_{\substack{sH=H_1\\ sb''=b' \\ s'sB=b'}} \frac{μ_{s'}s'}{\lvert sB \backslash b' \rvert} μ_ss v=\sum_{\substack{sH=H_2\\ sb''=b' \\ s'sB=b'}} \frac{μ_{s'}s'}{\lvert sB \backslash b' \rvert} μ_ss v. 
\end{split}
\end{equation}

Notice first that, on both sides, if the summation is not empty, then $H\in b''$.
Now, we take the left-hand side, and see that

\begin{equation*}
\begin{split}
\sum_{\substack{sH=H_1\\ sb''=b' \\ s'sB=b'}} \frac{1}{\lvert sB \backslash b' \rvert} μ_{s'}s'μ_ss v & =\sum_{\substack{sH=H_1\\ sb''=b' \\ s'sB=b'}} \frac{1}{\lvert sB \backslash b' \rvert} μ_{s'}μ_ss(s^{-1}s's) v =\sum_{\substack{sH=H_1\\ sb''=b' \\ aB=b''}} \frac{1}{\lvert B \backslash b'' \rvert} μ_aμ_s sav \\
&=\sum_{\substack{sH=H_1\\sb''=b'}}μ_ss\sum_{aB=b''} \frac{μ_aa}{\lvert B\backslash b'' \rvert}  v=\sum_{\substack{sH=H_1\\sb''=b'}}μ_ssv_0\\
\end{split}
\end{equation*}
where $v_0=\sum_{aB=b''} \frac{μ_aa}{\lvert B \backslash b'' \rvert }v$.
For the passage to the second line, when $s'$ runs through all reflections that take $sB$ to $b'$, then $a=s^{-1}s's$ runs through all reflections mapping $B$ to $b''$; also, $μ_a=μ_{s'}$, since $a$ and $s'$ are conjugate, and $\lvert sB\backslash b' \rvert = \lvert s^{-1}sB \backslash s^{-1}b'\rvert =\lvert B\backslash b'' \rvert $.

Correspondingly, the right-hand side of Equation \eqref{eq:lastsumeq} equals $\sum_{\substack{sH=H_2\\sb''=b'}}μ_ssv_0.$

Now, notice that $v_0\in V^{b''}$ and recall that $H\in b''$. So, by Lemma \ref{lem:commutativity}, 
$\sum_{\substack{sH=H_1\\sb''=b'}}μ_ssv_0=ε_{H_1}^{(b'')}v_0$ which, by the same lemma equals $ε_{H_2}^{(b'')}v_0=\sum_{\substack{sH=H_2\\sb''=b'}}μ_ssv_0$, since $H\in b''$ is not transverse with $H_1,H_2$. This concludes the proof.
\end{proof}

\begin{proposition}\label{prop:imageconstruction} If $V_0\neq 0$, then the transverse collection $B_0$ is a $(B_0\big| V_0)$-maximal collection, for which $(B_0\big| V_0)_B\cong_{K\Stab(B_0)} V_0$. Furthermore, $(B_0\big| V_0)=(B_0\big| V_0)_{\mathcal{B}}$ (see Definition \ref{def:orbvb}).
\end{proposition}
\begin{proof} The element $e_H$ acts on $(B_0\big| V_0)$ as $ε_H$, and, as we remarked in the beginning of the proof of the previous proposition, $ε_H(B_0\big| V_0)=\oplus_{B\in \mathcal{B}, H\in B}V^{B}$. Now since $e_{H_1},e_{H_2}$ commute for $H_1\pitchfork H_2$, then, for any $B'\in \mathcal{C}$, we have $e_{B'}(B_0\big|V_0)=\cap_{H\in B'}e_H(B_0\big| V_0)=\oplus_{B\in \mathcal{B}, B'\subseteq B} V^{B}$. In particular, for $B\in \mathcal{B}$ we have $e_{B}(B_0\big| V_0)=V^{B}$, and, so $e_{B_0}(B_0\big| V_0)=V^{B_0}=\hat{V_0}\cong_{K\Stab(B_0)}V_0$. Also, for all $B'\in \mathcal{B}$ such that $B_0\subset B'$ we get
$e_{B'}(B_0\big| V_0)=0$, which implies that $B_0$ is $(B_0\big| V_0)$-maximal. Finally, $(B_0\big| V_0)=\oplus_{B\in \mathcal{B}} V^B=\oplus_{B\in \mathcal{B}} (B_0\big|V_0)_B=(B_0\big|V_0)_{\mathcal{B}}$.
\end{proof}

\begin{remark}\label{rem:emptyconstruction} One can see from the above proposition and Remark \ref{rem:emptyadmissibility} that if $B_0$ is the empty collection, this construction turns all $KW$-modules into $\Br^K(W)$-modules where all $e_H, H\in \mathcal{H}$ act as zero.  
\end{remark}

\subsection{Conclusions}

In this subsection, we combine the results of Subsections \ref{sec:structuregeneral} and \ref{sec:construction} and obtain the classification of simple $\Br^K(W)$-modules.  

\begin{proposition}\label{prop:admissibility} Let $B\in \mathcal{C}$ and $V$ a $K\Stab(B)$-module. Then there exists a $\Br^K(W)$-module $V'$ with maximal collection $B$ and such that $V'_B$ is a $K\Stab(B)$-module isomorphic to $V$ if and only if $(B,V)$ is admissible. Furthermore, $(B,V)$ is admissible if and only if $\Rel(B)\hat{V}=0$.
\end{proposition}
\begin{proof}  If $(B,V)$ is admissible, i.e. $\Ann_{KW}(e_B)\hat{V}=0$, then, since $\Rel(B)\subseteq \Ann_{KW}(e_B)$, we get that $\Rel(B)\hat{V}=0$. In that case, the $\Br^K(W)$-module $(B\big|V)$ satisfies the properties of the statement, by Proposition \ref{prop:imageconstruction}. Conversely, suppose that a module $V'$, as in the first statement, exists. Since $V_B=e_BV$, then $\Ann_{KW}(e_B)V_B=0$, and, since $V_B\subseteq V_{\mathcal{B}}=\oplus_{w\in W/\Stab(B)}wV_B$, where $\mathcal{B}$ is the $W$-orbit of $B$ (Proposition \ref{prop:directsum}), then $\Ann_{KW}(e_B)\hat{V}_B=0$. Since $V_B$ and $V$ are isomorphic as $K\Stab(B)$-modules, we have $\Ann_{KW}(e_B)\hat{V}=0$, and, hence $(B,V)$ is admissible. 

For the second part of the statement, since $\Rel(B)\subseteq \Ann_{KW}(e_B)$, then if $(B,V)$ is admissible we have $\Rel(B)\hat{V}=0$. For the converse, if $\Rel(B)\hat{V}=0$, then the module $(B\big| V)$ satisfies the properties of the statement, and hence $(B,V)$ is admissible.
\end{proof}

Proposition \ref{prop:imageconstruction} tells us that the module $(B\big| V)$ satisfies the conditions for most of the results of Subsection \ref{sec:structuregeneral}.

\begin{proposition}\label{prop:irreduconstruction} Let $(B,V)$ be an admissible pair. Then the $\Br^K(W)$-module $(B\big| V)$ is simple (in which case it is, also, absolutely irreducible) if and only if $V$ is a simple $K\Stab(B)$-module.
\end{proposition}
\begin{proof} By Proposition \ref{prop:imageconstruction}, the $\Br^K(W)$-module $(B\big| V)$ satisfies the conditions of Proposition \ref{prop:irreducibility}, and the result follows from the latter proposition.
\end{proof}

%\begin{proposition}\label{prop:isomconstruction} Let $(B,V)$ and $(B',V')$ be admissible pairs. Then, the $\Br^K(W)$-modules $(B\big| V)$ and $(B'\big|V') $ are isomorphic if and only if $(B,V)$ and $(B',V')$ are conjugate. 
%As a consequence, for any pair $(B,V)$, counting $V'$ up to $K\Stab(B')$-isomorphism, there are $\sfrac{\lvert W \rvert }{\lvert \Stab(B) \rvert } $ pairs $(B',V')$ for which $(B\big| V')\cong (B\big| V)$.
%\end{proposition}

%\begin{proof} Since, by Proposition \ref{prop:imageconstruction}, the $\Br^K(W)$-module $(B\big| V)$ satisfies the conditions of Proposition \ref{prop:nonisomorphic} the result follows from the latter proposition.
%\end{proof}

\begin{proposition}\label{lem:constructioncontained} Let $V$ be a $\Br^K(W)$-module, and $\mathcal{B}$ an orbit of $V$-maximal collections. Then, the pair $(B,V_B)$ is admissible and  the $\Br^K(W)$-module $V_{\mathcal{B}}$ is isomorphic to $(B\big| V_B)$, for every $B\in \mathcal{B}$. As a consequence, $V$ contains a module isomorphic to $(B\big| V_B)$.
\end{proposition}
\begin{proof} Proposition \ref{prop:admissibility} implies that the pair $(B,V_B)$ is admissible. Now, both modules $V_{\mathcal{B}}$ and $(B\big| V_B)$ satisfy the conditions of Proposition \ref{prop:nonisomorphic} (the latter by Proposition \ref{prop:imageconstruction}), which implies the result.
\end{proof}

The next theorem may be considered as the central result of this section. When we say that two pairs $(B,V),(B',V')$ are \emph{conjugate}, we mean that there is $w\in W$ such that $wB=B'$ and also $wρ_Vw^{-1}$ is isomorphic to $ρ_{V'}$, where $ρ_V,ρ_{V'}$ are the  representations corresponding to $V,V'$, respectively.

\begin{theorem}[Simple modules of $\Br^K(W)$]\label{th:simplemodules}
The simple $\Br^K(W)$-modules are the modules $(B\big| V)$ where $(B,V)$ is an admissible pair and $V$ a simple $K\Stab(B)$-module; they are, in that case, absolutely irreducible. Two such $\Br^K(W)$-modules $(B\big| V)$ and $(B'\big|V')$ (in fact, not necessarily simple) are isomorphic if and only if $(B,V)$ and $(B',V')$ are conjugate. 
\end{theorem}
\begin{proof} By Proposition \ref{prop:irreducibility}, a $\Br^K(W)$-module $V'$ is simple and, in fact, absolutely irreducible, if and only if $V'=
V'_{\mathcal{B}}$ for an orbit $\mathcal{B}$ of $V'$-maximal collections and $V'_B$ is a simple $K\Stab(B)$-module for some $B\in \mathcal{B}$. Propositions \ref{lem:constructioncontained} and  \ref{prop:irreduconstruction} yield, now, the first part of the theorem. 

For the second statement,  since, by Proposition \ref{prop:imageconstruction}, the $\Br^K(W)$-module $(B\big| V)$, for any admissible pair $(B,V)$ satisfies the conditions of Proposition \ref{prop:nonisomorphic} the result follows from the latter proposition.
\end{proof}

\begin{corollary}\label{cor:countisomconstruction} If $(B,V)$ is an admissible pair, then, counting $V'$ up to $K\Stab(B')$-isomorphism, there are exactly $\sfrac{\lvert W \rvert }{\lvert \Stab(B) \rvert }$ admissible pairs $(B',V')$ for which $(B'\big| V')\cong (B\big| V)$.
\end{corollary}
\begin{proof} This is implied by the second statement of the above theorem. 
\end{proof}

\begin{proposition} Let $B$ be a transverse collection and $V_i, i\in I$ be $K\Stab(B)$-modules. Let $V=\oplus_{i\in I}V_i$ Then, $(B,V)$ is admissible if and only if $(B,V_i)$ is admissible for all $i\in I$, in which case:
\begin{equation}\label{eq:constructiondecomposition}
(B\big| V)\cong \oplus_{i\in I}(B\big| V_i).
\end{equation}
In particular, if $V_i$ is simple for all $i\in I$, then the above sum gives the decomposition of $(B\big| V)$ into simple $\Br^K(W)$-modules.

\end{proposition}\label{prop:decompconstruct}
\begin{proof} Verifying that ``$(B,V)$ is admissible if and only if each $(B,V_i)$ is'', is straighforward using the definition. 
By Proposition \ref{prop:imageconstruction}, $B$ is a $(B\big| V_i)$-maximal collection and $(B\big| V_i)_B\cong V_i$. Again, one can see that $B$ is also $\oplus_{i\in I}(B\big| V_i)$-maximal, and 
$\oplus_{i\in I}(B\big| V_i)_B=e_B(\oplus_{i\in I}(B\big| V_i))=\oplus_{i\in I}e_B(B\big| V_i)\cong \oplus_{i\in I}V_i.$
Since this is the case for $(B\big| V)$ as well, the isomorphism \eqref{eq:constructiondecomposition} is implied by Proposition 
\ref{prop:nonisomorphic}. 
\end{proof}

\section{Semisimplicity}\label{ch:semisimplicity}

In this section, we show that the Brauer-Chen algebra is split-semisimple over
any proper field (Definition \ref{def:propering}). We obtain a lower bound for the dimension of $\Br^K(W)$, given the simple 
modules we constructed in the previous section, after a necessary analysis on
admissibility, which will be extensively used for all that will follow.
Next we find an adequate upper bound for the dimension, and prove that the algebra is semisimple.
Finally, we give a decomposition of the simple modules of $\Br^K(W)$ with respect to 
the Brauer-Chen algebra of a parabolic subgroup.

\subsection{Analysis of admissibility}

For this subsection, we fix a transverse collection $B\in \mathcal{C}$. Recall that for a $K\Stab(B)$-module $V$ we denote by $\hat{V}$ its image in the induced $KW$-module $\Ind_{K\Stab(B)}^{KW}(V)$. Also, $W/\Stab(B)$ will denote a set of left coset representatives of $\Stab(B)$ in $W$. 

\begin{definition} For $w\in W$, let $π_w$ be the projection $KW\to wK\Stab(B)$ with respect to the decomposition $KW=\oplus_{w'\in W/\Stab(B)} w'K\Stab(B)$, i.e. if $x=\sum_{h\in W}λ_{h}h \in KW$, then $π_w(x):=\sum_{h\in w\Stab(B)}λ_{h}h$.
\end{definition}

\begin{remark}\label{rem:projconj} Suppose that $B'\in \mathcal{B}$ and denote  by $π'_w$ the corresponding projection $KW\to wK\Stab(B')$. One can verify that if $w'B=B'$, then $w'π_w(x){(w')}^{-1}=π' _{w'w{(w')}^{-1}}(w'x{(w')}^{-1})$, for all $x\in KW$.
\end{remark}

\begin{lemma}\label{lem:projecting} Let $V$ be a $K\Stab(B)$-module and $x\in KW$. Then $x\hat{V}=0$ is equivalent to $π_w(x)\hat{V}=0$ for all $w\in W$, or to $w^{-1}π_w(x)V=0$ for all $w\in W$ (the last expression makes sense since $w^{-1}π_w(x)\in K\Stab(B)$).
\end{lemma}
\begin{proof} This is immediate from the decomposition $\Ind_{K\Stab(B)}^{KW}(V)=\oplus_{w\in W/\Stab(B)} w\hat{V}.$
\end{proof}

Recall that a pair $(B,V)$ is admissible if $\Ann_{KW}(e_B)\cdot \hat{V}=0$, or equivalently, by Proposition \ref{prop:admissibility}, if $\Rel(B)\hat{V}=0$.
Based on the above lemma, we obtain alternative characterisations of admissibility. We group them all together in Proposition \ref{prop:admequiv}, after establishing some terminology.

\begin{definition}\label{def:annihilators} We define the following sets:
\begin{equation*}
\begin{split}
\Ann_{KW}(B):=\sum_{w\in W}π_w(\Ann_{KW}(e_B)),& \hspace{20pt} \overline\Rel(B):=\cup_{w\in W}π_w(\Rel(B)),\\
\end{split}
\end{equation*}

$$\Ann_{K\Stab(B)}(B):=\sum_{w\in W}w^{-1}π_w(\Ann_{KW}(e_B))$$

\end{definition}

\begin{remark}In view of Remark \ref{rem:projconj},
the properties $w\Ann_{KW}(e_B)w^{-1}=\Ann_{KW}(e_{wB})$ and $w\Rel(B)w^{-1}=\Rel(wB)$ imply that conjugation by $W$ on the above sets commutes with the action of $W$ on $\mathcal{C}$, i.e., for example, $w\Ann_{KW}(B)w^{-1}=\Ann_{KW}(wB)$.
\end{remark}

\begin{lemma}\label{lem:wproj} For all $w,w'\in W, s\in \Stab(B)$ and $x\in KW$, we have $π_{w}(w'xs)=w'π_{{(w')}^{-1}w}(x)s$.
\end{lemma}
\begin{proof} Let $x=\sum_{h\in W}λ_hh\in KW$ and $w,w'\in W, s \in \Stab(B)$. We have 
\begin{equation*}
\begin{split}
π_{w}(w'xs)&=
π_{w}\left(\sum_{h\in W}λ_hwhs\right)=\sum_{w'hs\in w\Stab(B)}λ_hw'hs =\sum_{h\in {(w')}^{-1}w\Stab(B)}λ_hw'hs=w'π_{{(w')}^{-1}w}(x)s\\
\end{split}
\end{equation*}
\end{proof}

\begin{proposition}\label{prop:annann} The set $\Ann_{KW}(B)$ is a left $KW$-ideal which is also stable under right multiplicatiοn by $\Stab(B)$. The set $\Ann_{K\Stab(B)}(B)$ is a two-sided $K\Stab(B)$-ideal. Finally, we have
$$\Ann_{KW}(B)=\oplus_{w\in W/\Stab(B)}w\Ann_{K\Stab(B)}(B).$$

\end{proposition}
\begin{proof} The annihilator $\Ann_{KW}(e_B)$ is a left ideal of $KW$ stable under right multiplication by $\Stab(B)$; together with Lemma \ref{lem:wproj} above, this implies the first statement.

 For the second statement, by the same lemma, for $w\in W$ and $s, s'\in \Stab(B)$ we have $$sw^{-1}π_w(\Ann_{KW}(e_B))s'=π_{s}(sw^{-1}\Ann_{KW}(e_B)s').$$ Summing over all $w\in W$, since $sw^{-1}\Ann_{KW}(e_B)s'=\Ann_{KW}(e_B)$, we get that $\Ann_{K\Stab(B)}$ is stable under left and right multiplication with $\Stab(B)$.
 
  Finally, in the same way, for all $w' \in W$ we have $w'w^{-1}π_w(\Ann_{KW}(e_B))=π_{w'}(w'w^{-1}\Ann_{KW}(e_B))$. Again, since $w'w^{-1}\Ann_{KW}(e_B)=\Ann_{KW}(e_B)$, summing over all $w,w'\in W$ yields that $$W\cdot \Ann_{K\Stab(B)}(B)=\Ann_{KW}(B).$$ Since $\Ann_{K\Stab(B)}(B)$ lies inside $K\Stab(B)$, this implies that  $$\Ann_{KW}(B)=\oplus_{w\in W/\Stab(B)}w\Ann_{K\Stab(B)}(B).$$
\end{proof}

\begin{proposition}\label{prop:admequiv} Let $V$ be a $K\Stab(B)$-module.
The following statements are equivalent:
\begin{itemize}
\item[(1)] $(B,V)$ is admissible, i.e., $\Ann_{KW}(e_B)\hat{V}=0$,
\item[(2)] $\Ann_{KW}(B) \hat{V}=0,$
\item[(3)] $\Ann_{K\Stab(B)}(B) V=0$.
\item[(4)] $\Rel(B) \hat{V}=0$,
\item[(5)] $\overline\Rel(B) \hat{V}=0$.
\end{itemize}
\end{proposition}

\begin{proof} As mentioned earlier, the equivalence of assertions (1) and (4) is given by Proposition \ref{prop:admissibility}. Lemma \ref{lem:projecting} yields the equivalences $(1)\Leftrightarrow (2) \Leftrightarrow (3)$ and $(4) \Leftrightarrow (5) $.
\end{proof}

\begin{corollary}\label{cor:nonadm} The following statements are equivalent:
\begin{enumerate}
\item $B$ is admissible,
\item $\Ann_{K\Stab(B)}(B)\neq K\Stab(B),$
\item $\Ann_{KW}(B)\neq KW$
\end{enumerate} 
\end{corollary}
\begin{proof} The equivalence of (2) and (3) is immediate from the definitions of $\Ann_{K\Stab(B)}(B)$ and $\Ann_{KW}(B)$ and (1)$\Leftrightarrow$ (2) is implied by the equivalence (1)$\Leftrightarrow$ (3) of Proposition \ref{prop:admequiv}.
\end{proof}

Τhe following proposition provides an explicit relation of $\Ann_{KW}(B)$ and $\Ann_{K\Stab(B)}(B)$ with the sets $\Rel(B), \overline\Rel(B)$. 

\begin{proposition}\label{prop:relann} Let $I$ denote the two-sided ideal of $K\Stab(B)$ generated by the union of the sets $w^{-1}π_w(\Rel(B)), w\in W$ (for this proposition only). Then,
\begin{enumerate} 
\item $\Ann_{K\Stab(B)}(B)=I$,
\item $\Ann_{K\Stab(B)}(B)=π_1(KW\cdot \Rel(B))$,
\item $\Ann_{KW}(B)=KW\cdot\overline\Rel(B)$.
\end{enumerate}
\end{proposition}

\begin{proof} By Proposition \ref{prop:admequiv}, for any $K\Stab(B)$-module $V$, the condition $\Rel(B)\hat{V}=0$ is equivalent to $π_w(\Rel(B))\cdot \hat{V}=0,$ for all $w\in W$. The latter is in turn equivalent to $w^{-1}π_w(\Rel(B))\cdot V=0,$ for all $w\in W$ or, by definition of $I$, to $IV=0$. Since, by $(4)\Leftrightarrow (3)$ of the same proposition, this is equivalent to $\Ann_{K\Stab(B)}(B)V=0$ and $I,\Ann_{K\Stab(B)}(B)$ are both two-sided ideals of $K\Stab(B)$, considering their corresponding quotients as the $K\Stab(B)$-module $V$, we get (1).

For (2), by (1) we get that $\Ann_{K\Stab(B)}(B)$ is equal to 
$$K\cdot \sum_{s,s'\in \Stab(B),w\in W}sw^{-1}π_w(\Rel(B))s'.$$
By Lemma \ref{lem:wproj} we can write this as
$$K\cdot \sum_{s,s'\in \Stab(B),w\in W}π_s(sw^{-1}\Rel(B)s').$$
Now, for $s\in \Stab(B)$ we have $π_s=π_1$, and, since $\Rel(B)$ is stable under conjugation by $\Stab(B)$, we can write $sw^{-1}\Rel(B)s'=sw^{-1}s'\Rel(B)$. In the above sum, the element $sw^{-1}s'$ runs through all elements of $W$, which yields (2).

Finally, Proposition \ref{prop:annann} implies that $\Ann_{KW}(B)=W \Ann_{K\Stab(B)}(B)$. By (2), we have that $\Ann_{K\Stab(B)}(B)=\sum_{w,w'\in W}wπ_1(Kw'\Rel(B))$. By Lemma \ref{lem:wproj} again, this is 
$$\sum_{w,w'\in W}ww'π_{{(w')}^{-1}}(K\Rel(B))$$ which one can see that equals $KW\cdot\overline\Rel(B)$.
\end{proof}

%\begin{corollary}\label{cor:admcardone} If $B$ consists of a single hyperplane $H$, then the ideal $\Ann_{K\Stab(B)}(B)$ coincides with the augmentation ideal $\Aug(W_H)$ of $K\Stab(B)$ of the pointwise stabilizer $W_H$ of $H$. In particular, $B$ is admissible.
%\end{corollary}
%\begin{proof} By Proposition \ref{prop:relann} above, we have  $\Ann_{K\Stab(B)}(B)=π_1(KW\cdot \Rel(B))$. Going back to Definition \ref{def:rel}, one can see that if $B=\{H\}$, then $Σ_B^{H'}=\emptyset$ or $\{0\}$, for all $H'\in \mathcal{H}$. This implies that $\Rel(B)=(\mathcal{R}_B-1)\cup \{0\}$ and hence, $\Ann_{K\Stab(B)}(B)=π_1(KW\cdot (\mathcal{R}_B-1))$. Now, since $\mathcal{R}_B-1$ lies inside $K\Stab(B)$,  $π_1(KW\cdot (\mathcal{R}_B-1))$ is equal to the ideal of $K\Stab(B)$ generated by $\mathcal{R}_B-1$. Furthermore, since $B=\{H\}$, $\mathcal{R}_B$ is the set of all reflections of $W$ with reflecting hyperplane $H$, which generates the subgroup $W_H$. This yields the result.  
%\end{proof}

\subsection{Lower bound for the dimension}
\label{sec:lowerbound}

 Recall that $\mathcal{C}$ denotes the set of transverse collections and $\mathcal{C}^K_{adm}$ the subset of admissible ones. Also the notation $\Irr(K\Stab(B))$ will be used for the set of simple modules of $K\Stab(B)$.

By Theorem \ref{th:simplemodules}, the simple modules of $\Br^K(W)$ are the modules $(B\big|V)$, where $(B,V)$ is an admissible pair, which are, in fact, absolutely irreducible. By Corollary \ref{cor:countisomconstruction}, there are exactly $\lvert W \rvert / \lvert \Stab(B) \rvert $ such pairs for every isomorphism class of $\Br^K(W)$-modules. Hence, the sum of squares of dimensions of the simple $\Br^K(W)$-modules equals:

\begin{equation}
\label{eq:lboundgeneral}
\sum_{\substack{B\in \mathcal{C},V \in \Irr(K\Stab(B))\\ (B,V)\textnormal{ adm. }}} \frac{\lvert \Stab(B) \rvert}{\lvert W\rvert}(\dim_K (B\big|V))^2.
\end{equation}
Recall that the module $(B\big|V)$ is $KW$-isomorphic to $\Ind_{K\Stab(B)}^{KW}(V)$. Hence, its dimension equals $\faktor{\lvert W \rvert }{ \lvert \Stab(B) \rvert }\cdot \dim_K V$.
Consequently, for any $B\in \mathcal{C}$ we have,

\begin{equation*}
\begin{split}
\sum_{\substack{ V \in \Irr(K\Stab(B))\\ (B,V)\textnormal{ adm. }}} (\dim_K (B\big|V))^2
&=\left(\frac{\lvert W \rvert}{\lvert \Stab(B) \rvert}\right)^2 \sum_{\substack{ V \in \Irr(K\Stab(B))\\ (B,V)\textnormal{ adm. }}} (\dim_K V)^2 
\end{split}
\end{equation*}

Now, by characterisation (3) of admissibility of Proposition \ref{prop:admequiv}, a pair $(B,V)$ is admissible if and only if the $K\Stab(B)$-module $V$ factors through the quotient algebra $K\Stab(B)/\Ann_{K\Stab(B)}(B)$. Since the latter, as a quotient of $K\Stab(B)$, is split-semisimple over $K$ (see Definition \ref{def:propering}), then its dimension equals the sum of squares of dimensions of its simple modules. Combining these two properties, for any $B\in \mathcal{C}$, we have: 

\begin{equation*}
\begin{split}
\sum_{\substack{V\in \Irr(K\Stab(B))\\(B,V)\textnormal{ adm.}}} (\dim_K V)^2 &= \dim_K\left(\faktor{K\Stab(B)}{\Ann_{K\Stab(B)}(B)}\right)=\frac{\lvert \Stab(B)\rvert}{\lvert W \rvert}\cdot \dim_K \left(\faktor{KW}{\Ann_{KW} (B)}\right).\\
\end{split}
\end{equation*}
This implies that 
\begin{equation*}
\begin{split}
\sum_{\substack{V\in \Irr(K\Stab(B))\\(B,V)\textnormal{ adm.}}} (\dim (B\big|V))^2
&=\frac{\lvert W \rvert}{\lvert \Stab(B)\rvert}\cdot \dim_K\left(\faktor{KW}{\Ann_{KW} B}\right).\\
\end{split}
\end{equation*}
Replacing this in the sum (\ref{eq:lboundgeneral}), we obtain the following result.

\begin{lemma}\label{lem:lowerbound} The sum of squares of dimensions of all absolutely irreducible $\Br^K(W)$-modules equals:
$$\sum_{B\in \mathcal{C}} \dim_K\left(\faktor{KW}{\Ann_{KW} B}\right).$$

\end{lemma}

\subsection{Upper bound and semisimplicity}\label{sec:semisimplicity}

We prove now that the Brauer-Chen algebra is semisimple by bounding its dimension from above. The following lemma is a slight improvement of  \cite[Lemma 5.2]{Ch} which can be seen to be true with essentially the same proof. In any case, we give a quick proof here as well. 

\begin{lemma} \label{lem:genset} The elements $we_B$, where $w\in W$ and  $B\in \mathcal{C}$, span $\Br(W)$ as an $\mathcal{A}$-module.
\end{lemma}
\begin{proof} Remark \ref{remark:ebmult}, giving a formula for the product $e_He_B$, implies that the $\mathcal{A}$-span of the elements in the statement, is, in fact, a left ideal of $\Br(W)$. Since they contain $1$, the result follows. 
\end{proof}

\begin{definition}\label{def:ideals} Let $R$ be an $\mathcal{A}$-algebra. For $r \in \mathbb{N}$, we define $I_r^{(R)}:=\sum_{B\in \mathcal{C}, \lvert B \rvert \geq r} RWe_B\subseteq \Br^R(W)$. If $R=\mathcal{A}$ we omit it from the notation and write $I_r$. 
\end{definition}

\begin{proposition} Let $R$ be an $\mathcal{A}$-algebra. The set $I_r^{(R)}$ is a two-sided ideal of $\Br^R(W)$.
\end{proposition}
\begin{proof} By definition, $RW\cdot I_r^{(R)}=I_r^{(R)}$. Since $we_B=e_{wB}w$ for any $w\in W$ and $B\in \mathcal{C}$, then $I_r^{(R)}KW=I_r^{(R)}$ and in fact, $I_r^{(R)}=\sum_{B\in \mathcal{C}, \lvert B \rvert \geq r} e_BRW$. Combining the latter with the formula for the product $e_He_B$ of Remark \ref{remark:ebmult}, one can see that $e_HI_r^{(R)}\subseteq I_r^{(R)}$ for all $H\in \mathcal{H}$. In the same way, using the formula for $e_Be_H$ of the same remark, we get that $I_r^{(R)}e_H \subseteq I_r^{(R)}$ as well.
\end{proof}

The ideals $I_r$ have already been introduced in \cite{Ma} , and in the same paper the quotient algebra $\Br(W)/ I_{r+1}$ is denoted $\Br_r(W)$. The following lemma will provide our setting for all situations where we want to suitably restrict the spanning set of Lemma \ref{lem:genset}.

\begin{lemma}\label{lem:unionideals} Let $R$ be an $\mathcal{A}$-algebra. The ideals $I^{(R)}_r, r\in \mathbb{N}$ form a descending chain of ideals of $\Br^R(W)$, such that $\Br^R(W)=\cup_{r\in \mathbb{N}}I^{(R)}_r$.
\end{lemma}
\begin{proof} It is clear by the definition that $I^{(R)}_{r+1}\subseteq I^{(R)}_{r}$ for all $r\in \mathbb{N}$. Furthermore,  Lemma \ref{lem:genset} is equivalent to the property $\Br^R(W)=\cup_{r\in \mathbb{N}}I^{(R)}_r$.
\end{proof}

\begin{lemma}\label{lem:upperbound}
For any $B\in \mathcal{C}$, we have $\Ann_{KW} (B) e_B \subseteq I_{\lvert B \rvert+1}^{(K)}$.
\end{lemma}

\begin{proof}

Let $r=\lvert B \rvert$.
If $e_B \in I_{r+1}^{(K)}$, then the result is immediate, since $ΚW I_{r+1}^{(K)}\subseteq I_{r+1}^{(K)}$. So, we suppose that $e_B \not \in I_{r+1}^{(K)}$. Consider the quotient $\Br_r^K(W)$, and let $a\in KW$. Of course, $ae_B=0$ in $\Br_r^K(W)$ is equivalent to $ae_B\in I_{r+1}^{(K)}$. We will show that in $\Br_r^K(W),  ae_B=0$ implies  $π_{w}(a)e_B=0$ for all $w\in W$. Hence, since $\Ann_{KW}(B)=\sum_{w\in W}π_w(\Ann_{KW}(e_B))$, we get the statement.

Let $\mathcal{B}$ denote the orbit of $B$, and let $V=\sum_{B'\in  \mathcal{B}} e_{B'}(KW + I_{r+1}^{(K)}) \subseteq \Br_r^K(W)$. Relation (B2) of Definition \ref{def:algebra} or Remark \ref{remark:WactiononB} imply that $wV=V$ for $w\in W$. Furthermore, if $H$ is a hyperplane and $B'\in \mathcal{B}$, Remark \ref{remark:ebmult} yields that $e_He_{B'}$ either belongs to $KWe_{B'}\subseteq V$, or is equal to some $e_{B''}$, where $B''$ strictly contains $B'$, in which case $e_{B''}\in I_{r+1}^{(K)}$. So, $e_HV\subseteq V$ for all hyperplanes $H$, as well, and hence, $V$ is a $\Br^K(W)$-module.

Now, it is clear that $(e_{B} + I_{r+1}^{(K)})KW \subseteq e_{B}V$ since $e_{B} (e_{B}x+I_{r+1}^{(K)})=δ^r (e_{B}x + I_{r+1}^{(K)})$ for any $x\in KW$. In particular, $e_BV\neq 0$. Furthermore, $B$ is $V$-maximal since if $B'$ strictly contains $B$, then $e_{B'}\in I_{r+1}^{(K)}$, and so, $e_{B'}V=0$. Now Proposition \ref{prop:directsum} yields that $\sum_{B'\in \mathcal{B}} V_{B'}=\oplus_{B'\in \mathcal{B}} V_{B'}$. Since $wV_B=V_{wB}$ for all $w\in W$, this implies that for every $a\in KW$ and $v\in V_B, av=0$ implies $π_w(a)v=0$. So, $ae_B=0$ in $\Br^K_r(W)$ implies $ π_w(a)e_B\in I_{r+1}^{(K)}$ for all $a\in KW$ and $w\in W$ and the result follows.
\end{proof}

\begin{theorem}\label{th:semisimplicity} The algebra $\Br^K(W)$ is split-semisimple, with dimension

$$\sum_{B\in \mathcal{C}} \dim_K \left(\faktor{KW}{\Ann_{KW}(B)}\right) \textnormal{ or, equally, }\sum_{B\in \mathcal{C}^K_{adm}} \dim_K \left(\faktor{KW}{\Ann_{KW}(B)}\right). $$

\end{theorem}

\begin{proof}
By Lemma \ref{lem:unionideals}, we have that $\Br^K(W)=\cup_{r\in \mathbb{N}}I^{(K)}_r$. We write $I^{(K)}_r$ as $\sum_{\lvert B \rvert =r} KWe_B + I^{(K)}_{r+1}$. By Lemma \ref{lem:upperbound} above, we have $\Ann_{KW}(B)e_B\subseteq I_{r+1}^{(K)}$ for all $B$ of cardinality $r$, which implies that 

$$\dim_K I^{(K)}_r \leq \sum_{\lvert B \rvert =r}\dim_K\left(\faktor{KW}{\Ann_{KW} (B)}\right) + \dim_K I^{(K)}_{r+1}.$$
By induction on $r$, this yields

$$\dim_K\Br^K(W)\leq \sum_{B\in \mathcal{C}} \dim_K\left(\faktor{KW}{\Ann_{KW} (B)}\right).$$
The above sum is equal to the the sum of squares of dimensions of simple $\Br^K(W)$-modules (Lemma \ref{lem:lowerbound}), which yields the main result. Equality with the second sum of the statement is implied by Corollary \ref{cor:nonadm}, which states that if $B$ is not admissible, then $\Ann_{KW}(B)=KW$.
\end{proof}

\begin{corollary}\label{cor:dimension} The dimension of $\Br^K(W)$ is equal to:

\begin{equation}\label{eq:dimenalt}
\sum_{B\in \mathcal{C}^K_{adm}}\frac{\lvert W \rvert }{\lvert \Stab(B) \rvert }\cdot \dim_K\left(\faktor{K\Stab(B)}{\Ann_{K\Stab(B)}(B)}\right).
\end{equation}
\end{corollary}
\begin{proof} The statement is implied by Theorem \ref{th:semisimplicity}, since 
$\Ann_{KW}(B)=\oplus_{w\in W/\Stab(B)}w\Ann_{K\Stab(B)}(B)$, and, hence, 
$$\dim_K \left(\faktor{KW}{\Ann_{KW}(B)}\right)=\frac{\lvert W \rvert }{\lvert \Stab(B) \rvert }\dim_K\left(\faktor{K\Stab(B)}{\Ann_{K\Stab(B)}(B)}\right).$$
\end{proof}

\begin{remark} In the case of the empty collection $B$, one can verify that $\Ann_{K\Stab(B)}(B)=0$. This, as expected (see also Remark \ref{rem:emptyconstruction}), implies that the part of \eqref{eq:dimenalt} in the above corollary corresponding to $B$ equals 
the order of $W$. 
\end{remark}

\section{Admissibility and a basis for $\Br^K(W)$}\label{ch:admissibility}

We apply now the established results of the previous sections to obtain, for all irreducible complex reflection groups, the admissible
collections and corresponding admissible pairs. Following the results of the previous section, for a transverse collection $B$, these elements are determined by the ideal $\Ann_{K\Stab(B)}(B)$ of $K\Stab(B)$. 

Case-by-case analysis shows that for every admissible transverse collection $B$, there is
a certain subgroup $K_B$ of $\Stab(B)$ such that $\Ann_{K\Stab(B)}(B)$ is
some kind of modified augmentation ideal of this subgroup. In fact, for all but two groups (and specific proper fields of definition), it is equal to the augmentation ideal of $K_B$. This translates, equivalently, into the fact that the admissible
$K\Stab(B)$-modules $V$ are those that induce a certain representation on $K_B$, which,
again, for most cases, is the trivial one. 

\subsection{Overview of this section's results}\label{sec:overviewadmissibility}

We sum up, in this first subsection, the results that we obtain throughout the rest of this section by the case-by-case study of all irreducible complex reflection groups, and provide some general conclusions as well.
  
For the next definition, recall that for $B\in \mathcal{C}$, $\mathcal{R}_B=\{r\in \mathcal{R}\big| H_r\in B\}$.

\begin{definition}\label{def:KB} Let $W$ be a complex reflection group and $B$ a transverse collection. We denote by $K_B$ the subgroup of $\Stab(B)$ generated by $\mathcal{R}_B$ together with all products $s_2^{-1}s_1$ where $s_1,s_2\in \mathcal{R}$ satisfy: 
\begin{itemize}
\item[(c1)] $s_1B=s_2B\neq B$ (in particular, $s_2^{-1}s_1\in \Stab(B)$),
\item[(c2)] $s_1H\neq s_2H$, for all $H\in B\backslash s_1B$.
\end{itemize}
If $B$ is empty, we define $K_B=1$.
\end{definition}

\begin{remark} One can verify that for $w\in W$ and $B\in \mathcal{B}$ we have $wK_Bw^{-1}=K_{wB}$. In particular, $K_B$ is a normal subgroup of $\Stab(B)$.
\end{remark}

\begin{remark}\label{rem:KBforcardone} If $B$ is a transverse collection of cardinality $1$, i.e., $B$ consists of a single hyperplane $H$, the above definition of $K_B$ yields the subgroup of $W$ generated by $\mathcal{R}_B$, which coincides with the pointwise stabilizer $W_H$ of $H$. 
\end{remark}

For a function $χ: K_B\to K$, we denote by $\Aug^χ(K_B)$ the two-sided ideal
of $K\Stab(B)$ generated by all $h-χ(h), h\in K_B$. Note that if $χ=1$, this is the 
usual augmentation ideal $\Aug(K_B)$. The following proposition is proved throughout the course of this section.

\begin{proposition}\label{prop:KBfunction} Let $W$ be an irreducible complex reflection group and $K$ a proper
field. Then, if $B\in \mathcal{C}^K_{adm}$, there is a group homomorphism $χ:K_B\to \mathbb{U}_6 \cap K$
such that $\Ann_{K\Stab(B)}(B)=\Aug^χ(K_B)$. Furthermore, if $W\neq G_{25},G_{32}$, then $χ=1$. If $W=G_{25}$ or $G_{32}$, then $\dim_K\Aug^χ(K_B)=\dim_K \Aug(K_B)$. 
\end{proposition}

A consequence of this proposition is the following result, providing a far more efficient formula for the dimension of $\Br^K(W)$ than the one of Theorem \ref{th:semisimplicity}, as well as a basis over $K$. Let $W/K_B$ denote a set of left coset representatives of $K_B$ in $W$.

\begin{theorem}\label{th:basisoverK} Let $W$ be an irreducible complex reflection group and $K$ a proper field.
Then

$$\dim_K \Br^K(W)=\sum_{B\in \mathcal{C}^K_{adm}} \faktor{\lvert W \rvert }{\lvert K_B \rvert},$$
and the set $\{we_B\big| B\in \mathcal{C}^K_{adm}, w\in W/K_B\}$ is a basis for $\Br^K(W)$.
\end{theorem}

\begin{proof} Recall the following formula for the dimension of $\Br^K(W)$ of Corollary \ref{cor:dimension}: 
$$\dim_K \Br^K(W)=\sum_{B\in \mathcal{C}^K_{adm}} \frac{\lvert W \rvert }{\lvert \Stab(B) \rvert}\cdot \left(\dim_K \faktor{K\Stab(B)}{\Ann_{K\Stab(B)}(B)} \right) $$
By Proposition \ref{prop:KBfunction}, the dimension of  $K\Stab(B)/\Ann_{K\Stab(B)}(B)$ is $\dim_K K\Stab(B) - \dim_K \Aug(K_B)$
which, in turn, equals $\lvert \Stab(B) \rvert /\lvert K_B \rvert$. Replacing this in the above formula, gives the dimension 
of the statement.

Recall, now, the setting of Lemma \ref{lem:unionideals}, i.e. $\Br^K(W)$ is the union of the descending chain of ideals $I^{(K)}_r=\sum_{B\in \mathcal{C}, \lvert B \rvert \geq r} KWe_B$ of $\Br^K(W)$ (Definition \ref{def:ideals}). We show that, for all $r\in \mathbb{N}$, we have 
 
\begin{equation}\label{eq:indstep}
I^{(K)}_{r}=\sum_{B\in \mathcal{C}^K_{adm}, w\in W/K_B} Kwe_B + I^{(K)}_{r+1}.
\end{equation}
Induction on $r$ then yields that the set $\{we_B\big| B\in \mathcal{C}^K_{adm}, w\in W/K_B\}$ spans $\Br^K(W)$. Since its cardinality is equal to $\dim_K \Br^K(W)$ as this is given above, it is a basis.
 
 To show Equation \eqref{eq:indstep}, recall that, by Lemma \ref{lem:upperbound}, for every $B\in \mathcal{C}$ of cardinality $r$ we have
$\Ann_{KW}(B)e_B\subseteq I^{(K)}_{r+1}$.

If $B\not \in\mathcal{C}^K_{adm}$, then Corollary \ref{cor:nonadm} implies that $\Ann_{KW}(B)=KW$, and, so, $KWe_B\subseteq I^K_{r+1}$. This already yields that 
$I^{(K)}_r=\sum_{B\in \mathcal{C}^K_{adm}} KWe_B + I^{(K)}_{r+1}$.

If $B\in \mathcal{C}^K_{adm}$, then, by Proposition \ref{prop:KBfunction}, we have $\Ann_{K\Stab(B)}(B)=\Aug^χ(K_B)$ for some 
$χ: K_B\to K$. Hence, since $\Ann_{K\Stab(B)}(B)\subseteq \Ann_{KW}(B)$ (Definition \ref{def:annihilators}), then $\Ann_{KW}(B)e_B\subseteq I^{(K)}_{r+1}$ implies that
$(h-χ(h))e_B\in I^{(K)}_{r+1}$ for all $h\in K_B$. Given this, if $w_1,w_2$ belong to the same coset of $K_B$, i.e. $w_2=w_1h$, for
some $h\in K_B$, then writing $w_2e_B=χ(h)w_1e_B+w_1(h-χ(h))e_B$ implies that $w_2e_B \in Kw_1e_B + I^{(K)}_{r+1}$, which finally yields Equation \eqref{eq:indstep}, and concludes the proof.

\end{proof}

\subsection{The infinite series}\label{subsec:Gmpn}

This subsection contains an analysis of the admissible collections and admissible pairs for the complex reflection groups in the infinite series. The reader may want to consult Subsection \ref{subsec:infiniteseries} for an overview of the necessary elements we will be using. In all statements, we will assume that $K$ is a proper field for the group in mention. 

We restate here Lemma \ref{lem:Gmpn0}.

\begin{lemma}\label{lem:Gmpn}
In $G(m,p,n)$,
\begin{enumerate}
\item if $i\neq j$, then $H_i$ and $H_j$ are non-transverse and the reflections that map one to the other (in any order) are the reflections $(ij)_κ, κ\in \mathbb{Z}$;

\item $H_{i'}$ and $H_{ij}^κ$ are transverse if and only if $i,j  \neq i'$; in any case, there exist no reflections mapping one to the other;

\item if $j_1\neq j_2$, then $H_{ij_1}^{κ_1}$ and $H_{ij_2}^{κ_2}$ are non-transverse and the only reflection mapping one to the other (in any order) is $(j_1j_2)_{κ_2-κ_1}$;

\item if $i_1,j_1 \neq i_2,j_2$, then $H_{i_1j_1}^{κ_1}$ and $H_{i_2j_2}^{κ_2}$ are transverse;

\item for $(m,p)\neq (2,2)$, if  $i\neq j$, then $H_{ij}^{κ_1}$ and $H_{ij}^{κ_2}$ are not transverse and the reflections mapping the former to the latter are  $(ij)_κ$, for $2κ\equiv κ_1+κ_2 \mod m$, and $t_i^{κ_1-κ_2}, t_j^{κ_2-κ_1}$, provided that $κ_2-κ_1 \equiv 0 \mod p$. 

\item if $(m,p)=(2,2)$, then $H_{ij}$ and $H_{ij}^{1}$ are transverse.  

\end{enumerate}
\end{lemma}

\begin{remark}
\label{remark:sumtransversality}
A quick way to summarize transversality according to this lemma is the following. In $G(m,p,n)$ for $(m,p)\neq (2,2)$, two hyperplanes expressed as above (notation of Subsection \ref{subsec:infiniteseries}) are not transverse if and only if they have at least one index in common; in $G(2,2,n)$ two hyperplanes are non-transverse if and only if they have exactly one index in common.
\end{remark}

\begin{remark}
\label{remark:conjugacyclasses}
From the above lemma it is also clear that, for $n>2$, all reflections $(ij)_κ$ belong to the same conjugacy class of $G(m,p,n)$ and so they share a common parameter in the definition of the Brauer-Chen algebra. We will denote this parameter by $μ$. 
\end{remark}

All mentioned transverse collections in this subsection are assumed to be non-empty. One can verify that if $B$ is the empty collection, then $\Ann_{K\Stab(B)}(B)=\Aug_{K\Stab(B)}(K_B)=0$.

\begin{lemma}\label{lem:hiproblem} Let $W=G(m,p,n)$, and $B\in \mathcal{C}$ be such that $\lvert B \rvert >1$ and $H_i\in B$ for some $i=1,\dots, n$. Then, there is $s\in \mathcal{R}$ such that $μ_ss\in \Rel(B)$. In particular, $B\not\in \mathcal{C}^K_{adm}$.
\end{lemma} 
\begin{proof} Since $Η_i, H_j$ are non-transverse for all $i,j$ (Lemma \ref{lem:Gmpn}), then $B$ can contain at most one such hyperplane. Let $α=H_i\in B$ and choose a second hyperplane in $B$, say $β$ which must be of the form $H_{i'j}^κ$. Let $γ=H_{i'i}$ ; then $γ\not \pitchfork α,β$, and so $σ^γ_{β,α}\in \Rel(B).$  Again, from Lemma \ref{lem:Gmpn}, we get that $σ^γ_{β,α}=μ(ij)_κ$, which implies the statement.
\end{proof}

\begin{proposition}\label{prop:onehi} Let $W=G(m,p,n)$ and $B=\{H_i\}$ for some $i$. Then, $K_B$ is the pointwise stabilizer $W_{H_i}$ of $H_i$, and has order $m/p$. Moreover, we have $\mathcal{A}W\Rel(B)=\mathcal{A}W(K_B-1)$.
\end{proposition}
\begin{proof} 
One may check that the definition of $K_B$ gives the pointwise stabilizer of $H_i$, i.e. the subgroup generated by $\mathcal{R}_B=\{ t_i \}$, which has order $m/p$.  Furthermore, since $B$ consists of a single hyperplane, then $Σ^ H_{B}=\{0\}$, for all $H\in \mathcal{B}$, and, hence $\Rel(B)=\mathcal{R}_B-1$, which gives the result.
\end{proof}

Lemma \ref{lem:hiproblem} implies that an admissible transverse collection with more than one hyperplane must contain only hyperplanes of the form $H_{ij}^κ$. For the groups $G(m,p,n)$ with $(m,p)\neq (2,2)$, by Lemma \ref{lem:Gmpn}, such a collection is transverse if and only if every index $i$ and $j$ as above appears in at most one hyperplane in $B$. The following proposition shows that all such transverse collections are admissible.

\begin{proposition}\label{prop:classgeneral} Let $W=G(m,p,n)$ and $B$ a transverse collection of the form $\{H_{i_1j_1}^{κ_1},\dots, H_{i_rj_r}^{κ_r} \}$, with all indices distinct. Then $\mathcal{A}W\Rel(B)=\mathcal{A}W(K_B-1)$, and, in particular, $B\in \mathcal{C}^K_{adm}$. Finally, the order of $K_B$ is 
$2^rm^{r-1}r !$.
 
\end{proposition}
\begin{proof}Suppose, first,that $B=\{H_{i_1j_1},H_{i_2j_2},\dots ,H_{i_rj_r}\}$; we show at the end of the proof that the general case is implied under a suitable conjugation. 
Recall that $\Rel(B)=(\mathcal{R}_B-1) \cup (\cup_{γ\in \mathcal{H}}Σ^γ_B)$, where 
$Σ^γ_B=\{σ^γ_{α,β}\big|α,β \in B, γ\not\pitchfork α,β \}$ (Definition \ref{def:rel}).

To prove the statement, we show that for every $\mathcal{A}W$-module $V$ and $v\in V$, $\Rel(B)v=0$ is equivalent to $(K_B-1)v=0$; let $V$ be such a module and $v\in V$. The relation $(\mathcal{R}_B-1)v=0$ is equivalent to $((ij)-1)v=0$, for all $H_{ij}\in B$.

We turn, now, to the relations $Σ^γ_Bv=0$.
First, let $γ=H_i$ for some $i$. By Remark \ref{remark:sumtransversality}, two hyperplanes are non-transverse if and only if they share some index. Since all indices that appear in $B$ are distinct, this implies that $γ$ is non-transverse with at most one hyperplane in $B$, which implies that $Σ^γ_B=\{0\}$.

Now, let $γ=H_{ij}^κ$, for some $i,j,κ$. By the same observation, we see that if $γ$ is non transverse with more than one hyperplane in $B$, then the indices $i,j$ appear in two different elements of $B$. That means that there are $α= H_{ij'}, β=H_{ji'}\in B$, and these are the only hyperplanes of $B$ non-transverse with $γ$. Lemma \ref{lem:Gmpn} now gives:
$σ^γ_{α,β}=μ(j'j)_κ-μ(i'i)_{-κ}=μ((j'j)_κ-(ii')_κ),$ so, $σ^γ_{α,β}v=0$ is
$((j'j)_κ-(ii')_κ)\cdot v=0$, or, equivalently $((ii')_κ(j'j)_κ-1)\cdot v=0$. Hence, letting $i,j$ run through all possible distinct indices appearing in $B$, we get that $(\cup_{γ\in \mathcal{H}}Σ^γ_B)\cdot v=0$ is equivalent to 
$((ii')_κ(j'j)_κ-1)\cdot v=0$, for all $H_{ij'},H_{i'j}\in B$, and $κ\in \mathbb{Z}$.

This shows, already, that $\Rel(B)\cdot v=0$ is equivalent to $(K_B^0-1)v=0$, where $K_B^0$ is the group generated by $\{(ij)\big| H_{ij}\in B\} \cup \{(ii')_κ(j'j)_κ\big| H_{ij'},H_{i'j}\in B,κ\in \mathbb{Z}\}$. We show that this subgroup coincides with $K_B$. For that, since the set
$\{(ij)\big| H_{ij}\in B\}=\mathcal{R}_B$ is included in both generating sets, it is enough to show that $\{(ii')_k(j'j)_k\big| H_{ij'},H_{i'j}\in B,κ\in \mathbb{Z}\}$ is equal to the set of $s_2^{-1}s_1$ where $s_1,s_2$ satisfy conditions (c1), (c2) of Definition \ref{def:KB}, i.e.
$s_1B=s_2B\neq B$ and $s_1H\neq s_2H$ for all $H\in B\backslash s_1B$.

For $H_{ij'},H_{i'j}\in B$, and $κ\in \mathbb{Z}$ the reflections $(ii')_κ, (j'j)_κ$ fix all hyperplanes of $B$ but $H_{ij'},H_{i'j}$, which they map to $H_{j'i'}^κ,H_{ij}^{κ}$, and 
$H_{ij}^κ,H_{j'i'}^κ$, respectively, in that order. So, we can see that they verify the conditions (c1), (c2). Hence, $K_B^0\subseteq K_B$.

For the inverse inclusion, let $s_1,s_2\in \mathcal{R}$ satisfy (c1), (c2). We exclude certain cases for $s_1, s_2$. If $s_1$ is of the form $t_i$ for some $i$, then one can see that it fixes all but at most one hyperplane of $B$, in which case conditions (c1), and (c2), cannot be satisfied. This is also the case when $s_1$ is of the form $(ij)_κ$ with at most one of $i,j$ appearing in $B$, as well as with their appearing in the same hyperplane. Hence, $s_1,s_2$ must be of the form $(ij)_{λ_1},(i'j')_{λ_2}$ with $i,j$ and $i',j'$ appearing, respectively, in different hyperplanes of $B$. Now, all hyperplanes not containing $i,j$ will be fixed by $s_1$ and vice versa for $s_2$; conditions (c1), (c2) imply that $s_1,s_2$ fix the same hyperplanes of $B$, and this in turn implies that
$H_{ij'}, H_{i'j}\in B$ or $H_{ii'},H_{jj'}\in B$. Both cases are equivalent, so suppose the second is true.
In that case, since $s_1B=s_2B$, we must have $λ_1=λ_2$, which implies that 
$s_2^{-1}s_1\in K_B^0$. Hence $K_B=K_{B}^0$, and this shows that $\Rel(B)v=0$ is equivalent to $(K_B-1)v=0$. As far as the order of $K^0_B$ is concerned we will calculate it in the appendix of this article, where we will also obtain an alternative characterization in terms of matrices.

Now, for the general case, notice that  
$\{H_{i_1j_1}^{κ_1},\dots ,H_{i_rj_r}^{κ_r}\}$ is the image of $B$ via the element 
$\prod_{i=1}^r \hat{t}_{j_i}^{k_i}\in G(m,1,n)$, where $\hat{t}_{j_i}$ is the reflection with hyperplane $H_{j_i}$ in $G(m,1,n)$ acting as multiplication by $ζ$ of the $j_i$-th coordinate. Now, the action of $G(m,1,n)$ respects transversality, and, also $g \cdot \mathcal{R}_{H\to H'}\cdot g^{-1} =\mathcal{R}_{gH\to gH'}$ for all $g\in G(m,1,n)$ and $H,H'\in \mathcal{H}$. From this, one can see that $g\Rel(B)g^{-1}=\Rel(gB)$, as well as, 
$gK_Bg^{-1}=K_{gB}$, and hence, $\mathcal{A}\Rel(gB)=\mathcal{A}W(K_{gB}-1)$, which concludes the proof.
\end{proof}

The above proposition already concludes the classification of admissible collections and their corresponding pairs for the groups $G(m,p,n)$, with $(m,p)\neq (2,2)$. It also enables us to obtain a numerical formula for the dimension of the Brauer-Chen algebra in these cases.

\begin{proposition}\label{prop:gmpngeneral} Let $W=G(m,p,n)$, and suppose that $(m,p)\neq (2,2)$. Then a nonempty transverse collection $B$ is $K$-admissible if and only if it has cardinality $1$ or does not contain a hyperplane of the form $H_i$. In that case, $\Ann_{K\Stab(B)}(B)=\Aug(K_B)$. Finally, the dimension of $\Br^K(W)$ is equal to:

$$n! \frac{m^n}{p}+(1-δ_{pm})n!m^{n-1}n+\frac{m^{n+1}}{p}\sum_{r=1}^{2r\leq n} \left(\frac{n!}{r!2^r(n-2r)!}\right)^2(n-2r)!,$$
where $δ_{pm}$ is $1$ if $p=m$ and $0$ otherwise.

\end{proposition}
\begin{proof} The characterisation of admissible collections is a consequence of the previous results and discussion. Now, Propositions \ref{prop:onehi} and \ref{prop:classgeneral} imply that if $B\in \mathcal{C}^K_{adm}$, then $KW\Rel(B)=KW\Aug(K_B)$. Proposition \ref{prop:relann} tells us that
$\Ann_{K\Stab(B)}(B)$ is equal to $π_1(KW\Rel(B))$, where $π_w$ denotes the projection of $KW$ onto the coset $wK\Stab(B)$. Now, by Proposition \ref{prop:classgeneral} the latter is equal to 
$π_1(KW(K_B-1))$ which is in turn equal to $\Aug(K_B)$ since $K_B-1$ lies inside $K\Stab(B)$. This gives the first part of the statement. In particular, we have that $\dim_K(K\Stab(B)/\Ann_{K\Stab(B)})=\lvert \Stab(B) \rvert/ \lvert K_B \rvert $.

Now, recall, the formula for the dimension from Corollary \ref{cor:dimension}:

$$\sum_{B\in \mathcal{C}^K_{adm}}\frac{\lvert W \rvert }{\lvert \Stab(B) \rvert }\cdot \dim_K\left(\faktor{K\Stab(B)}{\Ann_{K\Stab(B)}(B)}\right).$$

Replacing the above into this formula we get:

$$\lvert W \rvert + \sum_{B\in \mathcal{C}^K_{adm}, B\neq \emptyset} \frac{\lvert W \rvert }{\lvert K_B \rvert}.$$

If $p\neq m$ we have the collections $\{H_i\}$ with $\lvert K_B \rvert =m/p$.  
Apart from them, for all collections of the form $\{H_{i_1j_1}^{κ_1},\dots H_{i_rj_r}^{κ_r}\}$ the group $K_B$, by Proposition \ref{prop:classgeneral} has order $2^{r}m^{r-1}r!$, and for each $r\leq n/2$ the number of such collections is:

$$\frac{{n \choose 2}{n-2 \choose 2}\dots {n-2(r-1) \choose 2}m^r}{r!}=\frac{n!m^r}{2^r(n-2r)!r!}.$$
So the sum $\sum_B \lvert W \rvert/\lvert K_B \rvert,$
where $B$ runs all such collections of size $r$ is

$$\frac{n!m^r}{2^r(n-2r)!r!}\frac{\lvert W \rvert}{\lvert K_B \rvert}=\frac{m^{n+1}}{p}\left(\frac{n!}{r!2^r(n-2r)!}\right)^2(n-2r)!.$$
Adding it all together, one obtains the expression of the statement.
\end{proof}

We turn, now, to the groups $G(2,2,n)$ which only contain hyperplanes of the form $H_{ij}$ and $H^1_{ij}$, since $H_{ij}^κ=H_{ij}^{κ'}$ if $κ\equiv κ' \mod 2$. Remark \ref{remark:sumtransversality} implies that such a collection $B$ is transverse if and only if for every $H_{ij}^κ\in B$, $i,j$ do not appear as indices in any other hyperplane of $B$ apart from, possibly, $H_{ij}^{κ+1}$.

\begin{proposition}\label{prop:g22nnonadm} Let $W=G(2,2,n)$ and $B$ a transverse collection such that $H_{i_1j_1},H_{i_1j_1}^1,H_{i_2j_2}^κ\in B$ but $H_{i_2j_2}^{κ+1}\not\in B$, for some $i_1,j_1,i_2,j_2$. Then, $B\not\in \mathcal{C}^K_{adm}$.
\end{proposition}
\begin{proof} Let $V$ be a $K\Stab(B)$-module. We show that $\Rel(B)\cdot \hat{V}=0$  implies $V=0$ (recall that $\hat{V}$ is the image of $V$ in the induced module of $KW$). For this, take the element $σ^γ_{α,β}\in \Rel(B)$ for $α=H_{i_1j_1}, β=H_{i_2j_2}^κ$ and $γ=H_{i_1j_2}$, and see that $σ^ γ_{α,β}\hat{V}=0$ gives $μ((j_1j_2)-(i_1i_2)_κ)\hat{V}=0$. Now, we will show that the reflections $(j_1j_2)$ and $(i_1i_2)_κ$ belong to different cosets of $\Stab(B)$; hence, $μ((j_1j_2)-(i_1i_2)_κ)\hat{V}=0$ implies $(i_1i_2)_κ\hat{V}=(i_1i_2)_κ\hat{V}=0$ which yields, of course, $V=0$.

To verify that $(j_1j_2),(i_1i_2)_κ$ belong to different cosets of $\Stab(B)$, notice that the only hyperplanes that are not fixed by $(j_1j_2),(i_1i_2)_κ$ are $H_{i_1j_1},H_{i_1j_1}^1$ and $H_{i_2j_2}^κ$, so it suffices to check their respective images via $(j_1j_2),(i_1i_2)_κ$. For this,
$(j_1j_2)$ maps $H_{i_1j_1},H_{i_1j_1}^1,H_{i_2j_2}^κ$ to $H_{i_1j_2},H_{i_1j_2}^1,H_{i_2j_1}^κ$, and $(i_1i_2)_κ$ maps the same hyperplanes to 
$H_{i_2j_1}^κ,H_{i_2j_1}^{κ+1},H_{i_1j_2}^{2κ}$. One can see now that $(j_1j_2)B\neq (i_1i_2)_κB$, since, for example, $H_{i_1j_2}^1$ belongs to the first collection but not the second. This concludes the proof. 
\end{proof}

\begin{proposition}\label{prop:classg22n} Let $W=G(2,2,n)$ and $B$ a collection of transverse hyperplanes of the form 
$\{H_{i_1j_1},H_{i_1j_1}^1, \dots ,H_{i_rj_r},H_{i_rj_r}^1\}$. Then $\mathcal{A}W\Rel(B)=\mathcal{A}W(K_B-1)$ and, in particular, $B\in \mathcal{C}^K_{adm}$. Finally, the order of $K_B$ is $2^{r+n-1}r!$.

\end{proposition}
\begin{proof} We show, as we did in Proposition \ref{prop:classgeneral}, that for all $\mathcal{A}W$-modules $V$ and $v\in V$, $\Rel(B)v=0$ is equivalent to $(K_B-1)v=0$. So let $V$ be an $\mathcal{A}W$-module and $v\in V$.

The relations $(\mathcal{R}_B-1)v=0$ give $((ij)-1)v=((ij)_1-1)v=0$ for all $H_{ij}\in B$.
Next we see what the relations $Σ^γ_Bv=0, γ\in \mathcal{H}$ yield. Let $γ=H_{ij}^κ$, for some $i,j,κ$. By Remark \ref{remark:sumtransversality}, two hyperplanes 
of $G(2,2,n)$ are non-transverse if and only if they have exactly one index in common. So, if $γ$ is non-transverse with more than one element of $B$ (otherwise $Σ^γ_B=\{0\}$), then either hyperplanes $a=H_{i_1j}, b=H_{i_1j}^1, c=H_{ij_1, }, d=H_{ij_1}^1$ belong to $B$, in which case, these are the only hyperplanes $γ$ is non-transverse with, or $α=H_{ij_1}, β=H_{ij_1}^1\in B$, and $j$ does not appear in any hyperplane of $B$, in which case $α,β$ are the only hyperplanes of $B$ with which $γ$ is non-transverse.

For the first case, by Lemma \ref{lem:Gmpn}, for each of the hyperplanes $a,b,c,d$, there is a unique reflection mapping them to $γ$. These are, respectively, 
$r_1=(i_1i)_κ, r_2=(i_1i)_{κ+1}, r_3=(j_1j)_κ, r_4=(j_1j)_{κ+1}$. So, $Σ^γ_Bv=0$ yields $6$ relations (all the possible differences of the above reflections), which are, equivalent, for example, to
$μ(r_1-r_2)v=μ(r_2-r_3)v=μ(r_3-r_4)v=0$, or, equivalently $(r_2^{-1}r_1-1)v=(r_3^{-1}r_2-1)v=(r_4^{-1}r_3-1)v=0$. 

The second case yields the relation $σ^γ_{α,β}v=0$ which, by Lemma \ref{lem:Gmpn}, is $(u_1-u_2)v=0$, or equivalently, $(u_2^{-1}u_1-1)v=0$, where
$u_1=(j_1j)_κ$ and $u_2=(j_1j)_{κ+1}$.  

Given the above, one can see that letting $γ$ vary, and considering the relations 
$(\mathcal{R}_B-1)v=0$ as well, we obtain a subgroup $K_B^0$ for which $\Rel(B)v=0$ if and only if $(K_B^0-1)v=0$. Although it is not necessary for the rest of this proof, we mention here, for later reference, the following explicit description that we obtain for $K_B^0$:
$$K_B^0=\langle (ij_1)_κ, (i_1i)_κ(i_1i)_{κ+1},(i_1i)_κ(j_1j)_κ, (j_1j_2)_κ(j_1j_2)_{κ+1}\rangle,$$
with $H_{ij_1}, H_{i_1j}\in B, κ\in \mathbb{Z}$, and  $1\leq j_2 \leq n$. Elements of the first form in the above presentation come from the relations $(\mathcal{R}_B-1)v=0$. Elements of the second and third form come from the relations $(r_1r_2^{-1}-1)v=0$ and $(r_3r_1^{-1}-1)v=0$, respectively, and the elements of the fourth form from relations $(r_3r_4^{-1}-1)=0$ and $(u_1u_2^{-1}-1)v=0$, where, for the latter, we have replaced the variable $j$ with $j_1$ and $j_1$ with $j_2$ to give a more concise description. We show that $K_B^0$ coincides with $K_B$. 

For the inclusion $K_B^0\subseteq K_B$, since, by the definition of $K_B$ (Definition \ref{def:KB}), $\mathcal{R}_B\subseteq K_B$, it suffices to show that $(u_1,u_2)$ as well as all pairs among the reflections $r_1,r_2,r_3,r_4$ satisfy conditions (c1), (c2) of the definition of $K_B$ i.e. $s_1B=s_2B\neq B$, and $s_1H\neq s_2H$ for all $H\in B\backslash s_1B$. 

For $u_1=(j_1j)_κ, u_2=(j_1j)_{κ+1}$, with $j$ not appearing in any hyperplane of $B$, the only hyperplanes of $B$ that are not left invariant by $u_1,u_2$ are $α=H_{ij_1},β=H_{ij_1}^1$ which are mapped via $u_1,u_2$, respectively, to $H_{ij}^κ,
H_{ij}^{κ+1}$ and $H_{ij}^{κ+1}, H_{ij}^{κ}$, in that order. One can see from this that  $u_1,u_2$ satisfy conditions (c1),(c2).

Similarly, for the pairs of reflections among $r_1,r_2,r_3,r_4$ computing the images of the quadruple  $(a,b,c,d)$ via these reflections we find, respectively, the following quadruples:
$(H_{ij}^{k}, H_{ij}^{κ+1}, H_{i_1j_1}^{κ}, H_{i_1j_1}^{κ+1})$, $(H_{ij}^{κ+1}, H_{ij}^{κ}, H_{i_1j_1}^{κ+1}, H_{i_1j_1}^{κ})$, $(H_{i_1j_1}^{κ}, H_{i_1j_1}^{κ+1}, H_{ij}^{κ}, H_{ij}^{κ+1})$, $(H_{i_1j_1}^{κ+1}, H_{i_1j_1}^{κ}, H_{ij}^{κ+1}, H_{ij}^{κ})$. Since all hyperplanes of $B$ other than $a,b,c,d$ are fixed by these reflections, one can verify from the above mappings that all pairs $(r_l,r_m)$, for $l,m=1, 2, 3, 4$ verify (c1), (c2), and, hence, $r_m^{-1}r_l\in K_B, l,m=1, 2, 3, 4$. Now, since $K_B^0$ is generated by all such products for the different values of $γ$, together with $\mathcal{R}_B$, we get that $K_B^0\subseteq K_B$.

For the inverse inclusion, since, again, $\mathcal{R}_B\subseteq K_B^0$, we show that for every pair of reflections $s_1,s_2$ satisfying (c1),(c2), we have $s_2^{-1}s_1\in K_B^0$. Observe first, that for every pair of hyperplanes in $G(2,2,n)$, there is at most one reflection mapping one to the other, and all reflections belong to the same orbit (see, for example, Lemma \ref{lem:Gmpn}). We show that for every pair $s_1,s_2\in \mathcal{R}$ satisfying (c1), (c2), $μ(s_1-s_2)\in Σ^H_B$ for some hyperplane $H$, which implies that $s_2^{-1}s_1\in K_B^0$, since $K_B^0$ is generated, by definition, by all such elements. For this, 
by (c1), i.e. $s_1B=s_2B\neq B$ there is some hyperplane $H'\in B$ such that $s_1H'\neq H'$. 
Take $H=s_1H'$ and $H''=s_2^{-1}H\in B$. Now, by Lemma \ref{lem:transverse1}, since, 
of course $\mathcal{R}_{H'\to H}, \mathcal{R}_{H''\to H}\neq \emptyset$, we have that
$H$ is non transverse with $H',H''$, which implies that 
$σ^H_{Η',Η''}\in Σ^H_B$. By the observation in the beginning of this paragraph, we get that $\mathcal{R}_{H'\to H}=\{s_1\}$, and  $\mathcal{R}_{H''\to H}=\{s_2\}$, which implies that $σ^H_{Η',Η''}=μ(s_1-s_2)$. Hence, $K^0_B=K_B$.
The order of $K_B$ is calculated in the appendix using the group $K_B^0$.
\end{proof}

The above results conclude the classification of admissibility for the groups $G(2,2,n)$. Again, we sum up the situation in the next proposition, giving a numerical expression for the dimension of the Brauer-Chen algebra, as well.

\begin{proposition} \label{prop:g22n}
\label{prop:G22ngeneral} Let $W=G(2,2,n)$. A transverse collection $B$ is $K$-admissible if and only if whenever $H_{ij},H_{ij}^1,H_{i'j'}^κ\in B$, we have $H_{i'j'}^{κ+1}\in B$ as well. In that case, $\Ann_{K\Stab(B)}(B)=\Aug(K_B)$. Finally, the dimension of $\Br^K(W)$ is:

$$n!2^{n-1}+(2^n+1)\sum_{r=1}^{2r\leq n} \left(\frac{n!}{r!2^r(n-2r)!}\right)^2(n-2r)!.$$
\end{proposition}
\begin{proof}
The characterisation of admissible collections of the statement is immediate from the previous results. Notice also that it implies that a transverse collection of $G(2,2,n)$ is admissible if and only if it is of the form $\{H_{i_1j_1}^{κ_1}, \dots , H_{i_rj_r}^{κ_r}\}$, or $\{H_{i_1j_1},H_{i_1j_1}^1 \dots , H_{i_rj_r},H^1_{i_rj_r}\}$, with all indices distinct in both cases. Transverse collections of both of these types are admissible by Propositions \ref{prop:classgeneral} and \ref{prop:classg22n}, respectively, and the same propositions imply that if $B$ is of these types, then $KW\Rel(B)=KW(K_B-1)$.

Now, exactly as we did in the beginning of the proof of Proposition \ref{prop:gmpngeneral}, this yields that 
for $B\in \mathcal{C}^K_{adm}$, $\Ann_{K\Stab(B)}(B)=\Aug(K_B)$, and the formula for the dimension of $\Br^K(W)$ of Corollary \ref{cor:dimension} becomes:

 $$\lvert W \rvert + \sum_{B\in \mathcal{C}^K_{adm}, B\neq \emptyset} \frac{\lvert W \rvert }{\lvert K_B \rvert}.$$

For every $r\leq n/2 $ the number of transverse collections of the form $\{H_{i_1j_1},H_{i_1j_1}^1 \dots , H_{i_rj_r},H^1_{i_rj_r}\}$ is: 

$$\frac{{n \choose 2}{n-2 \choose 2}\dots {n-2(r-1) \choose 2}}{r!}=\frac{n!}{2^r(n-2r)!r!},$$
and the order of $K_B$, by Proposition \ref{prop:classg22n}, for each such collection is equal to $2^{r+n+1}r!$. So the sum $\sum_{B}\lvert W \rvert /\lvert K_B \rvert$, where $B$ runs over all such collections is: 
 
$$\sum_{r=1}^{2r\leq n} \frac{n!}{(n-2r)!2^rr!} \frac{n!2^{n-1}}{2^{r+n-1}r!}=\sum_{r=1}^{2r\leq n} \left(\frac{n!}{r!2^r(n-2r)!}\right)^2(n-2r)!.$$

Finally, for the transverse collections of the form $\{H_{i_1j_1}^{κ_1},\dots ,H_{i_rj_r}^{κ_r}\}$, the corresponding sum was calculated for Proposition \ref{prop:gmpngeneral}, and equals:

$$2^n\sum_{r=1}^{2r\leq n} \left(\frac{n!}{r!2^r(n-2r)!}\right)^2(n-2r)!. $$ 
Again, adding it all up (and the order of the group) gives the expression of the statement.
\end{proof}

\subsection{The exceptional groups}
\label{sec:admexceptionals}

In this subsection we describe the classification of admissible collections and the corresponding admissible pairs for the exceptional complex reflection groups. To obtain this classification we used computational methods. After some new definitions, we describe the computational results that lead to the classification and explain the algorithm we used to verify them. A table containing the obtained numerical data concerning the admissible pairs for each exceptional group as well as the dimension of the corresponding Brauer-Chen algebra can be found in the appendix. 

%Notice that for the exceptional groups of rank 2, Corollary \ref{cor:admcardone} already determines all admissible pairs, since there are, in fact, no transverse hyperplanes in these groups (every intersection of two distinct reflecting hyperplanes is the trivial subspace, contained in every other reflecting hyperplane). In any case, the following study includes these cases, as well.

All mentioned transverse collections will be assumed non-empty.

\subsubsection{A convenient description for $\overline\Rel(B)$}\label{subsec:aconvdescrforrel}

Let $W$ be a complex reflection group and $B$ a transverse collection. By Definition \ref{def:annihilators}, we have $$\overline\Rel(B)=\cup_{w\in W}π_w(\Rel(B)),$$ where $π_w$ denotes the projection 
$KW\to wK\Stab(B)$ with respect to the decomposition $KW=\oplus_{w\in W/\Stab(B)}wK\Stab(B)$.
The coset $w\Stab(B)$ depending only on the transverse collection $B'=wB$, we denote, in this subsection, the corresponding projection by $π_{B'}$, i.e. if $x=\sum_{w\in W}λ_ww\in KW$, then $π_{B'}(x)=\sum_{wB=B'}λ_ww$. With this notation we have
$\overline\Rel(B)=\cup_{B'\sim B} π_{B'}(\Rel(B))$.

Now, recall that 
$\Rel(B)=(\mathcal{R}_B-1)\cup Σ_B$, where $Σ_B$ is the set consisting of all elements $$σ^H_{H_1,H_2}=\sum_{s\in \mathcal{R}_{H_1\to H}}μ_ss-\sum_{s\in \mathcal{R}_{H_2\to H}}μ_ss,$$
with $H_1,H_2\in B, H\in \mathcal{H}$, and $H\not \pitchfork H_1,H_2$.  

For the projections of the set $\mathcal{R}_B-1$, since it lies inside $K\Stab(B)$, we have $π_{B}(\mathcal{R}_B-1)=\mathcal{R}_B-1$, and
$π_{B'}(\mathcal{R}_B-1)=0$ for all $B'\neq B$. 

Secondly, for all $H,H_1,H_2\in \mathcal{H}$, and $B'$ in the orbit of $B$, we have

$$π_{B'}(σ^H_{H_1,H_2})= \sum_{s\in \mathcal{R}_{H_1 \to H, B\to B'}}μ_ss-\sum_{s\in \mathcal{R}_{H_2\to H, B\to B'}}μ_ss.$$
If, now, $H_1,H_2\in B$, and $H\not \pitchfork H_1,H_2$, then $H\not \in B$; so, for $B'=B$, we have $π_{B'}(σ^H_{H_1,H_2})=0$. This is also the case when $\mathcal{R}_{B\to B'}=\emptyset$, since the above summations are empty. With this in mind, we have the following expression for $\overline\Rel(B)$, which we will use later:

\begin{equation}\label{eq:relbar}
\overline\Rel(B)=(\mathcal{R}_B-1)\cup (\cup_{B'\sim_{\mathcal{R}}B,B'\neq B}π_{B'}(Σ_B)),
\end{equation}
where $B'\sim_{\mathcal{R}} B$ stands for $\mathcal{R}_{B\to B'}\neq \emptyset$.

We define the sets $Θ_B^1:=\{μ_ss \big| s\in \mathcal{R}\backslash \mathcal{R}_B\}$ and $Θ_B:=Θ_B^1\cup \mathcal{R}_B\cup\{1\}$,
and let $M_B$ denote the free $\mathbb{Q}$-module on $Θ_B$. Also, let $M_B^1, M_B^2$ be, respectively, the submodules $\mathbb{Q}Θ_B^1$ and $\mathbb{Q}\mathcal{R}_B\cup\{1\}$ of $M_B$.

From the above, one can see that we can consider $\overline\Rel(B)$ as a subset of $M_B$ for which
$$\overline\Rel(B)=(\overline\Rel(B)\cap M_B^1 )\cup (\overline\Rel(B)\cap M_B^2),$$ and $$\overline\Rel(B)\cap M_B^2=\mathcal{R}_B-1.$$

\begin{definition}\label{def:dbpb} Let 
\begin{equation*}
\begin{split}
θ:& \mathcal{R}\cup\{1\} \to Θ_B;\\
& s\mapsto μ_ss, \textnormal{ if }s\in \mathcal{R}\backslash\mathcal{R}_B,\\
& s\mapsto s, \textnormal{ if }s\in \mathcal{R}_B\cup\{1\}.\\
\end{split}
\end{equation*}
We define the following sets:
\begin{equation*}
\begin{split}
&D_B:=\{θ(s_1)-θ(s_2)\big| s_1,s_2\in \mathcal{R}\}\cap \mathbb{Q}\overline\Rel(B)\subseteq M_B,\\
&P_B:=\{(s_1,s_2)\big| s_1,s_2 \in \mathcal{R}\backslash \mathcal{R}_B, μ_{s_1}s_1-μ_{s_2}s_2\in D_B\}\subseteq \mathcal{R}^2,\\
&D_B^0:=\{s_2^{-1}s_1-μ_{s_2}μ_{s_1}^{-1}\big| (s_1,s_2)\in P_B\}\cup (\mathcal{R}_B-1)\subseteq \mathcal{A}W.\\
\end{split}
\end{equation*}

\end{definition}

\begin{lemma} Let $s_1,s_2\in\mathcal{R}\cup\{1\}$. If $θ(s_1)-θ(s_2)\in \mathbb{Q}\overline\Rel(B)$, then either $s_1,s_2\in \mathcal{R}_B\cup \{1\}$, and, hence, $θ(s_1)-θ(s_2)=s_1-s_2$, or $s_1,s_2\in \mathcal{R}\backslash \mathcal{R}_B$, and, hence,
$θ(s_1)-θ(s_2)=μ_{s_1}s_1-μ_{s_2}s_2$. 
\end{lemma}
\begin{proof} Suppose that 
$θ(s_1)-θ(s_2)\in \mathbb{Q}\overline\Rel(B)$ and that $s_1,s_2$ do not satisfy the condition of the statement. Let
$s_1\in \mathcal{R}\backslash \mathcal{R}_B$ and $s_2\in \mathcal{R}_B\cup \{1\}$.  Hence, $θ(s_1)\in M_B^1$ and $θ(s_2)\in M_B^2$.  Since $\mathbb{Q}\overline\Rel(B)$ respects the decomposition $M_B=M_B^1\oplus M_B^2$,
i.e. $\overline\Rel(B)=(\mathbb{Q}\overline\Rel(B)\cap M_B^1)\oplus (\mathbb{Q}\overline\Rel(B) \cap M_B^1)$, then the projection of $θ(s_1)-θ(s_2)$ onto $M_B^2$ belongs to $\overline\Rel(B)$. Hence, we get $θ(s_2)=s_2 \in \overline\Rel(B)\cap M_B^2=\mathbb{Q}(\mathcal{R}_B-1)$, and one can see that this is impossible.
\end{proof}

\begin{corollary}\label{cor:passagetostab} We have 
$$D_B=\{μ_{s_1}s_1-μ_{s_2}s_2\big| (s_1,s_2)\in P_B\}\cup \{s_1-s_2\big| s_1,s_2\in \mathcal{R}_B\cup \{ 1 \} \}.$$ In particular,  $\mathcal{A}W D_B=\mathcal{A}W D_B^0$.
\end{corollary}

\subsubsection{The main points of the algorithm}

We explain now the main points of the algorithm we used to obtain the results for the admissibility for the exceptional complex reflection groups. Almost all created lists will be used again in the next section, where we study the freeness of the Brauer-Chen algebra.

\paragraph{\textbf{The main core.}}
We consider the following data. A complex reflection group $\verb+W+$, given as an abstract group (a group of permutations, for example),
and a list $\verb+R_dist+$ of the elements of $\verb+W+$ that correspond to the distinguished reflections (Subsection \ref{sec:basicscrg}). Let, also, $\verb+k_W+$ 
be the field of definition of $\verb+W+$ and $\verb+Roots+$ be a list containing a root for each distinguished reflection. 

The elements of $\verb+R_dist+$ are in bijection with the hyperplanes of $\verb+W+$; let $H_i$ be the hyperplane that corresponds to the
distinguished reflection in position $i$ of $\verb+R_dist+$. This way, testing, for an element $w\in \verb+W+$, whether 
$wH_i=H_j$, amounts to testing whether $w\verb+R_dist+[i]w^{-1}=\verb+R_dist+[j]$. Also, by Lemma \ref{lem:transroots}, checking that $H_i\pitchfork H_j$ is
equivalent to checking that there is not any root in $\verb+Roots+$ lying in the $\verb+k_W+$-span of $\{\verb+Roots[i],Roots[j]+\}$
apart from these two.

It is convenient to create a table containing all the necessary relations of transversality and conjugation between the hyperplanes
of $\verb+W+$. For that, let $\verb+R+$ be a list containing all reflections of $\verb+W+$ (take, for example, all non-trivial powers of the 
elements of $\verb+R_dist+$), and create a table $\verb+F+$ as follows: if $H_i\pitchfork H_j$, then $\verb+F[i,j]=true+$; otherwise,
$\verb+F[i,j]+$ is a list with the positions of the elements of $\verb+R+$ that map $H_i$ to $H_j$. 

Using the table $\verb+F+$, we can find all transverse collections of $\verb+W+$, and, using standard orbit algorithms, we can
obtain a representative for each orbit of transverse collections. We represent a transverse collection as a list of positions in
$\verb+R_dist+$.

\paragraph{\textbf{Lists for $\overline\Rel(B), D_B, P_B$.}} Now, let $\verb+B+$ be such a list representing a transverse collection $B$. Let $N$ be the cardinality of $\mathcal{R}$. We create a list 
$\verb+Rel_Bar+$ of vectors in  $\mathbb{Z}^{N+1}$ representing the set $\overline\Rel(B)$ with respect to the basis $Θ_B$ of $M_B$.

First, we make a list $\verb+R_B+$ containing the positions in $\verb+R+$ 
of all non-trivial powers of the elements $\verb+R_dist[i]+$, for $i\in \verb+B+$; this list represents the set $\mathcal{R}_B$.
Also, we make a list $\verb+Smallorbit+$ which contains the lists representing the images of $B$ under all reflections; this represents the set of all $B'\sim_{\mathcal{R}}B$.

Let $v(i,n)$ denote the vector of length $n$, with $1$ in position $i$ and $0$ elsewhere. First, for all
$i\in \verb+R_B+$ we add the vector $v(i,N+1)-v(N+1,N+1)$ to $\verb+Rel_Bar+$. This vector represents the element 
$\verb+R[i]+-1$ with respect to $Θ_B$. In this way, we include all vectors  representing $\mathcal{R}_B-1$ to our list.

Next, for all $B'\sim_{\mathcal{R}}B, B'\neq B$ we add the vectors representing the set $π_{B'}(Σ_B)$, which (apart from $0$) consists of the 
elements $π_{B'}(σ^H_{H_1,H_2})$ for $H_1,H_2\in B, H\in B'$ and $H\not\pitchfork B$. Instead of imposing controls on all candidate triples of hyperplanes, going back to Lemma \ref{lem:precom}, one may see that,
since $\mathcal{R}_{B\to B'}\neq \emptyset$, choosing $H\in B'$ and $H_1,H_2\in B$ non-transverse with $H$ amounts to 
choosing $H\in B'\backslash B$ and $H_1,H_2\in B\backslash B'$. 

With this in mind, for every $\verb+B'+$ in $\verb+Smallorbit+$, we first
create a table $\verb+F(B,B')+$, with rows and columns corresponding to the elements of $\verb+B+\backslash \verb+B'+$ and $\verb+B'+\backslash \verb+B+$, respectively, and
whose $(i,j)$-cell is a list of positions of all reflections in $\verb+R+$ that map $H_{\verb+B[i]+}$ to $H_{\verb+B'[j]+}$ and $B$ to $B'$ (note that, in this way, for the indices to agree,
one has to make sure that the intersection of $\verb+B+, \verb+B'+$ takes up the last spots on each list). Now, for all 
$i,j,k$ in the range of $\verb+F(B,B')+$, we add the vector 
$\sum_{l\in \verb+F(B,B')[i,k]+}v(l,N+1)-\sum_{l\in \verb+F(B,B')[j,k]+}v(l,N+1) $ to the list $\verb+Rel_Bar+$, as
well. This vector represents the element $π_{B'}(σ_{H',H''}^H)$ where $H,H',H''$ are the hyperplanes $H_{\verb+B'[k]+},H_{\verb+B[j]+},H_{\verb+B[j]+}$, respectively. In this way we include
all vectors representing $π_{B'}(Σ_B)$ to $\verb+Rel_Bar+$, as well. 

Before we state our first computational result, let $\verb+D_B+$ be a list of all vectors $v(i,N+1)-v(j,N+1)$, for $i,j=1,\dots N+1$ that lie in 
the $\mathbb{Q}$-span of $\verb+Rel_Bar+$; this represents the set $D_B$ with respect to the basis $Θ_B$. Also, let
$\verb+P_B+$ be the list of all pairs $(i,j), i,j=1,\dots ,N$, with $i,j\not \in \verb+R_B+$, such that $v(i,N+1)-v(j,N+1)\in \verb+D_B+$; this 
represents the set $P_B$.

\begin{mdframed}
\textbf{First Computational Result.} Let $W$ be an exceptional complex reflection group and $B$ a transverse
collection. Then

\begin{itemize}
\item[(A1)]$θ(s)\in \overline\Rel(B)$,  for some $s\in \mathcal{R}\cup\{1\}$, or
\item[(A2)] 
\begin{itemize}
\item[a.] $\mathbb{Q}\overline\Rel(B)=\mathbb{Q}D_B$, and
\item[b.] $\langle s_2^{-1}s_1, s \big| (s_1,s_2)\in P_B, s\in \mathcal{R}_B\rangle = K_B$
\end{itemize}
 
\end{itemize}
\end{mdframed}
The verification of the above result for a given transverse collection represented by $\verb+B+$ is straightforward using the previously constructed lists. Checking, in particular, the first condition amounts to checking whether $\verb+Rel_Bar+$ contains the vector $u(i,N+1)$ for some $i$.

 Let $W$ be an exceptional complex reflection group, $K$ a proper field, and $B$ a transverse collection. Notice, first of all, that the
two properties (A1), (A2) above are mutually exclusive, since a vector $θ(s)$ cannot be a linear combination of vectors of the form $θ(s_1)-θ(s_2)$, which comprise $D_B$.

If $B$ satisfies (A1), then, since $θ(s)$ is invertible, $B$ is not
$K$-admissible. 
%Actually, one can see that $\overline\Rel(B)$ does not contain elements of $\mathcal{R}_B \cup \{1\}$, since $\overline\Rel(B)\cap M_B^2=\mathbb{Q}(\mathcal{R}_B-1)$, so $s\in \mathcal{R}\backslash \mathcal{R}_B$, and, hence $θ(s)=μ_ss$.
Suppose now that $B$ satisfies (A2). By (A2.a), $\mathbb{Q}\overline\Rel(B)=\mathbb{Q}D_B$, and, thus, admissibility of a pair
$(B,V)$ is equivalent to $D_B\cdot \hat{V}=0$, which is, is turn, equivalent to $D_B^{0}\cdot \hat{V}=0$, as Corollary \ref{cor:passagetostab} implies.
Since $K_B\subseteq \Stab(B)$, (A2.b) implies 
that $s_2^{-1}s_1\in \Stab(B)$ for all $(s_1,s_2)\in P_B$, which, in turn, yields that 
$D_B^{0}\subseteq K\Stab(B)$. So we can drop the hat on $\hat{V}$, and write $D_B^{0}\cdot V=0$. This already implies, since, for any $K\Stab(B)$-module $V$, the pair $(B,V)$ is $K$-admissible if and only if $\Ann_{KStab(B)}(B)\cdot V=0$, that $\Ann_{K\Stab(B)}(B)$ is equal to $(D_B^0)$, where $(D_B^0)$ denotes the two-sided ideal of $K\Stab(B)$ generated by $D_B^0$.

Now, if $μ_{s_1}=μ_{s_2}$ for all $(s_1,s_2)\in P_B$, then $D_B^{0}=\{s_2^{-1}s_1-1\big| (s_1,s_2)\in P_B\}\cup (\mathcal{R}_B-1)$, and, in combination with (A2.b), we get that $(B,V)$ is admissible if and only if 
$(K_B-1)V=0$; in other words, $(D_B^0)=\Aug(K_B)$. This is especially the case if $s_1,s_2$ are conjugate for all $(s_1,s_2)\in P_B$. 

\begin{definition}
Let $W$ be an exceptional complex reflection group and $B$ a transverse collection. If $B$ satisfies (A2) and $P_B$ contains pairs $(s_1,s_2)$ with $s_1\not\sim s_2$, we say that $B$ is of \textit{conditional type}. Denote by $\mathcal{C}_{cond}$ the set of collections of conditional type.
\end{definition}

The discussion preceding the above definition yields the next lemma.

\begin{lemma}\label{lem:sumupconditions} Let $W$ be an exceptional complex reflection group and $B\in \mathcal{C}\backslash \mathcal{C}_{cond}$. For any proper field $K$, $B\in \mathcal{C}^K_{adm}$ if and only if $B$ does not satisfy (A1), in which case $\Ann_{K\Stab(B)}(B)=\Aug(K_B)$.
\end{lemma}

Verifying whether a transverse collection is of conditional type is straightforward
using the list $\verb+P_B+$.

\begin{mdframed}
\textbf{Second Computational Result.} Let $W$ be an exceptional complex reflection group other than $G_{25},G_{32}$. Then
$W$ does not have any transverse collections of conditional type. 
\end{mdframed}

We, thus, obtain the following result.
\begin{proposition}\label{prop:classpart1} Let $W$ be an exceptional complex reflection group other than $G_{25},G_{32}$ and $K$ a proper field. Then $B\in \mathcal{C}^K_{adm}$ if and only if $\overline\Rel(B)$ does not contain an element of $Θ_B$. If $B\in \mathcal{C}^K_{adm}$, then $\Ann_{K\Stab(B)}(B)=\Aug(K_B).$
\end{proposition}

\vspace{20pt}
For the groups $G_{25}$ and $G_{32}$, there exist transverse collections of conditional type. Unlike collections of non-conditional type, admissibility of such collections depends, in general, on the field of definition. The following general algebraic fact determines the admissibility of these collections for a large number of cases. 

\begin{lemma} \label{lem:invertibility} Let $L$ be a field and $G$ a group. Let $l\in L$ and $g$ be an element of $G$ of finite order $n$. Then, if $l^n\neq 1$, the element $g-l\in LG$ is invertible.
\end{lemma}
\begin{proof} Set $x=g-l$. Since, $g^{n}=1$ and $g=x+l$, then $(x+l)^{n}=1\Leftrightarrow x^n+nx^{n-1}l+ \dots +nxl^{n-1}+l^n=1\Leftrightarrow x(x^{n-1}l+\dots +nl^{n-1})=1-l^{n}$. Now, if $1-l^n\neq 0$, then $1-l^n$ is invertible, and so is $x$.
\end{proof}

\begin{mdframed}
\textbf{Third Computational Result, first part.} Let $B$ be a transverse collection of $G_{25}$ (resp. $G_{32}$). Then
$B$ satisfies satisfies (A2) of the First Computational Result and is of conditional type if and only if $\lvert B \rvert =2$ (resp. $2$ or $3$). In that case, for all $(s_1,s_2)\in P_B$,
with $s_1\not\sim s_2$, the order of the element $s_2^{-1}s_1$ is $6$. 
\end{mdframed}

Let $W$ be the group $G_{25}$ or $G_{32}$, and $K$ a proper field for $W$. Then $W$ contains two orbits of reflections, one containing all distinguished reflections, and the other one their inverses. Let $μ_1,μ_2$ be their respective parameters and set $μ:=μ_1/μ_2$. 

Let $B$ be a collection of conditional type, i.e., $P_B$ contains some $(s_1,s_2)$ such that $s_1\not \sim s_2$. Recall that since, by definition, $B$ satisfies (A2) of the First Computational Result, we have that $\Ann_{K\Stab(B)}(B)=(D_B^0)$, and the latter ideal is generated by $
(\mathcal{R}_B-1)\cup\{s_2^{-1}s_1-μ_{s_2}μ_{s_1}^{-1}\big| (s_1,s_2)\in P_B\}$. By (A2.b), $\mathcal{R}_B\cup \{s_2^{-1}s_1\big| (s_1,s_2)\in P_B\}$ generates $K_B$, so, if $(D_B^0)\neq K\Stab(B)$, then mapping $\mathcal{R}_B\mapsto 1$, and $s_2^{-1}s_1\mapsto μ_{s_2}μ_{s_1}^{-1}$ induces a group homomorphism $χ:K_B\to \mathbb{U}_6 \cap K$, and $(D_B^0)=\Aug^χ(K_B)$.

Now, by the Third Computational Result above and its preceding lemma (\ref{lem:invertibility}), if $μ^6\neq 1$ in $K$, then the element $s_2^{-1}s_1-μ_{s_2}μ_{s_1}^{-1}$ is invertible, since
$μ_{s_2}μ_{s_1}^{-1}\in \{μ,μ^{-1}\}$; hence, $B\not \in \mathcal{C}^K_{adm}$. If $μ^6=1$ in $K$, to determine the admissible pairs for $B$, one may resort - as we did - to computing the ideal $(D_B^0)$.

\paragraph{\textbf{Computing $(D_B^{0})$ for $G_{25},G_{32}$.}}
Let $\verb+W+$ be the group $G_{25}$ or $G_{32}$, and $\verb+B+$ a list representing a transverse collection $B$ of conditional type (the conditionality property is actually irrelevant for the algorithm).
Let $μ\in \mathbb{U}_6$. Since $D_B^{0}\subset \mathbb{Q}[μ]\Stab(B)$, it suffices to calculate the ideal of $\mathbb{Q}[μ]\Stab(B)$ generated by $D_B^{0}$. Let $\verb+Stab_B+$ be a list containing the elements of $\Stab(B)$ and let $N'$ be its length; suppose, also, that $1$ is in the first spot of this list. Finally, let $\verb+Orb_1,Orb_2+$ be the lists containing the positions of reflections in $\verb+R+$ that make up the two orbits.

We create a list $\verb+DBstab+$ containing the vectors that represent $D_B^0$ with respect to the basis $\verb+Stab_B+$. 
In order to do this, for all $i\in \verb+R_B+$ we add the vector $v(\verb+Position(R[i],Stab_B)+,N')-v(1,1)$ to the list $\verb+DBstab+$; this is the vector representing $\verb+R[i]+-1$. In this way we include all vectors representing $\mathcal{R}_B-1$. Next, for every $(i,j)\in \verb+P_B+$:

\begin{enumerate}

\item  if $i,j\in \verb+Orb_1+$ or $i,j \in \verb+Orb_2+$, then we add the vector 
$v(\verb+Position(R[j]+^{-1}\verb+R[i]+,\verb+Stab_B)+,N'))-v(1,1)$, representing the element $\verb+R[j]+^{-1}\verb+R[i]+-1$,

\item if $i\in \verb+Orb_1+$ and $j\in \verb+Orb_2+$, then we add the vector
$v(\verb+Position(R[j]+^{-1}\verb+R[i]+,\verb+Stab_B)+,N'))-μ^{-1}\cdot v(1,1)$, representing the element $\verb+R[j]+^{-1}\verb+R[i]+-μ^{-1}$

\item if $i\in \verb+Orb_2+$ and $j\in \verb+Orb_1+$, then we add the vector 
$v(\verb+Position(R[j]+^{-1}\verb+RR[i]+,\verb+Stab_B)+,N'))-μ\cdot v(1,1)$, representing the element $\verb+R[j]+^{-1}\verb+R[i]+-μ$
\end{enumerate}
This gives the desired list $\verb+DBstab+$.

We give, now, ths algorithm that we used to obtain the subspace of $\mathbb{Q}[μ]^{N'}$ that corresponds to the ideal $(D_B^0)$. Let $\verb+Gens+$ be a list of generators for $\Stab(B)$. 
For every element $g\in\verb+Gens+$ we produce matrices $\verb+left+(g), \verb+right+(g)$, which are, respectively, the matrices of left and right multiplication by $\verb+g+$ on $\mathbb{Q}[μ]\Stab(B)$, given in the basis $\verb+Stab_B+$. These are the monomial matrices that correspond to the permutations of $\verb+Stab_B+$ induced by multiplying each element by $g\in \verb+Gens+$ on the left and right, respectively.
 
Now, let $\verb+L+$ be a list of vectors of $\mathbb{Q}[μ]^{N'}$ representing a set of elements $L$ of $\mathbb{Q}[μ]\Stab(B)$. To calculate the subspace of $\mathbb{Q}[μ]^{N'}$ that corresponds to the ideal generated by $L$, we start with the list $\verb+L(0)=L+$, and, in step $i$, we obtain a  basis $\verb+L(i)+$ of the $\mathbb{Q}[μ]$-span of $\{v,\verb+left+(g)\cdot v,\verb+right+(g)\cdot v \big | v\in \verb+L(i-1)+, g\in \verb+Gens+\}$. The list $\verb+L(i)+$ corresponds to a basis of the vector space spanned by $L_{i-1}\cup (\verb+Gens+\cdot L_{i-1})\cup (L_{i-1}\cdot \verb+Gens+)$. We stop when the lengths of $\verb+L(i-1)+$ and $\verb+L(i)+$ are equal; this is the dimension of the ideal generated by $L$, and $\verb+L(i)+$ represents a basis of that ideal. 

We now state the last computational result for this section.

\begin{mdframed}

\textbf{Third Computational Result, second part.} Let $W$ be the group $G_{25}$ or $G_{32}$ and $K$ a proper field. If $μ^6=1$ in $K$, and $B\in \mathcal{C}_{cond}$, then $\dim_{K}(D_B^0)=\lvert \Stab(B) \rvert - \lvert \Stab(B) \rvert/ \lvert K_B \rvert $. In particular, $B$ is $K$-admissible.

\end{mdframed}

Finally, we have the following result for $G_{25}$ and $G_{32}$.

\begin{proposition}\label{prop:case2532} Let $W$ be the group $G_{25}$ (resp. $G_{32}$) and $K$ a proper field. Let $B$ a transverse collection. 

\begin{enumerate}
\item If $\lvert B \rvert \neq 2$ (resp. $2,3$), then $B\in \mathcal{C}^K_{adm}$, and $\Ann_{K\Stab(B)}(B)=\Aug(K_B)$ 

\item If $\lvert B \rvert = 2$ (resp. $2$ or $3$) and $μ^6\neq 1$ in $K$, then $B\not \in \mathcal{C}^{K}_{adm}$. 

\item If $\lvert B \rvert = 2$ (resp. $2$ or $3$) and $μ^6=1$ in $K$, then $B\in \mathcal{C}^K_{adm}$ and there is a homomorphism $χ:K_B\to \mathbb{U}_6 \cap K$ such that $\Ann_{K\Stab(B)}(B)=\Aug^χ(K_B)$. Furthermore, in that case, $\dim_K \Ann_{K\Stab(B)}(B)=\dim_K \Aug(K_B).$
\end{enumerate}
\end{proposition}

\begin{proof} The first case is a consequence of Lemma \ref{lem:sumupconditions} and the first part of the Third Computational Result. The discussion following the latter also implies the second case, and, in combination with the second part of the Third Computational Result, the third case, as well.
\end{proof}

\section{Freeness of $\Br(W)$}\label{ch:freeness}

In this section we study for which irreducible complex reflection groups the Brauer-Chen algebra is free over its ring of definition $\mathcal{A}$. For the cases where the algebra is indeed free, our method will, naturally, be to show that the basis obtained in Theorem \ref{th:basisoverK} for $\Br^K(W)$ carries over to the ring of definition.

\subsection{A useful lemma}

The following lemma will be critical for the freeness property of $\Br(W)$ in both the infinite series and the exceptional groups case so we mention it here. It is a substitute of Lemma \ref{lem:upperbound} that works over a ring. Recall the definition of the ideals $I_r=\sum_{B\in \mathcal{C}, \lvert B \rvert\geq r} e_B\mathcal{A}W$ (Definition \ref{def:ideals}) of $\Br(W)$. Also, let
$π_w$ denote the projection $\mathcal{A}W\to w\mathcal{A}\Stab(B)$, with respect to the decomposition $\mathcal{A}W=\oplus_{w\in W/\Stab(B)}w\mathcal{A}\Stab(B)$.

\begin{lemma}\label{lem:barcondition} Let $W$ be a complex reflection group and $R$ an $\mathcal{A}$-algebra. Let also $B$ be a transverse collection of cardinality $r$ and $x=\sum_{w\in W}λ_ww\in RW$ be such that $xe_B\in I^{(R)}_{r+1}$. If $\cup_{λ_w\neq 0}wB$ is a transverse collection, then $π_w(x)e_B\in I^{(R)}_{r+1}$ for all $w\in W$.
\end{lemma}
\begin{proof} Consider the quotient $\Br^R_{r}(W)=\Br^R(W)/I^{(R)}_{r+1}$. In $\Br^R_r(W)$ we have $\sum_{w\in W}λ_wwe_B=0$; we will show that $π_w(x)e_B=0$ in $\Br^R_r(W)$ for all $w\in W$. Pick some $w_0\in W$. If $π_{w_0}(x)=\sum_{w\in w_0\Stab(B)}λ_wwe_B \neq 0$ in $\Br^R(W)$, then we can assume that $λ_{w_0}\neq 0$ (otherwise we pick some other $w\in w_0\Stab(B)$ such that $λ_w\neq 0$). We write $\sum_{w\in W}λ_wwe_B=0$ as $\sum_{w\in w_0\Stab(B)}λ_wwe_B=-\sum_{w\not\in w_0\Stab(B)}λ_wwe_B$, or, equivalently, 

$$\sum_{w\in w_0\Stab(B)}λ_we_{w_0B}w=-\sum_{w\not\in w_0\Stab(B)}λ_we_{wB}w.$$ Multiplying both sides of the last equality by $e_{w_0B}$ on the left, yields
 
$$\sum_{w\in w_0\Stab(B)}δ^rλ_we_{w_0B}w=-\sum_{w\not\in w_0\Stab(B)}λ_we_{w_0B}e_{wB}w.$$ 
Now, for the right hand side, if $λ_w\neq 0$, then $w_0B\cup wB$ is a transverse collection by assumption. Furthermore, if $w\not\in w_0\Stab(B)$, then $w_0B\neq wB$ and so, $\lvert w_0B\cup wB \rvert >r$. Hence, $e_{w_0B}e_{wB}=0$ in $\Br^R_r(W)$. So, we have $\sum_{w\in w_0\Stab(B)}δ^rλ_we_{w_0B}w=0$ or, equivalently, since $δ$ is invertible in $R$, $\sum_{w\in w_0\Stab(B)}λ_wwe_{B}=0\Leftrightarrow π_{w_0}(x)e_B=0$, which implies the statement.
\end{proof}

\begin{definition}\label{def:barcondition}
Let $W$ be a complex reflection group and $R$ an $\mathcal{A}$-algebra. If an element $x=\sum_{w\in W}λ_ww \in RW$ satisfies $\cup_{λ_w\neq 0}wB\in \mathcal{C}$, we say that it satisfies the \textit{bar condition} for $B$.
\end{definition}

The following result is a straightforward corollary of Lemma \ref{lem:barcondition}.

\begin{corollary}\label{cor:lem1} Let $W$ be a complex reflection group, $R$ an $\mathcal{A}$-algebra and $B$ a transverse collection of cardinality $r$. If every element of $\Rel(B)$ satisfies the bar condition with respect to $B$, then $\overline\Rel(B)\cdot e_B\subseteq I^{(R)}_{r+1}$.
\end{corollary}

\subsection{Freeness in the infinite series}

We already have all necessary elements to prove the following result concerning the freeness of $\Br(W)$ for the irreducible complex reflection groups of the infinite series.

\begin{proposition}\label{prop:freenessGmpn} Let $W=G(m,p,n)$. The algebra $\Br(W)$ is a free $\mathcal{A}$-module with basis $\{we_B | B \in \mathcal{C}^K_{adm}, w\in W/K_B\}$, for any proper field $K$.
\end{proposition}
\begin{proof} Let $K$ be a proper field for $W$. Notice first that, the set of the statement is a basis for $\Br^K(W)$ by Theorem \ref{th:basisoverK}. Hence, it suffices to show that it spans $\Br(W)$ to obtain the result. Our method will be the same as in the proof of Theorem \ref{th:basisoverK}. We remind it here.

Recall the setting of Lemma \ref{lem:unionideals}, i.e. that $\Br(W)$ is the union of the descending chain of ideals $I_r=\sum_{B\in \mathcal{C}, \lvert B \rvert \geq r} \mathcal{A}We_B$ of $\Br(W)$ (Definition \ref{def:ideals}). We show that for all $r\in \mathbb{N}$, we have 
 
\begin{equation}\label{eq:indstepfreeness}
I_{r}=\sum_{B\in \mathcal{C}^K_{adm}, w\in W/K_B} \mathcal{A}we_B + I_{r+1}.
\end{equation}
Induction on $r$ then yields that the set $\{we_B\big| B\in \mathcal{C}^K_{adm}, w\in W/K_B\}$ spans $\Br(W)$, which implies the result.

Let $B$ be a transverse collection of cardinality $r$.

Suppose that $B\not \in \mathcal{C}^K_{adm}$. We first consider the case $(m,p)\neq (2,2)$. By Proposition \ref{prop:gmpngeneral}, $B$ is of cardinality at least two and it contains some $H_i$, in which case, Lemma \ref{lem:hiproblem} gives that, for some $s\in \mathcal{R}$ we have 
$μ_ss\in \Rel(B)$. Since $\Rel(B)e_B=0$ and $μ_ss$ is invertible in $\mathcal{A}$, this yields $e_B=0$. Hence, if $(m,p)\neq (2,2)$, for all $r\in \mathbb{N}$, we have, in particular:

\begin{equation}\label{eq:indstepfreeness1}
I_{r}=\sum_{B\in \mathcal{C}^K_{adm}} \mathcal{A}We_B + I_{r+1}.
\end{equation}

We show the same equality for $(m,p)=(2,2)$. By Proposition \ref{prop:g22n}, if $B\not \in \mathcal{C}^K_{adm}$, there are $i_1,j_1,i_2,j_2$ and $κ$ such that
$H_{i_1j_1}^0,H_{i_1j_1}^1,H_{i_2j_2}^κ\in B$ but $H_{i_1j_1}^{κ+1}\not\in B$.
Taking the element $σ^γ_{α,β}\in \Rel(B)$ for $α=H_{i_1j_1}, β=H_{i_2j_2}^κ$ and $γ=H_{i_1j_2}$, gives that $((j_1j_2)-(i_1i_2)_κ)e_B=0$. We verify that this element satisfies the bar condition (see Definition \ref{def:barcondition}), i.e. $(j_1j_2)B\cup (i_1i_2)_κB$ is a transverse collection. For this, notice that, if $i,j\not\in \{i_1,i_2,j_1,j_2\}$, then $H_{ij},H_{ij}^1$ are fixed by $(j_1j_2),(i_1i_2)_κ$. Hence, the only hyperplanes in $B$ that are not fixed by $(j_1j_2),(i_1i_2)_κ$ are $H_{i_1j_1},H_{i_1j_1}^1$ and $H_{i_2j_2}^κ$, and it suffices to check transversality of their respective images, i.e. that $(j_1j_2)B\cup (i_1i_2)_κB$ is a transverse collection. We have that 
$(j_1j_2)$ maps $H_{i_1j_1},H_{i_1j_1}^1,H_{i_2j_2}^κ$ to $H_{i_1j_2},H_{i_1j_2}^1,H_{i_2j_1}^κ$ respectively, and $(i_1i_2)_κ$ maps the same hyperplanes to 
$H_{i_2j_1}^κ,H_{i_2j_1}^{κ+1},H_{i_1j_2}^{2κ}$. One can now see that $(j_1j_2)B\cup (i_1i_2)_κB$ is a transverse collection, using, for example Remark \ref{remark:sumtransversality}. Hence,
$(j_1j_2)-(i_1i_2)_κ$ satisfies the bar condition. By Lemma \ref{lem:barcondition}, this implies that $π_w((j_1j_2)-(i_1i_2)_κ)e_B\in I_{r+1}$, for all $w\in W$. As one can see from the above calculations, $(j_1j_2)B\neq (i_1i_2)_κB$, hence $(j_1j_2),(i_1i_2)_κ$ belong to defferent cosets of $\Stab(B)$. So, for $w=(j_1j_2)$ for example, we have
$π_w((j_1j_2)-(i_1i_2)_κ)=(j_1j_2)$ which implies that $(j_1j_2)e_B\in I_r$, or, since $(j_1j_2)$ is invertible, that $e_B\in I_{r+1}$. This yields Equations \eqref{eq:indstepfreeness1} above, for the case $(m,p)=(2,2)$ as well.

Suppose now that $B\in \mathcal{C}^K_{adm}$. If $(m,p)\neq (2,2)$, then by Proposition \ref{prop:gmpngeneral}, $B=\{H_i\}$ for some $i$ or $B=\{H_{i_1j_1}^{k_1},\dots , H_{i_rj_r}^{k_r}\}$ with $i_1,j_1,\dots i_i,j_r$ pairwise distinct. In both cases, by Propositions \ref{prop:onehi} and \ref{prop:classgeneral}, respectively, we have that $(K_B-1)e_B=0$. If $(m,p)=(2,2)$ we obtain the same result by Propositions \ref{prop:g22n} and \ref{prop:classg22n}. This implies that if $w_1K_B=w_2K_B$, then 
$w_1e_B=w_2e_B$. In particular $w_1e_B - w_2e_B \in I_{r+1}$, which, together with Equations \eqref{eq:indstepfreeness1} above, yields Equations \eqref{eq:indstepfreeness}, and concludes the proof.
\end{proof}

\begin{remark} Note that  by Propositions \ref{prop:gmpngeneral} and \ref{prop:g22n}, for the groups in the infinite series admissibility of a transverse collection is, in fact, independend of the field $K$, and thus, so is the basis for $\Br(W)$ provided by the above proposition.
\end{remark}

\subsection{Freeness for the exceptional groups}\label{sec:freenessexceptionals}

We turn now to the study of the freeness property for $\Br(W)$ for the exceptional complex reflection groups. As for the corresponding study of admissibility in Section \ref{ch:admissibility}, we used computational methods to obtain the results that lead to the determination of the freeness property of $\Br(W)$ for the exceptional complex reflection groups. Again, we explain these results as well as the algorithm used to verify them. Since the case of the group $G_{26}$ presented several peculiarities compared to the other exceptional groups, we do not include the computational methods for its treatement, but rather give a self-contained exposition using the description from Example \ref{ex:g26} of Subsection \ref{sec:crgclass}.

%Again, as for the property of admissibility in the cases of the exceptional groups of rank 2, the question of freeness of $\Br(W)$ could be treated quickly in these cases yielding a positive answer (see Remark \ref{rem:g26}), since these groups do not have any pairs of transverse hyperplanes. 

\subsubsection{Study of freeness for $W\neq G_{26}$}

%\paragraph{\textbf{Preliminaries}}
Let $W$ be a complex reflection group and $B$ a transverse collection. 

\begin{definition} 
\label{def:accep}

A hyperplane $H$ is \textit{acceptable} (with respect to $B$), if it satisfies the following properties: 
\begin{itemize}
\item[(a1)] $H\not \in B$,
\item[(a2)] $H\not \pitchfork B$,
\item[(a3)] for every $H'\in B$ non-transverse with $H$, $\cup_{s\in \mathcal{R}_{H'\to H}}sB$ is a transverse collection in the orbit of $B$, i.e. all $s\in \mathcal{R}_{H'\to H}$ map $B$ to the same transverse collection.
\end{itemize}
A pair of acceptable hyperplanes $(H',H'')$ will be said \textit{acceptable} if $H'\neq H''$ and there is $B'\sim_{\mathcal{R}}B$ (that is, with $\mathcal{R}_{B\to B'}\neq \emptyset $) that contains $H',H''$.
\end{definition}

For $H,H',H''\in \mathcal{H}$ we define the element 
$$τ^{H',H''}_H:=\sum_{s\in \mathcal{R}_{H\to H'}} μ_ss-\sum_{s\in \mathcal{R}_{H\to H''}} μ_ss\in \mathcal{A}W.$$

\begin{lemma}\label{lem:reltau} Let $(H',H'')$ be an acceptable pair, and $H\in B$ be non-transverse with both $H'$ and $H''$. Then, $τ^{H',H''}_H e_B=0$.
\end{lemma}
\begin{proof} Since $(H',H'')$ is an acceptable pair, there is $B'\sim_{\mathcal{R}} B$ containing $H',H''$. Let $r\in \mathcal{R}_{B\to B'}$. Notice, first of all, that $H',H''$ are transverse, and let $H_1=r^{-1}H'\in B$ and $H_2=r^{-1}H''\in B$. This implies, in particular that the pairs $H_1, H'$ and $H_2,H''$ are non-transverse (Lemma \ref{lem:transverse1}). Since $H'$ is acceptable, by definition, for all reflections $s\in \mathcal{R}_{H_1\to H'}$ we have $sB=B'$. So, $e_{H''}e_{H'}e_B=e_{H''}(\sum_{s\in \mathcal{R}_{H_1\to H'}}μ_ss)e_B=δ\cdot (\sum_{s\in \mathcal{R}_{H_1\to H'}}μ_ss)e_B=δe_{H'}e_B$. In the same way, we get that $e_{H'}e_{H''}e_B=δe_{H''}e_B$. Now, since $H'\pitchfork H''$, $e_{H'},e_{H''}$ commute; thus, $δe_{H'}e_B=δe_{H''}e_B$ whence $(e_{H'}-e_{H''})e_B=0$. Since $H\in B$ is non-transverse with both $H',H''$, this last equality yields $(\sum_{s\in \mathcal{R}_{H\to H'}}μ_ss - \sum_{s\in \mathcal{R}_{H\to H''}}μ_ss)e_B=0$, or, equivalently, $τ^{H',H''}_He_B=0$.
\end{proof}

We state now the computational result of this section. Its verification will be discussed in the end of this subsection. Let $\Rel_τ(B)$ be the set of all $τ^{H',H''}_H$, where $(H',H'')$ is an acceptable pair with respect to $B$ and $H$ is non-transverse with both $H',H''$. For the rest of this subsection, the reader may want to recall the definition of $Θ_B, θ_B, D_B, D_B^0$ found in Subsection \ref{subsec:aconvdescrforrel}.

\begin{mdframed}
\textbf{Fourth Computational Result.}
Let $W$ be an exceptional group other than $G_{26}$ and $B$ a transverse collection. Then 
\begin{itemize}
\item[(F1)] $θ(s)\in \Rel(B)$ for some $s\in \mathcal{R}\cup \{1\}$, or
\item [(F2)]
\begin{itemize} 
\item[(a)] every $x\in \Rel(B)$ satisfies the bar condition, and
\item[(b)] if $B$ satisfies property (A2) of the First Computational Result, then $D_B\subseteq \mathbb{Z}(\overline\Rel(B)\cup \Rel_τ(B))$
\end{itemize}
\end{itemize}
\end{mdframed}

\begin{proposition}\label{prop:bigspanningset}
 Let $W$ be an exceptional complex reflection group other than $G_{26}$ and $R$ a proper ring for $W$, with $K=\Frac(R)$. The set $\{we_B \big| B\in \mathcal{C}^{K}_{adm}\cup \mathcal{C}_{cond}, w\in W/K_B\}$ spans $\Br^R(W)$.
\end{proposition}
\begin{proof} The proof follows our standard method for results of this kind. We show that for all $i\in \mathbb{N}$, we have

\begin{equation}\label{eq:indstepfreenessex}
I^{(R)}_r= \sum_{B\in \mathcal{C}^K_{adm}\cup \mathcal{C}_{cond}, w\in W/ K_B}Rwe_B+I^{(R)}_{r+1},
\end{equation}
and the result follows by induction on $r$. We will use the First Computational Result of the previous section as well.

Let $B \in \mathcal{C}\backslash( \mathcal{C}_{adm}^K\cup \mathcal{C}_{cond})$. By Lemma \ref{lem:sumupconditions}, $B$ satisfies condition (A1) of the First Computational Result, i.e., $\overline\Rel(B)$ contains $θ(s)$ for some $s\in \mathcal{R}\cup\{1\}$. If $B$ satisfies (F1) above, then there is $s'\in \mathcal{R}\cup \{1\}$ such that $θ(s')\in \Rel(B)$, which implies that $θ(s')e_B=0$ or, equivalently $e_B=0$, since $θ(s')$ is invertible. If $B$ satisfies (F2), then Corollary \ref{cor:lem1} implies that $\overline\Rel(B)\cdot e_B\subseteq I^{(R)}_{r+1}$. By (A1), $\overline\Rel(B)$ contains some $θ(s)$, and this implies that $e_B\in I^{(R)}_{r+1}$. This yields that for all $r\in \mathbb{N}$,

\begin{equation}\label{eq:indstepfreenessex1}
I^{(R)}_r= \sum_{B\in \mathcal{C}^K_{adm}\cup \mathcal{C}_{cond}}RWe_B+I^{(R)}_{r+1},
\end{equation}

Suppose now that $B\in \mathcal{C}^K_{adm}\cup\mathcal{C}_{cond}$. Again, by Lemma \ref{lem:sumupconditions} and the First Computational Result, $B$ satisfies condition (A2). Then $B$ does not satisfy (F1), which would imply (A1); hence, $B$ satisfies (F2), which implies that
$\overline\Rel(B) e_B\subseteq I^{(R)}_{r+1}$. Now, by (F2.b), $D_B\subseteq \mathbb{Z} (\overline\Rel(B)\cup\Rel_τ(B))$, and since, by Lemma \ref{lem:reltau}, $\Rel_τ(B)e_B=0$, then 
we have that $D_B\cdot e_B\subset I^{(R)}_{r+1}$, which is equivalent to $D_B^{0}\cdot e_B\subset I^{(R)}_{r+1}$ (see Corollary \ref{cor:passagetostab}). Recall that 

$$D_B^{0}=\{s_2^{-1}s_1-μ_{s_2}μ_{s_1}^{-1}\big| (s_1,s_2)\in P_B\}\cup (\mathcal{R}_B-1),$$
and, by (A2.b), the set $\{s_2^{-1}s_1\big| (s_1,s_2)\in P_B\}\cup \mathcal{R}_B$ generates $K_B$. This implies that for every $h\in K_B$, there is some invertible $l_h\in \mathcal{A}$ (in fact, some product of the parameters) such that $(h-l_h)e_B\in I^{(R)}_{r+1}$.
Now, if $w_1,w_2$ belong to the same left coset of $K_B$ in $W$, i.e. $w_2=w_1h$ for some $h\in K_B$, then $w_2e_B=w_1he_B=w_1l_he_B+w_1(h-l_h)e_B\in w_1e_B + I^{(R)}_{r+1}$. This, together with Equations \eqref{eq:indstepfreenessex1}, yields Equations \eqref{eq:indstepfreenessex}, and concludes the proof.
\end{proof}

\begin{corollary}\label{cor:freenessmostexceptionals} Let $W$ be an exceptional complex reflection group other than $G_{25},G_{26},G_{32}$, and let $K$ be any proper field. Then, $\Br(W)$ is a free $\mathcal{A}$-module, with basis 
$\{we_B\big| B\in \mathcal{C}^K_{adm}, w\in W/K_B\}$. 
\end{corollary}
\begin{proof} By the previous proposition, $\Br(W)$ is spanned by 
$\{we_B \big| B\in \mathcal{C}^K_{adm}\cup \mathcal{C}_{cond}, w\in W/K_B\}$. 
By the Second Computational Result (Section \ref{sec:admexceptionals}), if $W$ is different from $G_{25},G_{32}$, we have that $\mathcal{C}_{cond}=\emptyset $. Hence, the above spanning set is, in fact, $\{we_B\big| B\in \mathcal{C}^K_{adm}, w\in W/K_B\}$, which, by Theorem \ref{th:basisoverK}, is a basis for $\Br^K(W)$, and this implies that it is a basis for $\Br(W)$ as well.
\end{proof}

\begin{remark} As in the case of the infinite series, by Proposition \ref{prop:classpart1}, for an exceptional group other than $G_{25}, G_{32}$, admissibility of a transverse collection is independend of the field $K$, and so is the basis provided by the above corollary.
\end{remark}

\begin{proposition}\label{prop:freeness2532} Let $W$ be the group $G_{25}$ or $G_{32}$.
\begin{enumerate}

\item The algebra $\Br(W)$ is not a free $\mathcal{A}$-module. 
\item If $R$ is a proper ring where $μ^6=1$ (recall that $μ$ is the quotient of the parameters corresponding to the two orbits of reflections in $W$), then the algebra $\Br^R(W)$ is a free $R$-module, with basis $\{we_B\big| B\in \mathcal{C}^K_{adm}, w\in W/K_B\}$, where $K=\Frac(R)$.
\end{enumerate}

\end{proposition}

\begin{proof} We show the second statement first. If $μ^6=1$ in $R$, then the second part of the Third Computational Result says that $\mathcal{C}_{cond}\subseteq \mathcal{C}^K_{adm}$. So, the spanning set for $\Br^R(W)$ of Proposition \ref{prop:bigspanningset} above coincides with the set of the statement. By Theorem \ref{th:basisoverK}, the latter is a basis of $\Br^K(W)$, and, hence, it is a basis of $\Br^R(W)$ as well. 

Now, by Theorem \ref{th:basisoverK} again, for any proper field $K$, the set $\{we_B\big| B\in \mathcal{C}^K_{adm}, w\in W/K_B\}$ is a basis for $\Br^K(W)$. Let $K,K'$ be proper fields such that $μ^6\neq 1$ in $K$ but $μ^6=1$ in $K'$ (see Example \ref{ex:propering} and Remark \ref{rem:properfractionfield} to verify the existence of such proper fields). Then, by Proposition \ref{prop:case2532}, we have that $\mathcal{C}^K_{adm}\subset \mathcal{C}^{K'}_{adm}$, and, hence, $\dim_{K'}\Br^{K'}(W)> \dim_{K}\Br^K(W)$. This implies that the algebra is not free over $\mathcal{A}$.
\end{proof}

\paragraph{\bf How to verify the fourth computational result.}

We discuss here the main points of the algorithm we used to verify the computational result of this section. The main core of the algorithm used is established in Subsection \ref{sec:admexceptionals}. We remind here the necessary data.

Let $\verb+W+$ be a complex reflection group with list of reflections and distinguished reflections $\verb+R+$ and $\verb+R_dist+$, respectively. We have a table $\verb+F+$ with rows and columns corresponding to hyperplanes of $W$ such that $\verb+F[i,j]=true+$ if $H_i\pitchfork H_j$, or, otherwise, $\verb+F[i,j]+$ is a list of positions of reflections in $\verb+R+$ mapping $H_i$ to $H_j$. Let $\verb+B+$ stand for a list representing a transverse collection $B$ of $\verb+W+$. The list $\verb+R_B+$ contains the positions of reflections in $\verb+R+$ with reflecting hyperplanes in $B$. Also, the list $\verb+D_B+$ contains the vectors in $\mathbb{Z}^{N+1}$, where $N=\lvert \mathcal{R}\rvert$, representing $D_B$ with respect to the basis $Θ_B=\{μ_ss \big| s\in \mathcal{R}\backslash \mathcal{R}_B\}\cup \mathcal{R}_B\cup \{1\}$.

Similarly to what we did in the previous section for the set $\overline\Rel(B)$, representing it with the list $\verb+Rel_Bar+$, we first construct two lists, $\verb+Rel+$ and $\verb+Rel_Tau+$ of vectors in $\mathbb{Z}^N$, representing the sets $\Rel(B)$ and $\Rel_τ(B)$, respectively, with respect to the basis $Θ_B$. 

For the list $\verb+Rel+$, we add first, for all $i\in \verb+R_B+$, the vector $v(i,N+1)-v(N+1,N+1)$ to $\verb+Rel+$, including, in this way, all vectors representing the set $\mathcal{R}_B-1$. Next, for all $k=1,\dots , N+1$ and $i,j\in \verb+B+$ such that $H_i,H_j\not\pitchfork H_k$ (we can test this using the table $\verb+F+$) we add the vector 
$\sum_{l\in \verb+F(i,k)+} v(l,N+1)-\sum_{l\in \verb+F(j,k)+}v(j,N+1)$ to $\verb+Rel+$; this vector represents the element $σ^{H_k}_{H_i,H_j}$. In this way we include all vectors representing the elements of $Σ_B$, which gives the desired list. 

Now, using the list $\verb+Rel+$, verifying condition 
(F1) amounts to checking whether there is $i=1\dots , N+1$ such that
$v(i,N+1)\in \verb+Rel+$. For condition (F2.a), verifying the bar condition for an element of 
$\mathbb{Z}Θ_B$, represented by a vector in $\mathbb{Z}^{N+1}$, is straightforward using the conjugation of distinguished reflections and the transversality information of table $\verb+F+$.

For the list $\verb+Rel_Tau+$, we first construct a list of the pairs $(i,j)$ corresponding to acceptable pairs. This is straightforward; all necessary relations for transversality and conjugation are contained in the table $\verb+F+$. For the acceptable pairs, we may use the list $\verb+Smallorbit+$ which contains the lists representing all $B'\sim_{\mathcal{R}}B$. Given this, for all $(i,j)$ corresponding to acceptable pairs, and $k\in \verb+B+$ such that 
$H_i,H_k\not \pitchfork H_k$ we add to $\verb+Rel_Tau+$ the vector 
$\sum_{l\in \verb+F(k,i)+}v(l,N+1)- \sum_{l\in \verb+F(k,j)+}v(l,N+1)$, which represents the element $τ^{H_i,H_j}_{H_k}$. This completes the construction of this list. 

Now, let $\verb+Rel_Bar_Tau+$ be the concatenation of $\verb+Rel_Bar+$ and $\verb+Rel_Tau+$. Condition (F2.b) is equivalent to $\verb+D_B+\subseteq \mathbb{Z}\verb+Rel_Bar_Tau+$. We conclude with the algorithm that we used to check for linear dependence over $\mathbb{Z}$.

Suppose we have vectors $l_1,\dots , l_m, v \in \mathbb{Z}^n$. To test whether $v$ belongs to the $\mathbb{Z}$-span of $l_1,\dots , l_m$, we use standard existing algorithms to obtain the information concerning the Smith Normal Form of a matrix. Specifically, let $A$ be the $n\times m$ matrix with columns $l_1,\dots ,l_m$. Then, the $\mathbb{Z}$-span of $l_1,\dots ,l_m$ is equal to $\mathbb{Z}^m \cdot A$, i.e.  the image of $\mathbb{Z}^m$ under right multiplication by $A$. The aforementioned algorithms for the Smith Normal Form of $A$ provide us with invertible $m\times m$ and $n\times n$, respectively, integer matrices, $S$ and $T$, and $D=\diag_{m\times n}(a_1,\dots ,a_r)$, such that $SAT=D$. Thus, $v\in \mathbb{Z}^m\cdot A$ is equivalent to $v\in \mathbb{Z}^m\cdot S^{-1}DT^{-1}$, which is, in turn, equivalent to $vT \in \mathbb{Z}^m\cdot D$, since $S^{-1}$ induces an isomorphism $\mathbb{Z}^m\to \mathbb{Z}^m$. Now, we can see that, to test if $vT \in \mathbb{Z}^m \cdot D$, we need to check that $(vT)_i=0$ for $i> r$ and that $a_i|(vT)_i$ for $1\leq 1 \leq r$, where $(vT)_i$ is the $i$-th coordinate of $vT$. This concludes the algorithm.

\subsubsection{Freeness for $\Br(G_{26})$ }

We prove here that the algebra $\Br(G_{26})$ is free.
For this, we remind the description of the irreducible complex reflection group $G_{26}$ of Example \ref{ex:g26}.

Let $z_1,z_2,z_3$ denote the standard coordinates of $\mathbb{C}^3$ and $ζ=e^{\frac{2πi}{3}}$. Assume that $\mathbb{C}^3$ is equiped with the standard inner product. The group $G_{26}$ is the subgroup of $\GL(\mathbb{C}^3)$, generated by the following $3$ types of distinguished unitary reflections:

\begin{enumerate}
\item $t_i, i=1,2,3$, with reflecting hyperplanes $H_i$ with equation $z_i=0$, and order $3$,
\item $t_{κ,λ}$, $κ,λ = 0,1,2$,  with reflecting hyperplanes $T_{κ,λ}$ with equation $z_1 + ζ^κ z_2 + {ζ}^λ z_3=0$, and order $3$, and
\item $(ij)_{κ}$, for $1\leq i\neq j \leq 3$ and $κ=0,1,2$, with reflecting hyperplanes $H_{i,j}^{κ}: z_i=ζ^κz_j$, and order $2$. If $κ=0$, we may omit it from the notation.

\end{enumerate}

Hyperplanes of the form $H_i$ and $T_{κ,λ}$ make up one orbit of hyperplanes which we denote by $\mathcal{O}_1$, and the hyperplanes of the third type form a second orbit $\mathcal{O}_2$.
\vspace{5pt}

\paragraph{\textbf{Transversality in $G_{26}$}}

\begin{lemma}\label{lem:g263to1} Every hyperplane in $\mathcal{O}_1$ is transverse with exactly $3$ hyperplanes  of $G_{26}$, which all belong to $\mathcal{O}_2$. Furthermore, if $H,H'\in \mathcal{O}_1$ are trasverse with the same $3$ hyperplanes of $\mathcal{O}_2$, then $H=H'$.
\end{lemma}
\begin{proof}

For the first statement, we show it for $H_3\in \mathcal{O}_1$ and the general case follows by conjugation.

First, we check the intersection $H_3\cap H_1 = \langle(0,1,0) \rangle $. This is clearly a subspace of $H_{1,3}^λ$ for every $λ$ and hence $H_1,H_3$ are non-transverse. Also, note that $H_3\cap H_1 \subseteq H_{1,3}^λ$ implies that $H_3\cap H_{1,3}^λ \subseteq H_1$ (all intersections of two hyperplanes are of co-dimension $2$), and hence, $H_3$ is non-transverse with $H_{1,3}^λ$ for all $λ$ as well. Similarly, we obtain that $H_3\not \pitchfork H_2$ and $H_3\not \pitchfork H_{2,3}^λ$ for all $λ$. 

We check now the intersections $H_3\cap T_{κ,λ} =\langle (-ζ^κ,1,0) \rangle $ which can be seen to be a subspace of $T_{κ,λ'}$ for every $λ' \neq λ$, and so $H_3, T_{κ,λ}$ are non-transverse.

Finally, the hyperplanes left to check are the hyperplanes $H_{1,2}^λ, λ=0,1,2,$ which are transverse with $H_3$. Indeed, we have $H_{1,2}^λ \cap H_3 =\langle (ζ^λ,1,0) \rangle$ and one can verify that this is not contained in any hyperplane other that $H_3, H_{1,2}^λ$. So, $H_3$ is transverse only with $H_{1,2},H_{1,2}^1,H_{1,2}^2 \in \mathcal{O}_2$.

For the second statement, we verify that $H_3$ is the only hyperplane in $\mathcal{O}_1$ transverse with $H_{1,2}^λ,λ=0,1,2$, and the result follows again by conjugation. To quickly check that, observe that the reflections $t_1,t_2$ permute the hyperplanes $H_{1,2}^λ$ among themselves and hence, by Lemma \ref{lem:transverse2}, a hyperplane transverse with $H_{1,2}^λ,λ=0,1,2$ would stay invariant under $t_1,t_2$. Such a hyperplane is only $H_3$.

\end{proof}

\begin{lemma} The orbit $\mathcal{O}_2$ does not contain any pair of transverse hyperplanes.
\end{lemma}
\begin{proof}
We take again a representative $H_{1,2}\in \mathcal{O}_2$, and we check that it is not transverse with any hyperplane in $\mathcal{O}_2$. For $H_{1,2}^λ, λ=1,2$, we have $H_{1,2} \cap H_{1,2}^λ =\langle (0,0,1) \rangle \subseteq H_1$ and so $H_{1,2},H_{1,2}^λ$ are non-transverse. For $H_{1,3}^λ$ and any $λ$ we have $H_{1,2}\cap H_{1,3}^λ=\langle (ζ^λ,ζ^λ,1) \rangle \subseteq H_{2,3}^λ$ which yields that $H_{1,2} \not \pitchfork H_{1,3}^λ$ and $T_{1,2} \not \pitchfork T_{2,3}^λ$ as well, completing the result.
\end{proof}

The two above lemmas imply that a transverse collection $B$ contains at most one hyperplane from each orbit and, consequently, there are no transverse collections of cardinality more than two. As for collections of cardinality $2$, since every hyperplane in $\mathcal{O}_1$ is transverse with exactly $3$ hyperplanes in $\mathcal{O}_2$, there are $36$ of them, each containing one hyperplane of $\mathcal{O}_1$ and one of $\mathcal{O}_2$. In fact, they all form one orbit under $W$, denoted $\mathcal{B}$. For that, notice that it is sufficient to show that all collections containing $H_3$ are conjugate, since every collection is conjugate to one containing $H_3$. Again, to see that, check that the transverse collections  $\{H_3,H_{1,2}\},\{H_3,H_{1,2}^1\},\{H_3,H_{1,2}^2\}$ (see the proof of Lemma \ref{lem:g263to1} for a verification of their transversality) are conjugate, since $t_2\{H_3,H_{1,2}\}=\{H_3,H_{1,2}^1\}$ and $t_2^2\{H_3,H_{1,2}\}=\{H_3,H_{1,2}^2\}$.

\paragraph{\textbf{The elements $e_B$ for transverse collections of cardinality $2$}}

For a transverse collection $B$, let $\Ann(e_B)$ denote the annihilator of $e_B$ in $\Br(W)$.

\begin{lemma} Let $B,B'$ be two transverse collections of $G_{26}$ of cardinality $2$. If $B,B'$ contain the same hyperplane of $\mathcal{O}_1$, then $\Ann(e_B)=\Ann(e_{B'})$. 
\end{lemma}
\begin{proof}

We consider the following collections which serve as convenient representatives for what we demonstrate: $B_0=\{H_3,H_{1,2}^1\}, B_1=\{T_{0,0},H_{1,2}\},B_2=\{T_{0,0},H_{1,3}\},B_3=\{T_{0,0},H_{2,3}\}$. Here one may want to verify that $B_1,B_2,B_3$ are indeed transverse collections. For $B_1$ we have $T_{0,0} \cap H_{1,2}=\langle (1,1,-2) \rangle $ and one can check that this is not contained in any hyperplane other than $T_{0,0}$ and $H_{1,2}$. For $B_2,B_3$ the same can be obtained by conjugation from $B_1$ with reflections $(23),(13)$ which stabilize $T_{0,0}$ and map $H_{1,2}$ to $H_{1,3}$ and $H_{1,2}$ to $H_{2,3}$ respectively. 

We show now that $\Ann(e_{B_1})=\Ann(e_{B_2})=\Ann(e_{B_3})$. For that, we compute $e_{T_{0,0}}e_{B_0}$ in two ways. First, we find the reflections taking $H_3$ to $T_{0,0}$; let $s$ be such a reflection. By Lemma \ref{lem:transverse1}, the hyperplane $H_s$ contains $H_3\cap T_{0,0}=\langle (1,-1,0)\rangle $. The only hyperplanes different from $H_3, T_{0,0}$ containing this last subspace are $T_{0,λ}, λ=1,2$ and so, the possible reflections are $t_{0,λ}^p, λ,p\in \{1,2\}$. To calculate which of these reflections take $H_3$ to $T_{0,0}$ we use the following formula for a reflection, which uses a root of the reflection (a vector perpendicular to its reflecting hyperplane). If $r$ is a unitary reflection of $V$, and $u\in H_r\backslash \{0\}$ is such that $ru=αu$ where $α$ is some root of unity, then for every $v\in V, r(v)=v-(1-α)\frac{\langle v,u \rangle}{\langle u,u \rangle}u$.

A vector perpendicular to $T_{0,λ}$ is $u_λ=(1,1,ζ^λ)$ and since $t_{0,λ}$ is the distinguished reflection of $T_{0,λ}$ of order $3$, then $t_{0,λ}^p(v)=v-(1-ζ^p)\frac{\langle v,u_λ\rangle}{3}u_λ$. Now, to check whether $t_{0,λ}^p(H_3)=T_{0,0}$, it suffices to check that $t_{0,λ}^p(H_3^{\perp})=T_{0,0}^{\perp}$. For that, we use the vectors $e_3=(0,0,1)$ and $u_0=(1,1,1)$ that generate $H_3^{\perp}$ and $T_{0,0}^{\perp}$, respectively. Let $x=-(1-ζ^p)\frac{\langle e_3,u_λ \rangle}{3}=(ζ^p-1)ζ^{-λ}/3$. We have, $t_{0,λ}^p(e_3)=e_3+xu_λ=(0,0,1)+x(1,1,ζ^l)=(x,x,1+xζ^λ)$. So, $t_{0,λ}^p(H_3)=T_{0,0}$ if only if $(x,x,1+xζ^λ)\in \langle (1,1,1) \rangle \Leftrightarrow x=1+xζ^λ \Leftrightarrow (1-ζ^λ)x=1$. The last equation becomes $(1-ζ^λ)(ζ_p-1)ζ^{-λ}=3 \Leftrightarrow ζ^{p-λ}+ζ^{p} +ζ^{-λ}=2$. One can verify that this is true only for $p=λ\neq 0$, which means that there are two reflections taking $H_3$ to $T_{0,0}$, namely $t_{0,1}, t_{0,2}^2$. Let $μ_1,μ_2$ be the respective Brauer parameters of $t_{0,1}, t_{0,2}^2$. We have, 
$e_{T_{0,0}}e_{B_0}=\sum_{s\in \mathcal{R}_{H_3\to T_{0,0}}} μ_sse_{B_0}=(μ_1t_{0,1}+μ_2t_{0,2}^2)e_{B_0}$. Also, $T_{0,0}$ is non-transverse with $H_{1,2}^1\in B_0$ and so $e_{T_{0,0}}e_{B_0}=\sum_{s\in \mathcal{R}_{H_{1,2}^1\to T_{0,0}}} μ_sse_{B_0}=0$, since $H_{1,2}^1$ and $T_{0,0}$ belong to different orbits. Combining the two, we get $(μ_1t_{0,1}+μ_2t_{0,2}^2)e_{B_0}=0$.

We write the last equation as $e_{t_{0,1}B_0}μ_1t_{0,1}=e_{t_{0,2}^2B_0}(-μ_2t_{0,2}^2)$. Since $μ_1t_{0,1}$ and $-μ_2t_{0,2}^2$ are invertible elements of $\mathbb{Z}[\underline{μ}^{\pm 1}]W,$ this implies that: 

\begin{equation}\label{eq:annihilators}
\Ann(e_{t_{0,1}B_0})=\Ann(e_{t_{0,2}^2B_0}),
\end{equation}

We find the collections $t_{0,1}B_0, t_{0,2}^2B_0$. Since $t_{0,1}H_3=t_{0,2}^2H_3=T_{0,0}$, we have $t_{0,1}B_0, t_{0,2}^2B_0 \in \{B_1, B_2, B_3\}$. However, note that since $t_{0,1}, t_{0,2}^2$ map $H_3$ to $T_{0,0}$ and $H_{1,2}$ is transverse with both $H_3$ and $T_{0,0}$, then by Lemma \ref{lem:transverse2}, $H_{1,2}$ is invariant under $t_{0,1}, t_{0,2}^2$. Thus, since $H_{1,2}\not \in B_0$,  $t_{0,1}B_0, t_{0,2}^2B_0$ do not contain $H_{1,2}$ either, meaning that $t_{0,1}B_0, t_{0,2}^2B_0\in \{B_2,B_3\}$ or equivalently, $t_{0,1}H_{1,2}^1, t_{0,2}^2H_{1,2}^1\in \{H_{1,3},H_{2,3}\}$. To find hyperplanes $t_{0,1}H_{1,2}^1, t_{0,2}^2H_{1,2}^1$ we can either use the aforementioned formula for a unitary reflection or search in the following way. Since for any reflection $s$ mapping $H$ to $H'$, we have $H\cap H' \subseteq H_s$ (Lemma \ref{lem:transverse1}), which implies $H_s \cap H \subseteq H'$, in order to find the possible images of $t_{0,1}H_{1,2}^1$, we can first check which hyperplanes among $H_{1,2},H_{1,3},H_{2,3}$ contain $H_{1,2}^1\cap T_{0,1}=\langle (ζ,1,-1-ζ^{-1})\rangle= \langle (ζ,1,ζ)\rangle$. We see that it is only $H_{1,3}$ and so $t_{0,1}H_{1,2}^1=H_{1,3}$. In the same way, the intersection $H_{1,2}^1\cap T_{0,2}=\langle (ζ,1,-ζ^{-1}-ζ^{-2})\rangle= \langle (ζ,1,1)\rangle$ is contained only in $H_{2,3}$ and hence, $t_{0,2}^2H_{1,2}^1=H_{2,3}$. So, $t_{0,1}B_0=B_2$, $t_{0,2}^2B_0=B_3$ and Equation \eqref{eq:annihilators} gives $\Ann(e_{B_2})=\Ann(e_{B_3})$. Now, observe that for the reflection $(23)$, which corresponds to permutation of the coordinates $z_2,z_3$, we have $(23)B_2=B_1$ and $(23)B_3=B_3$ and hence, conjugating Equation \eqref{eq:annihilators} by $(23)$ we get $\Ann(e_{B_1})=\Ann(e_{B_3})$. Thus, we finally obtain $\Ann(e_{B_1})=\Ann(e_{B_2})=\Ann(e_{B_3})$. 

Now, since $B_1,B_2,B_3$ are all the collections of cardinality $2$ containing $T_{0,0}\in \mathcal{O}_1$ (Lemma \ref{lem:g263to1}), by conjugation, we obtain the result of the statement.
\end{proof}

\begin{proposition} Let $B$ be a transverse collection of $G_{26}$ of cardinality $2$. Then $e_B=0$.
\end{proposition}
\begin{proof}

We first consider a product $e_{b'}e_{\{b,a\}}$ where $b', b\in \mathcal{O}_1$ and $a\in \mathcal{O}_2$ is transverse with $b$. Of course, if $b'\in B$ then $e_{b'}e_{\{b,a\}}=δe_{\{b,a\}}$; otherwise, recall that every hyperplane in $\mathcal{O}_1$ is determined by the hyperplanes in $\mathcal{O}_2$ with which it is transverse (Lemma \ref{lem:g263to1}), and so if $b\neq b'$, then there exists $a'\in \mathcal{O}_2$ transverse with $b$ but not with $b'$. This gives that $e_{b'}e_{\{b,a'\}}=\sum_{s\in \mathcal{R}_{a'\to b'}}μ_sse_{\{b,a'\}}=0$ since $b'$ and $a'$ belong to different orbits. But $\Ann(e_{\{b,a\}})=\Ann(e_{\{b,a'\}})$ and so $e_{b'}e_{\{b,a\}}=0$. Summing up, for all $b,b'\in \mathcal{O}_1,$ and $a\in \mathcal{O}_2$ transverse with $b$, we have:
\begin{equation}\label{eq:prodrule} 
e_{b'}e_{\{b,a\}}=δ_{b',b}\cdot δe_B.
\end{equation}
Now, we compute $e_{H_{1,3}}e_{B_0}$, where $B_0=\{H_3,H_{1,2}^1\}$ in two ways, as before, to obtain an invertible element that annihilates $e_{B'}$ for some conjugate $B'$ of $B$. For a verification of the mentioned transversality relations one may go back to the proof of Lemma \ref{lem:g263to1}. First of all, $H_{1,3}\not \pitchfork H_3$, and since $H_{1,3},H_3$ belong to different orbits, $e_{H_{1,3}}e_{B_0}=0$. In addition to that, $H_{1,3}$ is non-transverse with $H_{1,2}^1$ and so $e_{H_{1,3}}e_{B_0}=\sum_{s\in \mathcal{R}_{H_{1,2}^1\to H_{1,3}}} μ_ss e_{B_0}$, which yields that $\sum_{s\in \mathcal{R}_{H_{1,2}^1\to H_{1,3}}} μ_ss e_{B_0}=0$. We rewrite this as 

\begin{equation}\label{eq:conclusion}
\sum_{s\in \mathcal{R}_{H_{1,2}^1\to H_{1,3}}} e_{sB_0}μ_ss =0.
\end{equation}

One can quickly see that $(32)_1$ is a reflection mapping $H_{1,2}^1$ to $H_{1,3}$, and $H_3$ to $H_2$. By Lemma \ref{lem:transversepairs}, we know that there is at most one such reflection, so for each $s$ mapping $H_{1,2}^1$ to $H_{1,3}$ and different from $(32)_1$, we have $sH_3\neq H_2$. This means that for such $s$, the hyperplane $H_2$ is not contained in $sB_0$ and so $e_{H_2}e_{sB_0}=0$ (Equation \eqref{eq:prodrule}). So, multiplying Equation (\ref{eq:conclusion}) by $e_{H_2}$ on the left, we obtain
$\sum_{s\in \mathcal{R}_{H_{1,2}^1\to H_{1,3}}}e_{H_2} e_{sB_0}μ_ss =0$, which yields $δe_{(32)_1B_0}μ_{(32)_1}(32)_1=0$, or, equivalently,
$e_{(32)_1B_0} =0$. 

Since, as we showed earlier, all transverse collections of cardinality $2$ belong to the same orbit, this implies that $e_B=0$.
\end{proof}

\paragraph{\bf Conclusion}
Finally, we can prove the following result.

\begin{proposition}\label{prop:freeness26} Let $W=G_{26}$. The algebra $\Br(W)$ is a free $\mathcal{A}$-module, and the set $\{we_H\big| H\in \mathcal{H}, w\in W/W_H\}$ is a basis.
\end{proposition}
\begin{proof} By Theorem \ref{th:basisoverK}, for any proper field $K$, the set $\{we_B \big| B\in \mathcal{C}^K_{adm}, w\in W/K_B\}$ is a basis for $\Br^K(G_{26})$. As we showed in the discussion preceding the proposition, in $G_{26}$ there are transverse collections of cardinality at most $2$ and $e_B=0$ for all transverse collections of cardinality $2$. Thus, the above basis yields, in fact, the basis $\{we_H \big| H\in \mathcal{H}, w\in W/K_{\{H\}}\}$. By Remark \ref{rem:KBforcardone}, for every $H\in \mathcal{H}$, the group $K_{\{H\}}$ is just the pointwise stabilizer $W_H$ of $H$. Hence, the set $\{we_H\big| H\in \mathcal{H}, w\in W/W_H\}$ is a basis of $\Br^K(W)$. We show that it also spans $\Br(W)$, which then implies the result.

By Lemma \ref{lem:genset}, the set $\{we_B \big| B\in \mathcal{C}, w\in W \}$ spans $\Br(W)$. Again, by the discussion before the proposition, there are no transverse collections of cardinality more than $2$, and $e_B=0$ if $B$ has cardinality $2$. Hence, the set $\{we_H\big| H\in \mathcal{H}, w\in W\}$ spans $\Br(W)$. Now, for every $H\in \mathcal{H}$ and $w_1,w_2$ in the same left coset of $W_H$, i.e. $w_2^{-1}w_1 \in W_H$, by (B3) of Definition \ref{def:algebra} of the Brauer-Chen algebra, we have that $w_2^{-1}w_1e_H=e_H$, or, equivalently, $w_1e_H=w_2e_H$. This implies that set $\{we_H\big| H\in \mathcal{H}, w\in W/W_H\}$ spans $\Br(W)$, which concludes the proof.
\end{proof}

%\begin{remark}\label{rem:g26} Notice that the above proof can be applied with almost no modification to the cases of all irreducible complex reflection groups of rank 2 since, these groups do not have any pairs of transverse hyperplanes.
%\end{remark}

\section{Appendix}

\subsection{Calculation of the order of $ K_B $ for the infinite series}

We calculate here the order of the group $K_B$ that completes the proofs of Propositions \ref{prop:classgeneral} and \ref{prop:classg22n}. Recall that, in the setting of $G(m,p,n)$, $ζ$ denotes the primitive $m$-th root of unity $e^{\frac{2πi}{n}}$ and $z_1,\dots ,z_n$ the standard coordinates of $\mathbb{C}^n$. We begin with a useful remark.

\begin{remark}\label{rem:semidirect}
A matrix of $G(m,p,n)$ can be written uniquely as the product of a permutation matrix and a diagonal matrix. In fact, if $D(m,p,n)$ denotes the set of all diagonal ones, then we have $G(m,p,n)=G(1,1,n)\ltimes D(m,p,n)$, where $G(1,1,n)$ coincides with the subgroup of permutation matrices. For a matrix $A\in G(m,p,n)$ let $P_A,D_A$ denote the unique permutation and diagonal, respectively, matrices for which $A=P_AD_A$.
\end{remark}

\begin{proposition}
\label{prop:Gmpngeneral}Let $W=G(m,p,n)$ and $B$ be transverse collection of the form $\{H_{i_1j_1}^{κ_1},\dots ,H_{i_rj_r}^{κ_r}\}$. Then $K_B$ consists of all matrices in $\Stab(B)$ whose product of non-zero entries is $1$ and which leave invariant the coordinates corresponding to the indices not appearing in any hyperplane in $B$. Furthermore, the order of $K_B$ is $ 2^rm^{r-1} r!$.

\end{proposition}\label{prop:orderKBgeneral}
\begin{proof} First, consider $B$ of the form $\{H_{i_1j_1},\dots ,H_{i_rj_r}\}$. In the proof of Proposition \ref{prop:classgeneral} we identified $K_B$ with the group $K_B^0$  generated by $\{(ij)\big| H_{ij}\in B\} \cup \{(ii')_κ(j'j)_κ\big| H_{ij'},H_{i'j}\in B,κ\in \mathbb{Z}\}$, which we write more concisely as
$$K_B^0=\langle (ij),(ii')_κ(jj')_κ\big| H_{ij},H_{i'j'}\in B, κ \in\mathbb{Z}\rangle.$$
Notice that the element $(ii')_κ (jj')_κ$ can be written as:

$$(ii')(jj')\cdot \diag(1,\dots,\underset{(i)}{ζ^κ}, 1, \dots ,\underset{(j)}{ζ^{κ}},1, \dots ,\underset{(i')}{ζ^{-κ}},1, \dots ,\underset{(j')}{ζ^{-κ}},1,\dots).$$
Denote this element by $h_κ(i,i',j',j)$, and by $m_κ(i,i',j',j)$ the above diagonal matrix so that
$h_κ(i,i',j,j')=(ii')(jj')m_κ(i,i',j,j')$. Then $h_0(i,i',j,j')=(ii')(jj')$, and hence, $h_κ(i,i',j,j')=h_0(i,i',j,j')m_κ(i,i',j,j')$. So, we have that

\begin{equation}\label{eq:firstdescKB}
K^0_B=\langle(ij),(ii')(jj'),m_κ(i,i',j,j')\big| H_{ij},H_{i'j'}\in B\rangle.
\end{equation}

Now, consider the subgroups 
$$K_B^{(1)}=\langle(ij),(ii')(jj')\big| H_{ij},H_{i'j'}\in B\rangle \hspace{10pt}\textnormal{ and }\hspace{10pt} K_B^{(2)}=\langle m_κ(i,i',j,j')\big| H_{ij},H_{i'j'}\in B\rangle,$$
which generate $K_B^0$. One can see that $K_B^{(1)}$ consists of all permutation matrices that stabilize $B$ and  which induce the identity permutation on the coordinates $z_i$ for the indices $i$ that do not appear in any hyperplane in $B$. The order of this group is, hence, equal to the number of permutations of the indices appearing in $B$ that leave $B$ stable, which is $2^rr!$.

For the group $K_B^{(2)}$, a quick verification yields that it consists of all diagonal matrcices $\diag(ζ^{λ_1},\dots ,ζ^{λ_n})$ for which $λ_i=λ_j$ for all $H_{ij}\in B, λ_i=0$ for all $i$ not appearing in $B$, and such that $\sum_{i=1}^n λ_i=0$. It follows that its order is $m^{r-1}$. 

Now, $K_B^{(1)}, K_B^{(2)}$ consist, respectively, of permutation and diagonal matrices and generate $K_B^0$. By Remark \ref{rem:semidirect} above, this implies that $K^0_B=K_B^{(1)}\ltimes K_B^{(2)}$, and hence, the order of $K_B^0$ is $2^rm^{r-1}r!$.

For the description of $K_B^0$ as a subgroup of $\Stab(B)$, one needs only verify that a matrix $A=P_AD_A\in G(m,p,n)$, with $D_A=\diag(ζ^{λ_1},\dots ,ζ^{λ_n})$, stabilizes $B$ if and only the permutation induced by $P_A$ on the indices appearing in $B$ stabilizes $B$ and $λ_i=λ_j$ for all $H_{ij}\in B$. Comparing this with the description of $K_B^{(1)}, K_B^{(2)}$ above, gives the characterization of the statement for $B$.

Finally, as in the proof of Proposition \ref{prop:classgeneral}, a general transverse collection of the form $\{H_{i_1j_1}^{κ_1},\dots ,H_{i_rj_r}^{κ_r}\}$ is the image of $B$ under the element $\prod_{i=1}^r\hat{t}_i^{κ_i}\in G(m,1,n)$. Furthermore, as explained in the same proof, $K_{gB}=gK_Bg^{-1}$ for all $g\in G(m,1,n)$. This yields the order of $K_B$ for all such transverse collections. 

Moreover, the product of the non-zero entries of a monomial matrix $A=P_AD_A\in G(m,p,n)$ is equal to its determinant times the sign of the permutation corresponding to $P_A$ matrix, and is, hence, invariant under conjugation by $G(m,1,n)$. So $A\in \Stab(B)$ has product of non-zero entries equal to $1$ if and only if the same is true for $gAg^{-1}\in g\Stab(B)g^{-1}=\Stab(gB)$. Similarly, $A$ leaves invariant the coordinates corresponding to the indices not appearing in $B$ if and only if the corresponding property is true for the matrix $gAg^{-1}$ and the coordinates corresponding to the indices not appearing in the transverse collection $gB$. This implies that the characterization of the statement is invariant under conjugation by elements of $G(m,1,n)$, and is, hence, true for a general collection of the form $\{H_{i_1j_1}^{κ_1},\dots ,H_{i_rj_r}^{κ_r}\}$. This concludes the proof.
\end{proof}

\begin{proposition}
\label{prop:G22ngeneral}Let $W=G(2,2,n)$ and $B=\{H_{i_1j_1}, H_{i_1j_1}^1,\dots ,H_{i_rj_r}, H_{i_rj_r}^1\}$ (with all indices distinct). Then $K_B$ consists of all matrices in $\Stab(B)$ that are diagonal on the lines and rows corresponding to the indices not appearing in any hyperplane in $B$. Furthermore, its order is $2^{r+n-1}r!.$

\end{proposition}
\begin{proof} 

In the proof of Proposition \ref{prop:classg22n} we identified $K_B$ with the following subgroup of $G(m,p,n)$:

$$K_B^0=\langle (ij_1)_κ, (i_1i)_κ(i_1i)_{κ+1},(i_1i)_κ(j_1j)_κ, (j_1j_2)_κ(j_1j_2)_{κ+1}\rangle,$$
with $H_{ij_1}, H_{i_1j}\in B, κ\in \mathbb{Z}$, and  $1\leq j_2 \leq n$. Notice that elements of the second form in the above presentation are obtained by all elements of the fourth form. This, together with a change of variables for a clearer exposition, yields the following shorter presentation of $K_B^0$:

$$\langle (ij)_κ, (ii_1)_κ(jj_1)_{κ}, (j_1j_2)_κ(j_1j_2)_{κ+1}\big| H_{ij},H_{i_1j_1}\in B, κ\in \mathbb{Z}, 1\leq j_2 \leq n\rangle$$

As in the proof of Proposition \ref{prop:orderKBgeneral} above, we have

$$(ii_1)_κ(jj_1)_κ=(ii_1)(jj_1)\cdot \diag(1,\dots,\underset{(i)}{ζ^κ}, 1, \dots ,\underset{(j)}{ζ^{κ}},1, \dots ,\underset{(i_1)}{ζ^{-κ}},1, \dots ,\underset{(j_1)}{ζ^{-κ}},1,\dots),$$
and we set $h_κ(i,i_1,j_1,j)=(ii_1)_κ(jj_1)_κ$ and $m_κ(i,i_1,j,j_1)$ the above diagonal matrix, so that $h_κ(i,i_1,j_1,j)$\\$=(ii_1)(jj_1)\cdot m_κ(i,i_1,j,j_1)$. Note, again that $(ii_1)(jj_1)=h_0(i,i_1,j_1,j)$. One can also verify the following two identities,

$$(ij)_κ=(ij)\cdot \diag(1,\dots,\underset{(i)}{ζ^κ}, 1, \dots ,\underset{(j)}{ζ^{κ}},1, \dots ,1),$$
and 
$$(j_1j_2)_κ(j_1j_2)_{κ+1}=\diag(1,\dots,\underset{(j_1)}{ζ}, 1, \dots ,\underset{(j_2)}{ζ^{-1}},1, \dots ,1).$$ 
Let $l_κ(i,j)=\diag(1,\dots,\underset{(i)}{ζ^κ}, 1, \dots ,\underset{(j)}{ζ^{κ}},1, \dots)$ and $q_κ(j_1,j_2)=(j_1j_2)_κ(j_1j_2)_{κ+1}$, and consider the following subgroups of $K_B^0$: 

$$K_B^{(1)}=\langle (ij),(ii_1)(jj_1)\big| H_{ij},H_{i_1j_1}\in B \rangle$$
and
$$K_B^{(2)}=\langle l_κ(i,j), q_κ(j_1,j_2), m_κ(i,i_1,j,j_1) \big| H_{ij},H_{i_1j_1}\in B, κ\in \mathbb{Z}, 1\leq j_2 \leq n  \rangle.$$
By the above, it is clear that the subgroups $K_B^{(1)}, K_B^{(2)}$ generate $K_B$. Furthermore, in view of Remark \ref{rem:semidirect}, since they consist, respectively, of permutation and diagonal matrices, then $K_B=K_B^{(1)}\ltimes K_B^{(2)}$.

Now, one can see that $K_B^{(1)}$ consists of all permutation matrices stabilizing $B$ and inducing the identity permutation on the coordinates $z_i$ for $i$ not appearing in any hyperplane in $B$. Thus, its order equals the number of permutations of the indices appearing in $B$ that stabilize $B$, which is $2^rr!$.

For $K_B^{(2)}$, one can, in fact, verify that it consists of all diagonal matrices $\diag(ζ^{λ_1}, \dots ,ζ^{λ_n})$ for which $\sum_{i=1}^nλ_i=0$, i.e. all diagonal matrices in $G(m,p,n)$; this group has order $2^{n-1}$. This yields the order of $K_B^0$, which is 
$2^{n+r-1}r!$.

For the characterization of $K_B^0$ as a subgroup of $\Stab(B)$, since all diagonal matrices stabilize $B$, we notice that a matrix $A=P_AD_A\in G(2,2,n)$ stabilizes $B$ if and only if the permutation matrix $P_A$ stabilizes $B$. By the above, the group $K_B^0$ contains, as $\Stab(B)$, all diagonal matrices of $G(m,p,n)$ together with all permutation matrices stabilizing $B$, with the extra condition that they induce the identity permutation on the coordinates $z_i$ for indices $i$ not appearing in any hyperplane in $B$. One can see that this yields the characterization of the statement. 
\end{proof}

\subsection{Table of admissibility and dimensions for the exceptional groups}

The following table contains the information of admissibility for all exceptional complex reflection groups. Each row corresponds to an orbit of transverse collections. The first column shows the cardinality of a representative of the orbit and the second the cardinality of the orbit. In the third column, we mention the size of the quotient $\Stab(B)/K_B$.  If the corresponding orbit consists of non-admissible collections, then we mark $0$ in this column. If for some cardinality there are no admissible collections of that cardinality, we do not include the corresponding row in the table at all. Finally, in the last column we mention the dimension of the Brauer-Chen algebra, as given by the formula:

$$\lvert W \rvert + \sum_{B} \lvert \mathcal{B} \rvert^2 \cdot \lvert S_B / K_B \rvert, $$ 
where $B$ runs through a set of representatives from each orbit of transverse collections (this formula is directly implied by the formula of Theorem \ref{th:basisoverK}). For the groups $G_{25}$ and  $G_{32}$, we include with an asterisque the corresponding data for proper fields where $μ\in \mathbb{U}_6$ (recall that $μ$ is the quotient of the parameters corresponding to the two orbits of reflections in these groups).

\begin{table}[h]
\noindent\makebox[\textwidth]{
\begin{tabular}{c c}
\begin{tabular}{|c | c | c | c |c|}
\hline 
$W$ & $\lvert B \rvert $ & $ \lvert  \mathcal{B} \rvert$ & $\lvert S_B/K_B \rvert$ & $\dim_K \Br(W)$ \\
\hline
$G_4$ & 1 & 4 &2 & 56 \\
\hline
\multirow{2}{*}{$G_{5}$} &1&4&6 &\multirow{2}{*}{264}\\&1&4&6&\\ 
\hline
\multirow{2}{*}{$G_{6}$} &1&6&4 &\multirow{2}{*}{256}\\&1&4&4&\\ 
\hline
\multirow{3}{*}{$G_{7}$} &1&6&12 &\multirow{3}{*}{960}\\&1&4&12&\\&1&4&12&\\ 
\hline
$G_8$ & 1 & 6 & 4 & 240 \\
\hline
\multirow{2}{*}{$G_{9}$} &1&12&8 &\multirow{2}{*}{1,632}\\&1&6&8&\\ 
\hline
\multirow{2}{*}{$G_{10}$} &1&8&12 &\multirow{2}{*}{1,488}\\&1&6&12&\\ 
\hline
\multirow{3}{*}{$G_{11}$} &1&12&24 &\multirow{3}{*}{6,432}\\&1&8&24&\\&1&6&24&\\ 
\hline
$G_{12}$ & 1 & 12 & 2 & 336 \\
\hline
\multirow{2}{*}{$G_{13}$} &1&6&8 &\multirow{2}{*}{960}\\&1&12&4&\\ 
\hline
\multirow{2}{*}{$G_{14}$} &1&12&6 &\multirow{2}{*}{1,392}\\&1&8&6&\\ 
\hline
\multirow{3}{*}{$G_{15}$} &1&12&12 &\multirow{3}{*}{3,648}\\&1&8&12&\\&1&6&24&\\ 
\hline
$G_{16}$ & 1 & 12 & 10 & 2,040 \\
\hline
\multirow{2}{*}{$G_{17}$} &1&30&20 &\multirow{2}{*}{22,080}\\&1&12&20&\\ 
\hline
\multirow{2}{*}{$G_{18}$} &1&20&30 &\multirow{2}{*}{18,120}\\&1&12&30&\\ 
\hline
\multirow{3}{*}{$G_{19}$} &1&30&60 &\multirow{3}{*}{90,240}\\&1&20&60&\\&1&12&60&\\ 
\hline
$G_{20}$ & 1 & 20 & 6 & 2,760 \\
\hline
\multirow{2}{*}{$G_{21}$} &1&30&12 &\multirow{2}{*}{16,320}\\&1&20&12&\\ 
\hline
$G_{22}$ & 1 & 30 & 4 & 3,840 \\
\hline
\multirow{2}{*}{$G_{23}$} &1&15&4 &\multirow{2}{*}{1,045}\\&3&5&1&\\ 
\hline
\end{tabular}

&

\begin{tabular}{|c | c | c | c |c|}
\hline 
$W$ & $\lvert B \rvert $ & $ \lvert  \mathcal{B} \rvert$ & $\lvert S_B/K_B \rvert$ & $\dim_K \Br(W)$ \\
\hline
$G_{24}$ & 1 & 21 & 8 & 3,864 \\
\hline
\multirow{3}{*}{$G_{25}$} &1&12&18 &\multirow{3}{*}{3,272/3,416*}\\&2&12&0/1*&\\&3&4&2&\\ 
\hline
\multirow{2}{*}{$G_{26}$} &1&9&72 &\multirow{2}{*}{12,312}\\&1&12&36&\\ 
\hline
$G_{27}$ & 1 & 45 & 24 & 50,760 \\
\hline
\multirow{2}{*}{$G_{28}$} &1&12&48 &\multirow{2}{*}{14,976}\\&1&12&48&\\ 
\hline
$G_{29}$ & 1 & 40 & 96 & 161,280 \\
\hline
\multirow{2}{*}{$G_{30}$} &1&60&120 &\multirow{2}{*}{452,025}\\&4&75&1&\\ 
\hline
$G_{31}$ & 1 & 60 & 384 & 1,428,480 \\
\hline
\multirow{4}{*}{$G_{32}$} &1&40&1,296 &\\&2&240&0/2*&2,232,320/\\&3&160&0/2*&2,398,720*\\&4&40&2& \\ 
\hline
\multirow{3}{*}{$G_{33}$} &1&45&576 &\multirow{3}{*}{1,364,769}\\&2&270&2&\\&5&26&1&\\ 
\hline
\multirow{3}{*}{$G_{34}$} &1&126&155,520 &\multirow{3}{*}{2,653,218,099}\\&2&2,835&18&\\&6&567&1&\\ 
\hline
\multirow{3}{*}{$G_{35}$} &1&36&720 &\multirow{3}{*}{1,440,585}\\&2&270&6&\\&4&135&1&\\ 
\hline
\multirow{7}{*}{$G_{36}$} &1&63&23,040 &\multirow{7}{*}{139,613,625}\\&2&945&48&\\&3&3,780&0&\\&3&315&6&\\&4&3,780&0&\\&4&945&2&\\&7&135&1&\\ 
\hline
\multirow{5}{*}{$G_{37}$} &1&120&2,903,040 &\multirow{5}{*}{53,328,069,225}\\&2&3,780&720&\\&4&113,400&0&\\&4&9450&6&\\&8&2025&1&\\
\hline

\end{tabular}

\end{tabular}}
\end{table}

\end{document}